\documentclass[reqno,10pt]{amsart}

\usepackage[sort]{cite}
\usepackage{amsfonts} 
\usepackage{amsmath} 
\usepackage{amssymb} 
\usepackage{amsthm} 
\usepackage{latexsym} 
\usepackage{mathrsfs} 
\usepackage{mathtools} 
\usepackage{slashed}  

\usepackage{verbatim}
\usepackage{cases}
\usepackage[dvips]{epsfig}
\usepackage{epsf} 
\usepackage[bookmarksnumbered,pdfpagelabels=true,plainpages=false,colorlinks=true,
            linkcolor=black,citecolor=black,urlcolor=black]{hyperref}
\usepackage{url}
\usepackage{color}
\usepackage{xcolor}
\usepackage{graphicx}
\usepackage{enumitem} 

\usepackage{pifont} 
\usepackage{upgreek} 
\usepackage{fancyhdr} 
\usepackage{calligra} 
\usepackage{marvosym} 
\usepackage[percent]{overpic} 
\usepackage{pict2e} 
\usepackage[hypcap=false]{caption} 
\usepackage{wasysym} 
\newcommand{\la}{\langle}
\newcommand{\ra}{\rangle}

\newtheorem{theorem}{Theorem}[section]

\newtheorem{lemma}[theorem]{Lemma}
\newtheorem{proposition}[theorem]{Proposition}
\newtheorem{corollary}[theorem]{Corollary}

\theoremstyle{definition}
\newtheorem{remark}[theorem]{Remark}

\renewcommand{\tilde}{\widetilde}
\renewcommand{\bar}{\overline}
\numberwithin{equation}{section}
\allowdisplaybreaks

\newcommand{\cA}{\mathcal{A}}

\newcommand{\cL}{\mathcal{L}}

\newcommand{\cQ}{\mathcal{Q}}
\newcommand{\cR}{\mathcal{R}}

\newcommand{\cV}{\mathcal{V}}

\newcommand{\cK}{\mathcal{K}}

\newcommand{\bE}{\mathbb{E}}

\newcommand{\bN}{\mathbb{N}}

\newcommand{\bR}{\mathbb{R}}

\newcommand{\bX}{\mathbb{X}}

\newcommand{\sA}{\mathsf{A}}

\newcommand{\sF}{\mathsf{F}}
\newcommand{\sM}{\mathsf{M}}

\newcommand{\sN}{\mathsf{N}}

\newcommand{\sG}{\mathsf{G}}

\newcommand{\sS}{\mathsf{S}}

\newcommand{\sd}{\mathsf{d}}

\newcommand{\pa}{\partial}

\newcommand{\zbE}{\prescript{}{0}{\mathbb{E}}}

\newcommand{\Lis}{\mathcal{L}{is}}
\newcommand{\PH}{{\mathbb{P}_H}}
\newcommand{\Ph}{{\mathbb{P}_{h}}}
\newcommand{\dv}{{\rm div}}
\newcommand{\Ric}{{\rm Ric\,}}
\newcommand{\gd}{{\rm grad}}
\newcommand{\wgd}{{\rm grad}_\Sigma}     
\newcommand{\wdv}{{\rm div}_\Sigma}        
\newcommand{\wRic}{{{\rm Ric}}}      
\newcommand{\wDel}{\Delta_\Sigma}          
\newcommand{\p}{p}   
\newcommand{\avint}{\mathop{\,\rlap{--}\!\!\int}\nolimits}

\begin{document}

\title{The primitive equations on curved surfaces}

\author[M. Hieber]{Matthias Hieber}
\address{Department of Mathematics\\
        TU Darmstadt\\
        Schlossgartenstr. 7, 64289 Darmstadt\\
        Germany}
\email{hieber@mathematik.tu-darmstadt.de}

\author[Y. Shao]{Yuanzhen Shao}
\address{The University of Alabama\\
	Tuscaloosa, Alabama \\
	USA}
\email{yshao8@ua.edu}
\author[G. Simonett]{Gieri Simonett}
\address{Department of Mathematics\\
        Vanderbilt University\\
        Nashville, Tennessee\\
        USA}
\email{gieri.simonett@vanderbilt.edu}

\author[M. Wilke]{Mathias Wilke}
\address{Institut f\"ur Mathematik\\
        Martin-Luther-Universit\"at Halle-Wittenberg\\
        Halle (Saale)\\
        Germany}
\email{mathias.wilke@mathematik.uni-halle.de}

\thanks{This work was supported by a grant from the Simons Foundation (\#426729 and \#853237, Gieri Simonett).
  Matthias Hieber acknowledges the support by the German Science Foundation (DFG) through the Research Unit FOR5528.
Yuanzhen Shao gratefully acknowledges the support from the National Science Foundation (NSF)   grants DMS-2306991 and DMS-2512104. }






\begin{abstract}

\bigskip\noindent
In this paper, we study the primitive equations in a collar neighborhood of a smooth closed hypersurface in $\bR^3$.
Using the framework of maximal $L_q$-regularity and the hydrostatic Helmholtz projection, we establish existence, uniqueness, and regularity results for strong solutions. Building on these local well-posedness results, we derive suitable a priori estimates and prove the global existence of strong solutions without imposing any smallness assumptions for the initial data. Moreover, we show that the solutions are $C^\infty$ jointly in time and space.
\end{abstract}

\maketitle

\section{Introduction}\label{S:Intro}
Geophysical fluid dynamics concerns the motion of fluids on planetary scales, most notably in the Earth's atmosphere and oceans. Unlike classical fluid motion encountered in laboratory settings, geophysical flows evolve under the combined influence of planetary rotation, gravity, stratification, and the geometry of the Earth itself, see e.g., \cite{Ped87}.

One of the defining features of  geophysical flow is the disparity between horizontal and vertical length scales. Atmospheric and oceanic motions typically extend thousands of kilometers horizontally while remaining comparatively shallow in the vertical direction. As a consequence, vertical velocities are generally much smaller than horizontal velocities, and the resulting dynamics exhibit strong directional asymmetry.

These physical considerations lead to the \emph{primitive equations,} which form the fundamental mathematical model for large-scale atmosphere and ocean dynamics. The primitive equations arise from the Navier--Stokes equations through asymptotic approximations appropriate for geophysical regimes, most importantly the \emph{hydrostatic approximation.} This approximation reflects the near balance between vertical pressure gradients and gravitational forces and is remarkably accurate for synoptic and planetary scales.

The primitive equations play a central role in both applied and theoretical research. From the physical perspective, they constitute the foundation of modern weather prediction and climate modeling. Virtually all large-scale atmospheric and oceanic simulation frameworks are based, either directly or indirectly, on variants of these equations. From the mathematical point of view, the primitive equations provide a distinguished class of nonlinear partial differential equations.

An important aspect of geophysical flow is that it takes place on a curved
surface. Since the Earth is approximately spherical, the geometry of the underlying domain cannot be neglected in large-scale models. Curvature influences both the structure of the equations and the qualitative behavior of solutions.

Consequently, the study of primitive equations on sphere-like surfaces is not merely a mathematical generalization but is intrinsically connected to the physical setting in which these flows occur.

The analysis of the primitive equations on curved geometries presents substantial mathematical challenges. The interplay between nonlinear transport, anisotropic structure, and geometric effects requires tools from differential geometry, functional analysis, and the theory of partial differential equations. Moreover, questions concerning existence, uniqueness, regularity, long-time behavior, and stability of solutions remain active areas of investigation.

The purpose of this paper is to study the primitive equations in the setting of a sphere-like surface.
However, we mention that our results hold for any closed smooth hypersurface $\Sigma\subset \bR^3$.

 Our focus is on the mathematical structure induced by the geometry and on the analytical properties of the resulting system. By incorporating the effects of curvature into the primitive equation framework, one obtains a model that more faithfully reflects the large-scale dynamics of geophysical flow while simultaneously raising challenging and interesting mathematical problems.

\medskip
In this paper, we  study  the following system of  primitive equations
\begin{equation}
\label{PE}
\left\{\begin{aligned}
  \partial_t v +   \nabla_v v     + w \partial_r  v -\big( \Delta_\sN      + {\rm Ric}     \big) v    + \gd_\Sigma  \pi_s    &=0    &&\text{in}&&\sN ,\\
\wdv \overline{v} & = 0 &&\text{in}&&\sN,\\
w(\cdot , r) - \int_r^0 \dv_\Sigma v (\cdot,\xi)\, d\xi & = 0 &&\text{in}&&\sN, \\
\partial_r v ,\ w&= 0 &&\text{in}&&\Sigma_u, \\
v ,\  w&= 0 &&\text{in}&&\Sigma_b, \\
v(0) & = v_0 &&\text{in}&&\sN ,
\end{aligned}\right.
\end{equation}
where $(\sN, g_{\sN})$ is the product manifold
$\sN=\Sigma \times [-h,0]$, with $(\Sigma, g)$ being a smooth, closed and connected hypersurface in $\bR^3$.  Moreover,
$g_\sN=g \oplus g_{\bR}$  is the product metric, with $g_{\bR}$ the Euclidean metric on $[-h,0]$.
We observe that the tangent space $T_{(\p,r)}\sN$ splits as
$$
T_{(\p ,r)} \sN = T_p \Sigma \oplus T_r \bR \cong T_p \Sigma \times \bR, \quad (\p, r)\in \Sigma\times [-h,0].
$$
In the equations above,
$v\in C(\sN; T\Sigma):= \{f\in C(\sN; T\sN): f(\p,r) \in  T_{\p}  \Sigma \} $ and $w\in C(\sN)$
stand for the horizontal and vertical components of the fluid velocity, respectively;
and $\pi_s\in C(\Sigma)$ is the surface pressure.
Moreover, $\bar v(\cdot)= \frac{1}{h} \int_{-h}^0 v(\cdot, r)\,dr$ is the vertical average of $v$.

\medskip\noindent
The quantities
$\wgd$, $\nabla=\nabla^\Sigma$ and $\wdv$ denote the gradient,  the Levi-Civita connection, and the divergence operator of $(\Sigma,g)$, respectively.

In addition, $\Delta_\sN =\Delta_\Sigma + \partial_r^2$ is the connection Laplacian on  $(\sN,g_{\sN})$ with $\Delta_\Sigma$ being its counterpart on $(\Sigma, g)$, and   $\Ric$ denotes the Ricci curvature tensor on $(\sN,g_{\sN})$.
We remark that $\Ric (\cdot , r)= {\rm Ric}_\Sigma (\cdot )$, where ${\rm Ric}_\Sigma$ is  the Ricci tensor on $(\Sigma, g)$.
 It is well-known that ${\rm Ric}_\Sigma$ coincides with  $\kappa g$, where $\kappa$ is the Gaussian curvature of $\Sigma$.
However, we prefer the notation ${\rm Ric}_\Sigma$   in view of the properties of the Ricci tensor used in Section~\ref{S: hydrostatic Stokes operator}.
Finally, the expressions $\Sigma_b:= \Sigma\times \{-h\} $ and $\Sigma_u:= \Sigma\times \{0\}$ denote the ``bottom and upper" boundary of $\sN$, respectively.

\smallskip
We note that the condition $\dv_\Sigma \bar v =0$ in~$\eqref{PE}_2$ is a direct consequence of $\eqref{PE}_3$ and the boundary condition $w(\cdot, 0)=0$.
More general boundary conditions for $v$ may also be treated using the techniques developed in this paper. For instance, one may consider
$$
v=0 \quad \text{on } \Sigma_D,\qquad \partial_r v=0 \quad \text{on } \Sigma_N,
$$
where $\Sigma_D \subset \{\Sigma_u,\Sigma_b\}$ and $\Sigma_N= \{\Sigma_u,\Sigma_b\}\setminus \Sigma_D.$
However, in order to streamline the presentation and avoid unnecessary technicalities, we restrict our attention to the boundary conditions in \eqref{PE}.

\smallskip
The set of equations~\eqref {PE} is derived and justified in Appendix~\ref{S: Appendix B},
where we consider a  fluid occupying a collar neighborhood $S$,
$$
S:=
\left\{x \in \mathbb{R}^3 :x = \p + r \nu_\Sigma(p),
\quad\p \in \Sigma,\quad -h < r < 0\right\},
$$
of a smooth and closed hypersurface $\Sigma$,
where $h>0$ is sufficiently small and $\nu_\Sigma$ denotes the outward-pointing unit normal vector field along $\Sigma$.

We assume that the motion of the fluid is governed by the incompressible Navier--Stokes equations. After introducing canonical normal coordinates associated with the collar neighborhood, we rewrite the Navier--Stokes equations with respect to the induced metric, which we refer to as the \emph{collar metric} of $S$.
More precisely, if $u:S \to \mathbb{R}^3$ denotes the fluid velocity field, we employ the decomposition
$$
u (\p,r) = v(\p,r) + w  (\p,r)\nu_\Sigma(\p,r),\quad   (v(\p,r), w  (\p,r)) \in T_{\p}\Sigma \times  T_r\bR =T_{\p}\Sigma \times \bR.
$$
Hence, $v$ represents the tangential and $w$ the normal velocity in $S$, respectively.
We then express the Navier-Stokes equations with respect to the collar metric, using the variables $(v,w)$.
In the course of this reformulation, we perform approximations that exploit the smallness of the parameter $h$.

Up to this stage, the analysis applies to an arbitrary closed  hypersurface $\Sigma$.
Finally, under the additional assumption that $\Sigma$ is $C^3$-close to a sphere of radius $a$,
we derive a further approximation  that takes into account the smallness of the ratio $\frac{h}{a}. $
The resulting equations for $(v,w)$ are then expressed with respect to the \emph{product metric} $g_\sN$.
This leads to the primitive equations stated in~\eqref{PE}.

\medskip
We note that, at least formally, setting $h=0$ reduces $\sN$ to $\Sigma$. In this case, the vertical velocity component $w$ disappears, and the resulting system \eqref{PE} coincides precisely with the \emph{surface Navier--Stokes equations} studied, for instance, in~\cite{PSW21, SaTu20, SSW25, SiWi22}. This observation highlights a close connection between the present framework and the theory of incompressible flows on curved surfaces. It would be interesting to investigate this relationship further and to explore to what extent the longtime behavior of solutions to the surface Navier-Stokes equations
carries over to the primitive equations considered here.

\medskip\noindent
We now state the main result of the manuscript.
\begin{theorem} [Global existence of smooth solutions]
\label{thm: main-result}
Suppose $p,q\in (1,\infty )$ and $\frac{1}{p}+ \frac{1}{q}\le 1$.  Then for each initial value
$v_0 \in B^{2/q}_{qp,\bar\sigma} (\sN; T\Sigma)$ with $v_0=0$ on $\Sigma_b$, the primitive equations~\eqref{PE}
admit a unique global smooth solution 
$$v\in C^\infty((0,\infty)\times \sN; T\Sigma),\quad \gd_\Sigma\pi_s\in C^\infty((0,\infty)\times \Sigma; T\Sigma).$$
\end{theorem}
Here, $B^{2/q}_{qp,\bar\sigma}(\sf N; T\Sigma)$
denote the hydrostatic solenoidal fields in  $B^{2/q}_{qp}(\sN; T\Sigma)$.
See Section~\ref{S: hydrostatic Helmholtz} for more details.

\begin{proof}
This follows from Proposition~\ref{prop:local existence} with $\mu=1/p+1/q$, Proposition~\ref{prop: interpolation}, 
Remark~\ref{rem: pressure}, and Theorem~\ref{thm: global}.
\end{proof}
\begin{corollary}
\label{cor: global H1}
For each initial value $v_0\in H^1_{2,\bar\sigma}(\sN; T\Sigma)$  with $v_0=0$ on $\Sigma_b$,
the primitive equations~\eqref{PE}
admit a unique global smooth solution 
$$v\in C^\infty((0,\infty)\times \sN; T\Sigma),\quad \gd_\Sigma\pi_s\in C^\infty((0,\infty)\times \Sigma; T\Sigma).$$
\end{corollary}
\begin{proof}
Taking $p=q=2$ and observing that $B^1_{22}=H^1_2$, the assertion is an immediate consequence of 
Theorem~\ref{thm: main-result}.
\end{proof}
\begin{remark}
\label{rem: global}
Theorem~\ref{thm: main-result} and Corollary~\ref{cor: global H1} establish 
the global existence of solutions without imposing any smallness assumptions on the initial data.
Choosing $q$  large in Theorem~\ref{thm: main-result}  shows that the primitive equations~\eqref{PE} admit global solutions for initial data of low regularity.
\end{remark}
The primitive equations were first introduced and studied  by Lions, Temam and Wang in a series of articles \cite{LTW92, LTW92b, LTW95}.
In these works, the authors established the existence of global weak solutions for initial data
$v_0 \in L_2$. The uniqueness of such solutions is still an open problem.

In the Euclidean setting, a breakthrough result has been established by Cao and Titi \cite{CaTi07},
who proved global well-posedness of the primitive equations in the strong sense for initial data $v_0\in H^1_2$
via $L_\infty((0,T); H^1_2)\cap L_2((0,T); H^2_2))$ a priori estimates.
We refer  also to the results by Kobelkov \cite{Kob07}, see also a result by Kukavica and Ziane \cite{KuZI07} for mixed Dirichlet-Neumann
boundary conditions.

These results have been extended to the setting of initial data in critical spaces of the form $B^{2/q}_{qp}$ and
scaling invariant spaces $L_\infty((0,T); L_1)$ in \cite{HiKa16, GGHHK20, GGHHK21}. The approach in the latter articles  is based on the introduction of the
hydrostatic Helmholtz decomposition and the hydrostatic Stokes operator. An extension of the hydrostatic Stokes operator to the space of
bounded functions and its consequences for the well-posedness of the primitive equations on these spaces was investigated in \cite{GGHHK20b}.

For various derivations  of the primitive equations by the
scaled Navier-Stokes equations we refer to Li and Titi \cite{LiTi19} and Furukawa, Giga, Hieber, Hussein, Kashiwabara and Wrona
{\color{blue}\cite{FGHHKW20, FGHHKW25}}.
For results in domains with uneven bottom, we refer to an article by  Drutsa \cite{Dru09} and for a survey article reviewing certain classes of
geophysical models including the primitive equations we refer to the work of Li and Titi \cite{LiTi18}.

Very recently, Korn \cite{Kor26} established  strong well-posedness on $[0,T]$ for arbitrary $T>0$ of the primitive equations
on $\mathbb{S}^2 \times [-h,0]$ subject to a
rigid lid assumption for arbitrary initial data in $H^1_2$, using a discrete exterior calculus on Delaunay-Voroni meshes and a limit procedure based on
compactness.

Our approach, described in the sequel, is fundamentally different. It is based on the hydrostatic Stokes operator (including the Ricci tensor) on manifolds and yields not only global solutions, but  \emph{global smooth solutions}.
Moreover, it applies to initial data in critical spaces of the form $B^{2/q}_{qp}$ and accommodates a much more general geometric setting.

A previous result for a general geometry is due to Drusta \cite{Dru11}, who proved the existence and uniqueness of \emph{generalized} solutions for
initial data in $H^2_2$ on $[0,T]$ for arbitrary $T>0$. Note that this solution is neither strong nor smooth.

\medskip\noindent
The manuscript is organized as follows. In Section~\ref{S: hydrostatic Helmholtz}, we introduce the \emph{hydrostatic Helmholtz projection}~$\Ph$, following the approach of \cite{HiKa16, GGHHK17}. We also establish several fundamental properties of~$\Ph$. In Section~\ref{S: hydrostatic Stokes operator}, we define the \emph{hydrostatic Stokes operator}~$\sA$. The main result of this section, Theorem~\ref{thm: H-infinity calculus}, shows that~$\sA$ admits a bounded $H^\infty$-calculus, a result that is also of independent interest. This property is then used in Proposition~\ref{prop:local existence}
to show that the primitive equations~\eqref{PE} admit unique strong solutions for initial data in critical spaces.
Moreover, the resulting solutions generate a smooth semiflow on the natural phase space.
It is further shown that these solutions regularize instantaneously and become smooth in both space and time for $t>0$.
In Section~\ref{S: global existence}, we establish a priori estimates for solutions of~\eqref{PE}; see Theorem~\ref{thm: a priori est}. These estimates constitute one of the core results of the paper and provide the key ingredient in the proof of global existence, which is established in Theorem~\ref{thm: global}.

In Appendix~\ref{S: Appendix A}, we collect results from differential geometry that will be used throughout the paper. Finally, as mentioned above, Appendix~\ref{S: Appendix B} provides a derivation of the primitive equations~\eqref{PE}, starting from the three-dimensional incompressible Navier--Stokes equations confined to a  thin collar neighborhood of~$\Sigma$.

\goodbreak
\bigskip\noindent
\textbf{Notation:}
We introduce notation that is used throughout the manuscript.\\
Given $f\in C(\sN; T\Sigma)$,
$$
\overline{f} (\cdot ) =\avint_{-h}^0 f(\cdot ,r)\, dr:= \frac{1}{h} \int_{-h}^0 f(\cdot ,r)\, dr
$$
is the vertical average of $f$. We then set  $\tilde f:= f- \bar f$.

The expressions $\gd_\sN$, $\nabla^\sN$ and $\dv_{\sN}$ denote the gradient, the Levi-Civita connection, and the divergence operator on $(\sN,g_{\sN})$.

The notation
$\la \, \cdot \, , \cdot \ra_g $ and $\la \,\cdot \,  , \cdot \ra_{g_\sN}$ stands for the Riemannian metric on $\Sigma$ and $\sN$, respectively,
We refer the reader to Appendix~\ref{S: Appendix A} , where we collect several well-known properties and results for Riemannian manifolds.

\smallskip
Given  $f\in L_2(\sN; T\Sigma)$  and $f\in L_2(\sN; T\sN)$, we write
$$
|f|_g = \sqrt{ \la f , f \ra_g}, \qquad |f|_{g_\sN} = \sqrt{ \la f , f \ra_{g_\sN}},
$$
respectively.
The expressions  $(\cdot , \cdot )_\Sigma$ and $(\cdot , \cdot )_\sN$ denote the duality pairing between $L_q(\Sigma; T\Sigma)$ and $L_{q^\prime}(\Sigma; T\Sigma)$,
and  between  $L_q(\sN; T\Sigma)$ and $L_{q^\prime}(\sN; T\Sigma)$, respectively, given by
\begin{align*}
(u , v)_\Sigma := \int_\Sigma \la u ,  v\ra _g \, d\mu_g ,  \qquad  (u , v)_\sN := \int_\sN \la u , v\ra _{g_\sN} \, d\mu_{g_\sN},
\end{align*}
where $d\mu_g$ and $d\mu_{g_\sN}$ denote the volume element on $(\Sigma,g)$ and $(\sN,g_{\sN})$, respectively, and
where $q^\prime$ is the dual exponent of $q\in (1,\infty)$.

\smallskip
Let $X$ and $Y$ be two Banach spaces and  let $T: X\to Y$ be a linear mapping. We denote by $D(T)$ and ${\rm Ker}(T)$  the domain and null space of $T$, respectively.
 $\cL(X,Y)$ stands for the set of all bounded linear operators from $X$ to $Y$,  $\cL(X):=\cL(X,X)$ and
$\Lis(X,Y)$ denotes the subset of $\cL(X,Y)$ consisting of linear isomorphisms from $X$ to $Y$.
We write $X \doteq Y$  if $X$ and $Y$ coincide up to equivalent norms.

\smallskip
For any $0\leq t_1<t_2<\infty$, $p\in (1,\infty)$ and $\mu\in (1/p,1]$, the $X$-valued $L_p$-spaces with temporal weight are defined by
$$
L_{p,\mu}((t_1,t_2);X):=\left\{ f: (t_1,t_2)\to X: \, t\mapsto t^{1-\mu}f(t)\in L_p((t_1,t_2);X)  \right\}.
$$
Similarly,
$$
H^k_{p,\mu}((t_1,t_2);X):=\{ f \in    H^k_{1,loc}((t_1,t_2);X):\,  \partial_t^j f\in L_{p,\mu}((t_1,t_2);X), \, j=0,1,\ldots,k \}.
$$
Throughout this manuscript, we assume that $p,q\in (1,\infty)$.
\goodbreak
\bigskip
\section{The hydrostatic Helmholtz projection}\label{S: hydrostatic Helmholtz}

Following the ideas in \cite{HiKa16}, we introduce the \emph{hydrostatic Helmholtz projection $\Ph$} for  $\sN$
and we collect some properties of $\Ph$ and of the space of \emph{hydrostatic solenoidal vector fields} on $\sN$.
Before doing so, we state a result concerning the existence of the classical Helmholtz projection on $L_q(\Sigma; T\Sigma)$.

Let $\Sigma$ be a smooth closed hypersurface in $\bR^3$, and
let $u\in L_q(\Sigma; T\Sigma)$ be given. Then it follows from \cite[Appendix B]{SSW25}
that the weak Neumann problem
$$
(\wgd \psi  , \wgd \phi)_{\Sigma}=(u, \wgd \phi)_{\Sigma} ,\quad  \phi\in \dot{H}^1_{q'}(\Sigma),
$$
has a unique  solution $\wgd \psi_u   \in L_q (\Sigma; T\Sigma)$, with $q'$ being the dual exponent of $q$.
In fact, the result in \cite[Appendix B]{SSW25} is more general, as it also applies to Riemannian manifolds with boundary.
The \emph{Helmholtz projection} $\PH$ on $\Sigma$ is defined by
\begin{equation}
\label{def-PH}
\PH u := u - \wgd \psi_u .
\end{equation}
For $s>0$, let
\begin{align*}
L_{q,\sigma}(\Sigma;T\Sigma)& := \PH(L_q(\Sigma; T\Sigma)), \\
H^s_{q,\sigma} (\Sigma;T\Sigma) &:=H^s_q(\Sigma;T\Sigma) \cap L_{q,\sigma}(\Sigma;T\Sigma)
\end{align*}
equipped with the norm inherited from $H^s_q(\Sigma; T\Sigma)$.

\begin{proposition}
\label{pro: Helmholtz-Sigma}
The Helmholtz projection $\PH$ on $\Sigma$ has the following properties.
\begin{itemize}
\item[{\rm (a)}] Let $s\in [0,2]$
and $q\in (1,\infty)$. Then $\PH\in \cL(H^s_q (\Sigma; T\Sigma) , H^s_{q, \sigma} (\Sigma ; T\Sigma))$ and $\PH$ is a projection onto $H^s_{q, \sigma} (\Sigma; T\Sigma)$. Moreover, for any $u\in L_{q} (\Sigma; T\Sigma)$ and $v\in L_{q'} (\Sigma; T\Sigma)$, it holds that
$$
(\PH u, v)_\Sigma = (u, \PH v)_\Sigma.
$$
\item[{\rm (b)}]
$
H^k_{q,\sigma} (\Sigma;T\Sigma)=\{v\in H^k_q(\Sigma;T\Sigma) :\, \wdv v=0\}
$
for $k=0,1,2$.
When $k=0$, the condition $\wdv  v=0$ reads as
$$
  (v , \wgd  \phi)_\Sigma = 0 , \quad \phi\in \dot{H}^1_{q'}(\Sigma).
$$
\item[{\rm (c)}]
 ${\rm Ker}(\PH)=\{\wgd  \phi: \phi\in \dot{H}^1_q(\Sigma) \}$.

\end{itemize}

\end{proposition}
\begin{proof}
 We refer to \cite[Appendix B]{SSW25} for the first part of (a). The proof
of the remaining properties follows along the same lines as in the Euclidean case.
%
\end{proof}

Following \cite{HiKa16}, see also~\cite{GGHHK17},  we define the \emph{hydrostatic Helmholtz projection $\Ph$}  by
\begin{equation}
\label{def-Ph}
\Ph v := \PH \bar v + \tilde v, \qquad  v \in L_q(\sN;T\Sigma).
\end{equation}
Note that $\bar v $ is independent of the $r$-variable and, thus, can be considered a vector field on $\Sigma$.
We now list some properties of $\Ph$ that are immediate consequences of the definition.
It follows that
\begin{equation}
\label{I-Ph}
(I-\Ph) v  = v -\PH \bar{v} - \tilde{v} = \bar{v} -  \PH \bar{v} = (I-\PH) \bar{v}  .
\end{equation}
It is clear that
$
  \int_{-h}^0 \tilde{v}(\cdot, \xi)\, d \xi \equiv 0.
$
Thus,  for any $v\in C^1(\sN;T\Sigma)$, it holds that
\begin{equation*}
\wdv (\overline{\Ph v } )= \wdv  (\overline{\PH \bar{v}}) = \wdv   (\PH \overline{v}) = 0
\end{equation*}
by the definition of $\PH$.
Moreover, if $\phi(\cdot,r)\equiv  \bar{\phi}(\cdot)$, then
\begin{equation}
\label{Ph-grad}
\Ph (\gd_\Sigma \phi)= \PH (\gd_\Sigma  \bar{\phi}) + \widetilde{\gd_\Sigma  \phi} = \gd_\Sigma  \phi - \gd_\Sigma  \bar{\phi}=0.
\end{equation}
It follows from H\"older's  inequality that
\begin{equation}
\label{bdd-avg}
\begin{split}
\|\bar{v}\|_{L_q(\sN)} & = \frac{1}{h}\left\| \int_{-h}^0 v(\cdot,\xi)\, d\xi \right\|_{L_q(\sN)}
 = \frac{1}{h} \left( \int_\sN \left| \int_{-h}^0 v(\cdot,\xi)\, d\xi \right|^q_g \, d\mu_{g_\sN} \right)^{1/q} \\
&= h^{1/q -1} \left( \int_\Sigma   \left| \int_{-h}^0 v(\cdot,\xi)\, d\xi \right|^q_g   \, d\mu_g \right)^{1/q} \\
& \leq  \left(\int_\Sigma  \int_{-h}^0   |v(\cdot,\xi)|_g^q   \, d\xi  \,\, d\mu_g \right)^{1/q}   \\
&  =  \| v\|_{L_q(\sN)} ,
\end{split}
\end{equation}
and for any $k\in \bN$,
\begin{equation}
\label{bdd-avg-gd}
\begin{split}
\|( \nabla^\sN)^k \bar{v}\|_{L_q(\sN)} = \| \nabla^k \bar{v}\|_{L_q(\sN)}  = \| \overline{\nabla^k v}\|_{L_q(\sN)}    \leq   \| \nabla^k v\|_{L_q(\sN)}.
\end{split}
\end{equation}
\eqref{bdd-avg} and \eqref{bdd-avg-gd} readily imply
\begin{equation}
\label{cont hydro-Helmholtz}
\| \Ph v\|_{H^k_q(\sN)}   \leq \|\PH \bar{v}\|_{H^k_q(\sN)}  + \| v - \bar{v}\|_{H^k_q(\sN)}   \leq C \|v\|_{H^k_q(\sN)},
\quad k\in \{0,1,2\} .
\end{equation}
in virtue of \cite[Lemma~B.6]{SSW25}.
For $s \in [0,2]$, we define the space of \emph{hydrostatic solenoidal vector fields} as
\begin{align*}
L_{q,\bar{\sigma}}(\sN;T\Sigma)& := \Ph(L_q(\sN; T\Sigma)), \\
H^s_{q,\bar{\sigma}} (\sN;T\Sigma) &:=H^s_q(\sN;T\Sigma) \cap L_{q,\bar{\sigma}}(\sN;T\Sigma)
\end{align*}
equipped with the norm inherited from $H^s_q(\sN; T\Sigma)$.
\smallskip
We are now ready to collect several additional properties of $\Ph$ that will be used in the following.
\begin{proposition}
\label{pro: properties-Ph}
The hydrostatic Helmholtz projection $\Ph$ enjoys the following properties:
\begin{itemize}
\item[{\rm (a)}] Let $s\in [0,2]$ and $q\in (1,\infty)$. Then $\Ph\in \cL(H^s_q (\sN; T\Sigma) , H^s_{q, \overline{\sigma}} (\sN ; T\Sigma))$ and $\Ph$ is a projection onto $H^s_{q, \overline{\sigma}} (\sN; T\Sigma)$. Moreover, for any $u\in L_{q} (\sN; T\Sigma)$ and $v\in L_{q'} (\sN; T\Sigma)$, it holds that
$$
(\Ph u, v)_\sN = (u, \Ph v)_\sN.
$$
\item[{\rm (b)}]
Given $\theta\in (0,1)$ and $p\in (1,\infty)$, let $(\cdot,\cdot)_\theta$ stand for either the complex interpolation functor $[\cdot,\cdot]_\theta$,
or the real interpolation functor $(\cdot,\cdot)_{\theta,p}$, respectively.
Then
$$
(L_{q,\bar{\sigma}}(\sN;T\Sigma), H^2_{q,\bar{\sigma}}(\sN;T\Sigma))_\theta  \doteq (L_q(\sN;T\Sigma), H^2_q(\sN;T\Sigma))_\theta \cap  L_{q,\bar{\sigma}}(\sN;T\Sigma).
$$
\item[{\rm (c)}]
$
H^k_{q,\bar{\sigma}} (\sN;T\Sigma)=\{v\in H^k_q(\sN;T\Sigma) : \wdv \bar{v}=0\}
$
for $k \in \{0,1,2\}$.
When $k=0$, the condition $\wdv \bar{v}=0$ reads as
$$
  (\bar{v} , \wgd  \phi)_\Sigma = 0 , \quad \phi\in \dot{H}^1_{q'}(\Sigma).
$$
\item[{\rm (d)}]
Suppose $v\in L_q(\sN; T\Sigma)$. Then there exists $\phi\in \dot{H}^1_q(\Sigma)$ such that $ (I-\Ph) v = \wgd  \phi$.
\vspace{1mm}
\item[{\rm (e)}]
 ${\rm Ker}(\Ph)=\{\wgd  \phi: \phi\in \dot{H}^1_q(\Sigma) \}$.

\end{itemize}
\end{proposition}


\begin{proof}
(a)
It follows from~\eqref {cont hydro-Helmholtz} that $\Ph$ is a continuous linear operator from  $H^k_q(\sN; T\Sigma)$ into $H^k_{q,\bar{\sigma}} (\sN;T\Sigma) $ for $k\in \{0,1,2\}$.
The continuity of the map $\Ph$ for the non-integer cases follows from  interpolation theory. See \cite[Theorem~13.1]{Ama13}.
The definition of $\Ph$ and the property ${\mathbb P}^2_H= \PH$ imply
\begin{align*}
{\mathbb P}^2_h v &= \Ph(\PH \bar{v} + \tilde{v} )= \PH\overline{(\PH \bar{v} +  \tilde{v} )} + \widetilde{ \PH \bar{v} +  \tilde{v} }\\
&= {\mathbb P}^2_H \bar{v} + \widetilde{ \PH \bar{v}} +  \tilde{v} = \PH \bar{v} +   \tilde{v} =\Ph v.
\end{align*}
Therefore, $\Ph$ is a projection.

Given any $u\in L_{q} (\sN; T\Sigma)$ and $v\in L_{q'} (\sN; T\Sigma)$ we have by \eqref{def-Ph} and the properties of $\PH$,
\begin{align*}
(\Ph u, v)_\sN & = (\PH \bar u, v)_\sN + (\tilde u, v)_\sN \\
    & =  (\PH \bar u, \bar v)_\sN + (u, \tilde v)_\sN \\
    & =  (\bar u, \PH \bar v)_\sN + (u, \tilde v)_\sN  \\
    & =  (u, \PH \bar v)_\sN + (u, \tilde v)_\sN  = (u, \Ph v)_{\sN},
 \end{align*}
where we have repeatedly used  the fact that $(\bar f, \tilde g)_{\sN}=0$ for $f\in L_q (\sN; T\Sigma)$ and $g\in L_{q^\prime}(\Sigma; T\Sigma)$.

\medskip\noindent
(b) This follows from part (a) and \cite[Theorem~1.17.1.1]{Tri78}.

\medskip\noindent
(c) We assume first that $k=0$. Let $v\in L_{q,\bar{\sigma}}(\sN; T\Sigma )$ be given. Then $v=\Ph v$ and this implies $\bar v = \PH \bar v$.
Hence, for each $\phi \in \dot{H}^1_{q^\prime}(\Sigma)$  it  follows from Proposition \ref{pro: Helmholtz-Sigma}(c)
\begin{align*}
(\bar v, \gd_\Sigma \phi)_\Sigma = (\PH \bar v, \gd_\Sigma \phi)_\Sigma = (\bar v, \PH \gd_\Sigma \phi)_\Sigma =0.
\end{align*}
Suppose now that
$$(\bar v, \gd_\Sigma \phi)_\Sigma =0,\qquad \forall \phi \in \dot{H}^1_{q^\prime}(\Sigma),$$
for some $v\in L_q(\sN; T\Sigma)$.
It follows from  the definition of $\PH$, see \eqref{def-PH}, that $\PH \bar v= \bar v$. Hence, $\Ph v= \PH \bar v + \tilde v =\bar v + \tilde v=v$,
and this implies  $v\in L_{q,\bar\sigma}(\sN; T\Sigma)$.

\smallskip\noindent
Now we consider the case $k>0$.
Assume that $v\in H^k_{q,\bar{\sigma}} (\sN;T\Sigma)$. Then $v=\Ph v$ and it follows as above that $\PH \bar v =\bar v$.
It is clear that $\overline{v}\in H^k_q (\sN;T\Sigma)$ and we readily conclude that
\begin{align*}
\wdv \bar{v}   = \wdv \PH \bar v =0.
\end{align*}
This proves the inclusion
$
 H^k_{q,\bar{\sigma}} (\sN;T\Sigma) \subset \{v\in H^k_q(\sN;T\Sigma) :\, \wdv  \bar{v}=0\}.
$
Conversely,  assume that $v\in H^k_q(\sN;T\Sigma) $ and $\wdv \bar{v}=0$. Then it follows from~\eqref{I-Ph} that
\begin{align*}
(I-\Ph) v & = (I- \PH) \bar{v} = 0,
\end{align*}
where the second equality follows from the fact that $\bar{v}\in H^k_q(\Sigma;T\Sigma)$ and the definition of $\PH$. This proves the converse inclusion.

\medskip\noindent
(d) It follows from \eqref{I-Ph} that  $(I-\Ph) v = (I-\PH) \overline{v}$.
We note on a side that $\bar v \in L_q(\Sigma, T\Sigma)$.
By the definition of $\PH$, see \eqref{def-PH},
 $(I-\PH)\bar v =\gd_\Sigma \psi_{\bar v}$
 for a (unique) function $\psi_{\bar v}\in \dot{H}^1_{q}(\Sigma)$.
 Taking $\phi = \psi_{\bar v}$ yields the assertion.

 \medskip\noindent
 (e) Suppose $v\in {\rm Ker}(\Ph)$. Then $v\in L_q(\sN; T\Sigma)$ and $(I-\Ph)v=v$. By the assertion in (d), $v=\gd_\Sigma\phi$ for some function $\phi \in \dot{H}^1_{q}(\Sigma).$
Suppose, on the contrary, that $v= \gd_\Sigma \phi$ for  some function $\phi\in \dot{H}^1_q(\Sigma)$.
Then it follows from~\eqref{Ph-grad} that $\Ph v=0$.

\noindent
This concludes the proof of Proposition~\ref{pro: properties-Ph}.
\end{proof}
\begin{remark}
It follows from Proposition~\ref{pro: properties-Ph}(e) that $\Ph$ annihilates the pressure term $\gd_\Sigma \pi_s$.
\label{rem: Ph annihilates pressure}
\end{remark}

\section{The hydrostatic Stokes operator}\label{S: hydrostatic Stokes operator}
In this section, we introduce and study the properties of the \emph{hydrostatic Stokes operator}
 $$\sA : D( \sA  )\to L_{q, \bar{\sigma}}(\sN;T\Sigma),$$
 defined by
\begin{equation}
\label{def: hydrostatic-Stokes}
\sA  v  = - \Ph ( \Delta_\sN  +	 \Ric  )v ,
\end{equation}
where $D( \sA  )=\{v  \in H^2_{q,\bar{\sigma}} (\sN;T\Sigma) : \, v=0 \text{ on }\Sigma_b, \ \ \partial_r v=0 \text{ on }\Sigma_u\}.$

\medskip\noindent
It turns out to be beneficial to first consider a related but simpler operator $\cA$, obtained by temporarily neglecting the hydrostatic projection $\Ph$
and adding some lower-order terms.
Hence, after this simplification, we are dealing with an elliptic operator without having to contend with the additional nonlocal effects introduced by $\Ph$.
The lower-order terms are introduced in order to ensure that $\cA$ commutes with $\Ph$  in the sense described in Proposition~\ref{pro: properties of lambda-A}.

\goodbreak

\medskip\noindent
Let us  then first consider the operator  $\cA: D( \cA)\to L_q(\sN; T\Sigma)$, defined by
\begin{equation}
\label{def: cal-A}
\cA v= - (\Delta_\sN  - \Ric    -B_H) v
\end{equation}
with $D(\cA)=\{v  \in H^2_q (\sN;T\Sigma) : \, v=0 \text{ on }\Sigma_b,  \ \  \partial_r v=0 \text{ on }\Sigma_u\}$,
where
$$B_H v = \frac{1}{h} (\PH-I) (\partial_r v |_{\Sigma_b}).$$
One readily verifies that  $B_H \in \cL(W^{1+1/q+\delta}_q(\sN;T\Sigma), L_q(\sN;T\Sigma))$ for any $\delta>0$.
We observe that by~\eqref{I-Ph}
$$
\sA = \Ph(\cA - 2 \Ric - B_H)= \Ph(\cA - 2 \Ric).
$$

We will show that $\cA$ admits a bounded $H^\infty$-calculus on $L_q(\sN; T\Sigma)$. For this we will exploit the fact that   $\cA$ can be decomposed as
$\cA=\cA_0 + \cA_1 +\cQ$, where $\cA_0 = -\Delta_\Sigma$ is the connection Laplacian on $\Sigma$,
$\cA_1= -\partial^2_r$, and $\cQ=  \Ric  +B_H$ is a lower order perturbation.
Fortunately, we can then use results in~\cite{SSW25} to show that $\cA_0$ admits an $\cR$-bounded $H^\infty$-calculus,
while we prove an analogous property for $\cA_1$ by a careful analysis of the associated resolvent problem.
This implies, in turn, that $\cA_0 +\cA_1$ admits a bounded $H^\infty$-calculus as well, and
the assertion for $\cA$  then follows from a perturbation argument.

\begin{proposition}
\label{prop: H calculus of A}
Let $1<q<\infty$.
There exists some $\lambda_0 \geq 0$ such that for all $\lambda > \lambda_0$
the operator $\lambda +\cA$ admits a bounded $H^\infty$-calculus  in $L_q(\sN;T\Sigma)$ with $H^\infty$-angle $<\pi/2$.
\end{proposition}
\begin{proof}
We define the operator $ \cA_0:=- \Delta_\Sigma$ in $L_q(\Sigma;T\Sigma)$ with domain $D(\cA_0)=H_q^2(\Sigma;T\Sigma)$.
It follows from \cite[Proposition 3.2]{SSW25} that
 there exists $\lambda_0\ge 0$ such that for $\lambda>\lambda_0$,
$\lambda + \cA_0$ has a bounded $H^\infty$-calculus with $H^\infty$-angle $<\pi/2$ (note that in our case $\partial\Sigma=\emptyset$, hence $\mathcal{K}_1=\emptyset$ in \cite[Proof of Proposition 3.2]{SSW25}).
This property remains true for the canonical extension of $\cA_0$ to the base space $X:=L_q(\Sigma\times (-h,0);T\Sigma)$ with domain
$$
D(\cA_0)=L_q((-h,0); H^2_q(\Sigma; T\Sigma))=H_q^2(\Sigma;L_q((-h,0);T\Sigma)).
$$
Next, we define  $cA_1:= - \partial_r^2$ in $L_q(\Sigma\times (-h,0);T\Sigma)$ with domain
\begin{align*}
D(\cA_1)=&\Big\{u\in L_q(\Sigma;H_q^2((-h,0);T\Sigma))= H_q^2((-h,0); L_q(\Sigma; T\Sigma))   : \\
&\quad  u=0 \text{ on } \Sigma_b, \, \, \partial_r u=0 \text{ on }\Sigma_u \Big\}.
\end{align*}
Then $\cA_1$ has a bounded $H^\infty$-calculus with $H^\infty$-angle 0, as we will show below.

{Next, we note that for each $q\in (1,\infty)$, the space $X=L_q(\Sigma\times (-h,0);T\Sigma)$ has property $(\alpha)$, see e.g. \cite[Definition 7.5.1]{HvNVW17} for a definition. This can be proven by literally following the lines of the proof of \cite[Proposition 7.5.3]{HvNVW17}, taking into account that for each Hilbert space $H$, the $\alpha$-bound $\alpha_H$ equals 1, see \cite[Example 7.5.2]{HvNVW17}.}

Based on this fact, it then follows from \cite[Theorem 4.5.6 (iii)]{PrSi16} that the $H^\infty$-calculi of $\cA_0,\cA_1$ are even $\mathcal{R}$-bounded in $X$ with the same angles.
Consequently, \cite[Corollary 4.5.11]{PrSi16} implies that the sum $\cA_0+\cA_1$ with natural domain
\begin{align*}
D(\cA_0+\cA_1)&=D(\cA_0)\cap D(\cA_1)\\
&=H_q^2(\Sigma;L_q((-h,0);T\Sigma))\cap \{u\in H_q^2((-h,0); L_q(\Sigma; T\Sigma)): \\
& \qquad \qquad  u=0 \text{ on } \Sigma_b, \, \, \partial_r u=0 \text{ on }\Sigma_u\}
\end{align*}
admits a bounded $H^\infty$-calculus in $X$ with $H^\infty$-angle $<\pi/2$.
It then follows from \cite[Theorem~VII.4.6.2]{AmannBook19} that
$$D(\cA_0+\cA_1)=\left\{u\in H_q^2(\Sigma\times (-h,0);T\Sigma): \,\, u=0 \text{ on } \Sigma_b, \, \, \partial_r u=0 \text{ on }\Sigma_u   \right\}.$$
For the proof of the $H^\infty$-calculus of $\cA_1=- \partial_r^2$ in $L_q(\Sigma\times (-h,0);T\Sigma)$, we observe that, by means of local coordinates with respect to $\Sigma$, it suffices to prove that $-\partial_r^2$ admits a bounded $H^\infty$-calculus in $L_q((-h,0);\mathbb{R}^n)$. Indeed, the operator $\cA_1$ does not depend on $p\in\Sigma$ and therefore, $p\in\Sigma$ serves at most as a parameter.

In what follows, we use the transformation $[-h,0]\ni r\to z:=r+h\in [0,h]$. We first show that for any given $\lambda\in\Sigma_\pi$ and $f\in L_q((0,h);\mathbb{R}^n)$, the boundary value problem
\begin{equation}
\label{eq:ODE1}
\left\{\begin{aligned}
\lambda u - u''  &=f &&\text{in}&&[0,h] ,\\
u(0) & =0   , \\
u'(h) &=0   ,
\end{aligned}\right.
\end{equation}
has a unique solution $u\in H_q^2((0,h);\mathbb{R}^n)$. To see this, we define $\tilde{f}(z)=f(z)$ if $z\in [0,h]$ and $\tilde{f}(z)=0$ otherwise. Then we solve the problem
\begin{equation*}
\left\{\begin{aligned}
\lambda v - v''  &=\tilde{f} &&\text{in}&&\mathbb{R}_+,\\
v(0) & =0   ,
\end{aligned}\right.
\end{equation*}
to obtain a unique solution $v\in H_q^2(\mathbb{R}_+;\mathbb{R}^n)$ by \cite[Chapter 6]{PrSi16}, since $\tilde{f}\in L_q(\mathbb{R}_+;\mathbb{R}^n)$. Next, we consider the inhomogeneous boundary value problem
\begin{equation*}
\left\{\begin{aligned}
\lambda w - w''  &=0 &&\text{in}&&[0,h] ,\\
w(0) & =0   , \\
w'(h) &=g   ,
\end{aligned}\right.
\end{equation*}
for given $g\in \mathbb{R}^n$. This $\mathbb{R}^n$-valued ODE has the unique solution
$$w(z)=c_1e^{-\sqrt{\lambda}z}+c_2e^{\sqrt{\lambda}z}, $$
where $c_k\in \mathbb{R}^n$ are uniquely determined by the boundary conditions as long as $\lambda\in\Sigma_\pi$ (so that $\sqrt{\lambda}\in\Sigma_{\pi/2}$). Therefore, with $g=v'(h)$, the unique solution $u$ of \eqref{eq:ODE1} is given by
$$u=v|_{[0,h]}-w$$
and obviously, it holds that $u\in H_q^2((0,h);\mathbb{R}^n)$. In particular, we have proven that each $\lambda\in\Sigma_\pi$ belongs to the resolvent set of the operator $-\partial_z^2$ in $L_q((0,h);\mathbb{R}^n)$.

For what follows, let $U_0:=[0,5h/8)$ and $U_1:=(3h/8,h]$. Then $[0,h]=U_0\cup U_1$ and there exists a partition of unity $\{\xi_0,\xi_1\}$
subordinate to $U_0\cup U_1$
such that ${\rm supp}\ \xi_j\subset U_j$ and $\xi_0^2+\xi_1^2=1$. Let us define a smooth diffeomorphism $\varphi:(-\infty,h]\to \mathbb{R}_+$ by $\varphi(x)=-(x-h)$. Furthermore, for any $s\ge 0$, we set
$$\mathfrak{R}^cu:=(\xi_0u,\varphi_*(\xi_1u)),\quad u\in H_q^s((0,h);\mathbb{R}^n)$$
and
$$\mathfrak{R}(v_0,v_1):=\xi_0v_0+\xi_1\varphi^*v_1,\quad v_0,v_1\in H_q^s(\mathbb{R}_+;\mathbb{R}^n).$$
It is the straightforward to check that
$$\mathfrak{R}^c\in\mathcal{L}(H_q^s((0,h);\mathbb{R}^n),H_q^s(\mathbb{R}_+;\mathbb{R}^n)\times H_q^s(\mathbb{R}_+;\mathbb{R}^n))$$
and
$$\mathfrak{R}\in\mathcal{L}(H_q^s(\mathbb{R}_+;\mathbb{R}^n)\times H_q^s(\mathbb{R}_+;\mathbb{R}^n),H_q^s((0,h);\mathbb{R}^n)).$$
We use the identity $u=\mathfrak{R}\circ\mathfrak{R}^c u=\mathfrak{R}(v_0,v_1)$, where $v_j$ solves the problem
\begin{equation}
\label{eq:ODE2}
\left\{\begin{aligned}
\lambda v_j - v_j''  &=f_j+B_ju &&\text{in}&&\mathbb{R}_+ ,\\
v_0(0) & =0   , \\
v_1'(0) &=0   ,
\end{aligned}\right.
\end{equation}
with $(f_0,f_1)=\mathfrak{R}^c f$ and
$$B_j\in\mathcal{L}(H_q^1((0,h);\mathbb{R}^n),L_q(\mathbb{R}_+;\mathbb{R}^n)).$$
Let us denote by $L_j$ the negative 1D Laplacian with either Dirichlet BCs ($j=0$) or Neumann BCs ($j=1$) on $\mathbb{R}_+$. Since the operators $L_j$ admit a bounded $H^\infty$-calculus in $L_q(\mathbb{R}_+;\mathbb{R}^n)$ with $H^\infty$-angle 0, they are in particular sectorial with the same angle. Therefore, the identity
\begin{equation}
\label{eq:ODE3}
u=(\lambda+\cA_1)^{-1} f=\xi_0(\lambda+L_0)^{-1} (f_0+B_0 u)+\xi_1\varphi^*(\lambda+L_1)^{-1} (f_1+B_1 u)
\end{equation}
and perturbation theory imply that there exist $\lambda_0>0$ and $C>0$ such that for all $\lambda\in \lambda_0+\Sigma_\pi$ it holds that
\begin{equation}
\label{eq:ODE4}
|\lambda|\|u\|_{L_q((0,h);\mathbb{R}^n)}+\|u\|_{H_q^2((0,h);\mathbb{R}^n)}\le C\|f\|_{L_q((0,h);\mathbb{R}^n)}.
\end{equation}
We use \eqref{eq:ODE3} once more to obtain the result
\begin{equation*}
u=(\lambda+\cA_1)^{-1} f=\xi_0(\lambda+L_0)^{-1} f_0+\xi_1\varphi^*(\lambda+L_1)^{-1} f_1+R(\lambda)f,
\end{equation*}
where
$$R(\lambda)f:=\xi_0(\lambda+L_0)^{-1} B_0 u+\xi_1\varphi^*(\lambda+L_1)^{-1} B_1 u.$$
By \eqref{eq:ODE4} and since $L_j$ are sectorial, there exists $C>0$ such that for all $\lambda\in \lambda_0+\Sigma_\pi$ it holds that
\begin{equation*}
\|R(\lambda)f\|_{L_q((0,h);\mathbb{R}^n)}\le \frac{C}{|\lambda|^{3/2}}\|f\|_{L_q((0,h);\mathbb{R}^n)}.
\end{equation*}
Now the $H^\infty$-calculus of $\cA_1$ can be proven in the same way as in \cite[Proof of Proposition 3.2]{SSW25}.

\medskip\noindent
We have shown that $\lambda + \cA_0+\cA_1$ admits a bounded $H^\infty$-calculus for $\lambda>\lambda_0$.
The assertion for
$$\lambda +\cA= \lambda +\cA_0 +\cA_1 +\cQ$$
then follows from a perturbation result, see Proposition 3.3.14 and Corollary 3.3.15 in \cite{PrSi16}.

\end{proof}
Next we consider the  resolvent problem
\begin{equation*}
(\lambda +\cA) v  =f ,
\end{equation*}
which can be recast equivalently as
\begin{equation}
\label{resolvent-cA-explicit}
\left\{\begin{aligned}
\lambda v - \Delta_\sN   v +\Ric \,  v + B_H v  &=f &&\text{in}&&\sN ,\\
\partial_r v  & =0   &&\text{on}&&\Sigma_u,  \\
v &=0&& \text{on}&&\Sigma_b.
\end{aligned}\right.
\end{equation}
The resolvent $(\lambda +\cA)^{-1}$ enjoys the following surprising properties.
\goodbreak
\begin{proposition}
\label{pro: properties of lambda-A}
For all $\lambda>\lambda_0$,
\begin{itemize}
\vspace{1mm}
\item[{\rm (a)}]
$(\lambda +\cA)^{-1} L_{q,\overline{\sigma}}(\sN;T\Sigma)\subset D(\sA)$.
\vspace{2mm}
\item[{\rm (b)}]
$(\lambda  + \cA)D(\sA) \subset  L_{q,\overline{\sigma}}(\sN;T\Sigma)$.
\end{itemize}
\end{proposition}
\begin{proof}
(a)
Given  $f\in L_{q,\overline{\sigma}}(\sN;T\Sigma)$,
the resolvent problem~\eqref{resolvent-cA-explicit} admits a unique solution $v\in D(\cA)$ by Proposition \ref{prop: H calculus of A}. It remains to show that $\wdv  \bar{v}=0$ in view of
Proposition~\ref{pro: properties-Ph}(c).
Taking the average in the vertical direction of \eqref{resolvent-cA-explicit}$_1$
and using the fact that $\Delta_\sN = \Delta_\Sigma  + \partial^2_r$
yields
\begin{equation}
\label{averagee resolvent problem}
(\lambda - \wDel + \Ric ) \bar{v}  + \frac{1}{h} \PH  (\partial_r v |_{\Sigma_b}) = \bar{f}.
\end{equation}
We consider the elliptic problem
\begin{equation*}
\Delta_B \phi = \lambda\, \wdv \bar{v}\quad\text{on } \Sigma,
\end{equation*}
where $\Delta_B := \wdv \,  \wgd$ is the Laplace-Beltrami operator on $\Sigma$.
Direct computations, based on  \eqref{div-scalar-Sigma} and \eqref{Green-5}, yield
\begin{equation}
\label{commutator-Ricci}
\begin{split}
\| \gd_\Sigma \phi \|_{L_2(\Sigma)}^2 & = (- { \Delta_B} \phi , \phi )_{\Sigma}= -\lambda (\wdv \bar{v}  , \phi )_{\Sigma} = \lambda ( \bar{v}  , \wgd \phi)_{\Sigma} \\
& = (\wDel \bar{v}   , \wgd \phi)_{\Sigma} - (\Ric \bar{v}  , \wgd \phi)_{\Sigma} \\
& \quad -  h^{-1} ( { \PH (\partial_r } v|_{\Sigma_b}) , \wgd \phi)_{\Sigma} + ( \bar{f} , \wgd \phi)_{\Sigma}\\
& = (\wDel \bar{v}   , \wgd \phi)_{\Sigma} - (\Ric \bar{v}  , \wgd \phi)_{\Sigma}  \\
& = ( \bar{v}  , \wDel \wgd \phi)_{\Sigma} - (\Ric \bar{v}  , \wgd \phi)_{\Sigma} \\
& =( \bar{v}   , \wgd \Delta_B  \phi)_{\Sigma} =- (\wdv  \bar{v}  , \Delta_B \phi )_{\Sigma}= -\lambda \| \wdv \bar{v} \|_{L_2(\Sigma)}^2.
\end{split}
\end{equation}
In \eqref{commutator-Ricci}$_4$, we have used the properties of $\PH$ and \eqref{div-scalar-Sigma}.
In \eqref{commutator-Ricci}$_5$, we have applied  integration by parts twice, see~\eqref{Green-5}.
Finally, in \eqref{commutator-Ricci}$_6$,  we have used the fact that
$$
\wDel \wgd \phi = \wgd \Delta_B  \phi +    \Ric \, \wgd  \phi,
$$
see \cite[Lemma~B.2]{SSW25}.
For $\lambda>0$, this implies $\wdv \bar{v}=0$.
Therefore,
$$v=(\lambda + \cA)^{-1}\in H^2_{q,\overline{\sigma}}(\sN;T\Sigma)\cap D(\cA)=D(\sA).$$

\medskip\noindent
(b)
Assume that $v\in D(\sA)$. Let $f= (\lambda +\cA) v \in L_q(\sN;T\Sigma)$. Applying $\wdv$ to both sides of \eqref{averagee resolvent problem}, we obtain
$$
\wdv \bar{f}= -\wdv(  (\wDel - \Ric )\bar{v}),
$$
where we have used Proposition~\ref{pro: properties-Ph}(c)  to conclude that $\wdv \bar v=0$.
The equality is understood in the weak sense as in Proposition ~\ref{pro: properties-Ph}(c).
Given any
$$u=u^k \partial_k\in C^\infty  (\Sigma;T\Sigma),$$
employing the relation   $g^{ij}_{|k}=0$ for all $  i,j,k \in \{1, 2\}$, one obtains
\begin{equation}
\label{commutator-divergence}
\begin{split}
 \wdv \wDel u & = (g^{ij} u^k_{| i\, | j})_{|k} =  g^{ij} u^k_{| i\, | j \, | k}\\
 &= g^{ij} u^k_{| i\, | k \, | j} + g^{ij}( u^k_{| i\, | j \, | k} -  u^k_{| i\, | k \, | j})\\
& = g^{ij} u^k_{| i\, | k \, | j} + g^{ij} (u^m_{|i} (\wRic)_{jm} - u^k_{|m} R^m_{kji} ) \\
&=  g^{ij} u^k_{| k\, | i \, | j} +  g^{ij} (u^k_{| i\, | k \, | j} -   u^k_{| k\, | i \, | j})  \\
&= \Delta_B \wdv u + g^{ij} ((\wRic)_{im} u^m )_{|j}  \\
&= \Delta_B \wdv u +  ((\Ric)_m^j u^m )_{|j}  \\
&= \Delta_B \wdv u +  \wdv(\Ric \, u  )  .
\end{split}
\end{equation}
%
%
In \eqref{commutator-divergence}$_3$, we have applied the Ricci identity, cf. \cite[formula (A.3)]{SaTu20}.
Additionally,  we used the relation
$$
g^{ij} R^m_{kji} = R^m_k = g^{ml} \wRic_{kl}=  g^{ml} \wRic_{lk},
$$
see for instance \cite[formula (A.1)]{SaTu20}, to further conclude that
 $ g^{ij} (u^m_{|i} (\wRic)_{jm} - u^k_{|m} R^m_{kji} )=0.$
We refer to \eqref{covariant-local} for the notation $u^k_{|\,j}$.
A standard density argument yields
\begin{align*}
\wdv \bar f= -\wdv(  \wDel - \Ric )\bar{v} = -\Delta_B \wdv \bar{v}  =0.
\end{align*}
This implies that $f\in L_{q,\overline{\sigma}}(\sN;T\Sigma)$.
\end{proof}
\noindent
Proposition~\ref{pro: properties of lambda-A} shows  that  for $\lambda>\lambda_0$,
$$
(\lambda + \cA)|_{D(\sA)} \in \Lis(D(\sA), L_{q,\overline{\sigma}} {(\sN;T\Sigma))}.
$$
Thus, given any $v\in D(\sA)$,  Proposition~\ref{pro: properties of lambda-A}(b) yields
$$
(\lambda -\Delta_\sN  + \Ric )   v +B_H v \in L_{q,\overline{\sigma}}(\sN;T\Sigma)).
$$
Therefore, if $v\in D(\sA)$,
\begin{align*}
(\lambda -\Delta_\sN  + \Ric )   v +B_H v = \Ph (\lambda -\Delta_\sN  + \Ric )   v + \Ph B_H v = \lambda v -\Ph(\Delta_\sN - \Ric ) v .
\end{align*}
This implies by  \eqref{def: hydrostatic-Stokes} and \eqref{def: cal-A}
\begin{equation}
\label{inclusion}
\lambda + \sA    = (\lambda +\cA-  2 \Ph \Ric)|_{D(\sA)}  .
\end{equation}
We are now ready to state the main result of this section. The assertion parallels a corresponding result in
 \cite[Theorem 3.1]{GGHHK17}, where the simpler setting of flat Euclidean space is considered.
\begin{theorem}
\label{thm: H-infinity calculus}
There exists  $\lambda_0\geq 0$ such that for  all $\lambda>\lambda_0$,
the operator  $\lambda + \sA$ admits a bounded $H^\infty$-calculus on $L_{q,\bar\sigma}(\sN, T\Sigma)$
 with $H^\infty$-angle $\phi^\infty_{\sA}<\pi/2$.
\end{theorem}
\begin{proof}
The assertion follows from Proposition~\ref{prop: H calculus of A}, \eqref{inclusion} and \cite[Proposition 3.3.14 and Corollary 3.3.15]{PrSi16}.
\end{proof}

We will also provide a result for the \emph{hydrostatic Stokes operator} $\sA$ in higher order Bessel potential spaces. We start with the following
result,  which is interesting in its own right.
\begin{lemma}\label{lem:highOrd}
Let $q\in (1,\infty)$ and $s\ge 0$. Then there exists $\lambda_0>0$ such that for all $\lambda\ge \lambda_0$ it holds that
$$(\lambda-\Delta_{\sN})^{-1}:H^s_{q} (\sN;T\Sigma)\to \{v  \in H^{2+s}_{q} (\sN;T\Sigma) : \, v=0 \text{ on }\Sigma_b, \ \ \partial_r v=0 \text{ on }\Sigma_u\}.$$
\end{lemma}
\begin{proof}
For $\lambda>0$ and $f\in H^s_{q} (\sN;T\Sigma)$, we consider the equation
\begin{equation*}
\lambda u-\Delta_{\sN} u=f.
\end{equation*}
By Proposition \ref{prop: H calculus of A}, there exists a unique solution $u\in H^2_{q} (\sN;T\Sigma)$ satisfying the boundary conditions
\begin{equation}
\label{eq:HO2}
u=0 \text{ on }\Sigma_b, \ \ \partial_r u=0 \text{ on }\Sigma_u.
\end{equation}
We will prove that, in addition, $u\in H^{2+s}_{q} (\sN;T\Sigma)$. To this end, we apply a localization procedure as in \cite[Section 3.1]{SSW25}. In the sequel, we will  sketch  some of the details. We define an atlas $\{(U_k, \varphi_k)\}_{k\in \cK}$  of  $\sN$  with $\cK=\cK_0\sqcup \cK_1\sqcup \cK_2$ such that
$k\in \cK_0$ if $U_k\cap (\Sigma_u\cup \Sigma_b)=\emptyset$, $k\in \cK_1$ if $U_k\cap \Sigma_b \neq \emptyset$ and $k\in \cK_2$ if $U_k\cap \Sigma_u \neq \emptyset$. In local coordinates, equation \eqref{eq:HO2} then reads
\begin{equation*}
(\lambda - \bar{g}^{ij}_{(k)} \partial_i \partial_j ) \bar{u}_k     =\bar{f}_k + P_k  (u)  \quad \text{in } \bX_k ,
\end{equation*}
with $\bX_k=\bR^3$ if $k\in\cK_0$ and $\bX_k=\bR_+^3$ if $k\in \cK_1\sqcup \cK_2$. Here, $\bar{f}_k\in H_q^s(\bX_k;\bR^2)$, $\bar{G}_{(k)}:=[\bar{g}^{ij}_{(k)}]^i_j $ belong to  $BC^\infty(\bX_k; \bR^{3\times 3})$ and $\|\bar{G}_{(k)}-I_3\|_{\infty}$  can be made arbitrarily small. Furthermore,
$$
P_k\in \cL(H^1_q(\sM;T\Sigma), L_q(\bX_k;\bR^2)).
$$
In view of the boundary conditions on $\bR^2\simeq\partial\bR_+^3$, we have $\bar{u}_k=0$ on $\bR^2$, if $k\in \cK_1$ and $\partial_3\bar{u}_k=0$ on $\bR^2$, if $k\in \cK_2$.

Since $P_k$ is of lower order and $ \bar{g}^{ij}_{(k)} \partial_i \partial_j$ is a perturbation of the Laplacian $\Delta_{\bX_k}=\sum_{j=1}^3\partial_j^2$, it suffices to show that for given $\bar{f}_k\in H_q^s(\bX_k;\bR^2)$, the solution of the problem
\begin{equation*}
\left\{\begin{aligned}
(\lambda - \Delta_{\bX_k}) \bar{v}_k   &=\bar{f}_k &&\text{for}&&k\in \cK ,\\
 \bar{v}_k   &=0 &&\text{for}&&k\in\cK_1  , \\
\partial_3\bar{v}_k   &=0 &&\text{for}&&k\in\cK_2 ,
\end{aligned}\right.
\end{equation*}
satisfies $\bar{v}_k\in H_q^{2+s}(\bX_k;\bR^2)$.

For $k\in \cK_0$ (full space problem), this follows from the lifting property of the Bessel potential spaces, see e.g. \cite[Theorem 2.3.4]{Tri78}. For $k\in \cK_1\sqcup \cK_2$ one may apply semigroup theory. Indeed, by \cite[Lemma 2.9.3]{Tri78} there exists a linear and bounded extension operator $E:H_q^{s}(\bR_+^3;\bR^2)\to H_q^{s}(\bR^3;\bR^2)$ with corresponding retraction $R$ to $\bR_+^3$. It follows that
$$\bar{v}_k=R(\lambda-\Delta_{\bR^3})^{-1}E\bar{f}_k+\bar{w}_k,$$
where $\bar{w}_k$ solves
\begin{equation}
\label{eq:HO5}
\left\{\begin{aligned}
(\lambda - \Delta_{\bX_k}) \bar{w}_k   &=0 &&\text{for}&&k\in \cK ,\\
 \bar{w}_k   &=\bar{g}_k &&\text{for}&&k\in\cK_1  , \\
\partial_3\bar{w}_k   &=\bar{g}_k &&\text{for}&&k\in\cK_2 ,
\end{aligned}\right.
\end{equation}
and
\begin{align*}
\bar{g}_k :=
\begin{cases}
-[R(\lambda-\Delta_{\bR^3})^{-1}E\bar{f}_k]|_{\partial\bR_+^3}\in B_{qq}^{2+s-1/q}(\bR^2;\bR^2), \quad & \text{if } k\in \cK_1, \\
-[\partial_3R(\lambda-\Delta_{\bR^3})^{-1}E\bar{f}_k]|_{\partial\bR_+^3}\in B_{qq}^{1+s-1/q}(\bR^2;\bR^2), \quad & \text{if } k\in \cK_2 ,
\end{cases}
\end{align*}
since
$$R(\lambda-\Delta_{\bR^3})^{-1}E\bar{f}_k\in  H_q^{2+s}(\bR_+^3;\bR^2),$$
by the lifting property for the operator $\lambda-\Delta_{\bR^3}$. With the help of semigroup theory, the solution $\bar{w}_k$ of \eqref{eq:HO5} may be explicitly written as
\begin{align*}
\bar{w}_k=
\begin{cases}
e^{-L_\lambda x_3}\bar{g}_k, \quad & \text{if } k\in \cK_1, \\
-e^{-L_\lambda x_3}L_\lambda^{-1}\bar{g}_k, \quad & \text{if } k\in \cK_2 ,
\end{cases}
\end{align*}
where $L_\lambda:=(\lambda-\Delta_{\bR^2})^{1/2}$. The lifting property for the operator $L_\lambda$ yields $\bar{w}_k\in H_q^{2+s}(\bR_+^3;\bR^2)$.
Therefore, $\bar{v}_k$ enjoys the same regularity.

By perturbation theory and with the help of a partition of unity, one may then include the perturbed coefficients $ \bar{g}^{ij}_{(k)}$ as well as the lower order terms given by $P_k$ to obtain the desired result for $\Delta_\sN$.
\end{proof}

By the same strategy as in the proof of Lemma \ref{lem:highOrd}, one can prove that there exists $\lambda_0>0$ such that for all $\lambda\ge\lambda_0$ and all $s\ge 0$ it holds that
\begin{equation}
\label{eq:HO6}
(\lambda-\Delta_{B})^{-1}:H^s_{q} (\Sigma)\to  H^{2+s}_{q} (\Sigma),
\end{equation}
and
\begin{equation}
\label{eq:HO7}
\lambda\|u\|_{H_q^s(\Sigma)}+\|u\|_{H_q^{2+s}(\Sigma)}\lesssim \|f\|_{H_q^s(\Sigma)},
\end{equation}
where $u:=(\lambda-\Delta_{B})^{-1}f$ and $\Delta_B$ denotes the Laplace-Beltrami operator on $\Sigma$.

Indeed, the proof is much simpler, since $\partial\Sigma=\emptyset$ and thus, by using local coordinates, one can reduce the regularity problem to the Laplace operator in $H_q^s(\bR^2)$ combined with the lifting property, see \cite[Theorem 2.3.4]{Tri78}. We refrain from giving the details.

Making use of \eqref{eq:HO6} and \eqref{eq:HO7} one can improve the first part of Proposition \ref{pro: Helmholtz-Sigma} (a).
\begin{proposition}\label{pro:HP_improved}
For each $q\in (1,\infty)$ and all $s\ge 0$ it holds that
$$\PH\in \cL(H^s_q (\Sigma; T\Sigma) , H^s_{q, \sigma} (\Sigma ; T\Sigma)).$$
\end{proposition}
\begin{proof}
We recall that
$$\PH u := u - \wgd \psi_u$$
and $\psi_u$ solves $\Delta_B\psi_u=\wdv  u$, provided $u\in H_q^1(\Sigma;T\Sigma)$. Furthermore, $\psi_u$ satisfies the estimate
$$\|\wgd \psi_u\|_{H_q^1(\Sigma)}\lesssim \|u\|_{H_q^{1}(\Sigma)},$$
see \cite[Lemma B.6]{SSW25} for the case of compact manifolds with boundary. A similar argument applies to a closed manifold.
W.l.o.g. we may always assume that $\int_\Sigma\psi_ud\mu_g=0$, since $\psi_u$ is unique up to additive constants. The Poincar\'{e}-Wirtinger inequality therefore implies
$$\|\psi_u\|_{H_q^{2}(\Sigma)}\lesssim \|u\|_{H_q^{1}(\Sigma)}.$$
We consider then the shifted problem
\begin{equation}
\label{eq:ResPrbPsi_u}
\lambda \psi_u-\Delta_B\psi_u=\lambda\psi_u-\wdv  u
\end{equation}
for sufficiently large $\lambda>0$ such that \eqref{eq:HO6}-\eqref{eq:HO7} are applicable. We claim that for all $k\in\mathbb{N}$ it holds that
$$u\in H_q^k(\Sigma;T\Sigma)\Rightarrow \psi_u\in H_q^{k+1}(\Sigma)$$
and
$$\|\psi_u\|_{H_q^{k+1}(\Sigma)}\lesssim \|u\|_{H_q^{k}(\Sigma)}.$$
This can be seen by induction w.r.t. $k$. Assume that the claim is true for some $k\in\mathbb{N}$. Let $u\in H_q^{k+1}(\Sigma;T\Sigma)$. Then, by the induction hypothesis, $\psi_u\in H_q^{k+1}(\Sigma)$, hence $\lambda\psi_u-\wdv u\in H_q^{k}(\Sigma)$ and therefore $\psi_u\in H_q^{k+2}(\Sigma)$, by \eqref{eq:HO6} and \eqref{eq:ResPrbPsi_u}. Furthermore, by \eqref{eq:HO7} and the induction hypothesis,
\begin{align*}
\|\psi_u\|_{H_q^{k+2}(\Sigma)}&\lesssim \|\psi_u\|_{H_q^{k}(\Sigma)}+\|\wdv u\|_{H_q^{k}(\Sigma)}\\
&\lesssim \|\psi_u\|_{H_q^{k+1}(\Sigma)}+\|u\|_{H_q^{k+1}(\Sigma)}\\
&\lesssim \|u\|_{H_q^{k+1}(\Sigma)}.
\end{align*}
Evidently, this yields
$$\|\PH u\|_{H_q^k(\Sigma)}\le \|u\|_{H_q^k(\Sigma)}+\|\wgd \psi_u\|_{H_q^k(\Sigma)}\le  \|u\|_{H_q^k(\Sigma)}+\|\psi_u\|_{H_q^{k+1}(\Sigma)}\lesssim \|u\|_{H_q^k(\Sigma)}$$
for all $k\in\mathbb{N}$. The case of non-integer $s>0$ follows from complex interpolation.
\end{proof}
It is an immediate consequence of Proposition \ref{pro:HP_improved} that
$$\| \Ph v\|_{H^k_q(\sN)}   \leq \|\PH \bar{v}\|_{H^k_q(\sN)}  + \| v - \bar{v}\|_{H^k_q(\sN)}   \lesssim \|v\|_{H^k_q(\sN)},
$$
for all $k\in\bN$, hence
\begin{equation}
\label{eq:bdd_hydrHP_improved}
\Ph\in \cL(H^s_q (\sN; T\Sigma) , H^s_{q, \overline{\sigma}} (\sN ; T\Sigma))
\end{equation}
for all $s\ge 0$ (by complex interpolation for non-integer $s>0$). By \eqref{inclusion} we know that on $D(\sA)$, $\sA=-\Delta_{\sN}$, modulo lower order terms. Then, Lemma \ref{lem:highOrd}, perturbation theory and the boundedness of $\Ph$  yield the following result.
\begin{theorem}\label{thm:highReg_hydrStokes}
Let $q\in (1,\infty)$ and $s\ge 0$. Then there exists $\lambda_0>0$ such that for all $\lambda\ge \lambda_0$ it holds that
$$(\lambda+\sA)^{-1}:H^s_{q,\bar{\sigma}} (\sN;T\Sigma)\to \{v  \in H^{2+s}_{q,\bar{\sigma}} (\sN;T\Sigma) : \, v=0 \text{ on }\Sigma_b, \ \ \partial_r v=0 \text{ on }\Sigma_u\}.$$
\end{theorem}

\section{Local Wellposedness}\label{S:local wellposedness}
We will now turn our attention to the primitive equations~\eqref{PE} and we study the existence, uniqueness and regularity of solutions.
As in \cite{HiKa16, GGHHK20},  we apply the  hydrostatic projection $\Ph$  to \eqref{PE}.
As a result, the pressure term $\pi_s$ will be eliminated and we end up with the following formulation
\begin{equation}
\label{PE-abstract}
\left\{\begin{aligned}
  \partial_t v + \sA v    &= \sF (v)     ,\\
v(0) & = v_0 ,
\end{aligned}\right.
\end{equation}
where  $\sF(v)=F(v,v)$, with
$$
F(v_1,v_2)= - \Ph \left( \nabla_{v_1} v_2    + w (v_1) \partial_r  v_2  \right)  , \quad \text{and }w(v)(\cdot,r):=\int_r^0 \dv_\Sigma v (\cdot,\xi)\, d\xi .
$$
Thanks to Theorem~\ref{thm: H-infinity calculus} we can use well-established arguments to analyze  system~\eqref{PE-abstract}, see
\cite{PrSi16, PrWi17, PSW18} and also \cite{HiKa16, GGHHK20, GGHHK17}.
For the readers' convenience, we will include detailed and complete proofs.

\noindent
Let
$$X_1 = D(\sA){=\{v  \in H^2_{q,\bar{\sigma}} (\sN;T\Sigma) : \, v=0 \text{ on }\Sigma_b, \ \  \partial_r v=0 \text{ on }\Sigma_u\}}$$
and $X_0 = L_{q,\bar{\sigma}}(\sN;T\Sigma)$.
We seek a solution in the maximal regularity space
$$
\bE_{1,\mu}(T):= H^1_{p,\mu} ((0,T); X_0) \cap L_{p,\mu}((0,T); X_1),
$$
for some $\mu\in (1/p,1]$ and $p,q\in (1,\infty)$. We define
$$
\bE_{0,\mu}(T):= L_{p,\mu} ((0,T); X_0)
$$
and the initial data space
$$
X_{\mu-1/p,p}:= (X_0,X_1)_{\mu-1/p,p},
$$
where $(\cdot , \cdot)_{\theta,p}$ with $\theta\in [0,1]$ is the real interpolation functor. Similarly, the complex interpolation spaces {are} defined by
$$
X_\theta :=[X_0, X_1]_\theta.
$$
It is known that
$$
\bE_{1,\mu}(T) \hookrightarrow C \left([0,T]; X_{\mu-1/p,p} \right).
$$
The following proposition gives a complete characterization of these {interpolation} spaces.

\begin{proposition}\label{prop: interpolation}
Let $1<q, p<\infty$,  and $\theta\in (0,1) \setminus \{1/2q, 1/2+1/2q\}$. Then
{
\begin{align*}
X_\theta =
\begin{cases}
\left\{v  \in H^{2\theta}_{q,\bar{\sigma}} (\sN;T\Sigma) : \, v=0 \text{ on }\Sigma_b  \text{ and }  \partial_r v=0 \text{ on }\Sigma_u \right\},\ & 1/2+1/2q<\theta \leq 1 ;\\
\left\{v  \in H^{2\theta}_{q,\bar{\sigma}} (\sN;T\Sigma) : \, v=0 \text{ on }\Sigma_b   \right\}, \ & 1/2q<\theta < 1/2+1/2q  ;\\
  H^{2\theta}_{q,\bar{\sigma}} (\sN;T\Sigma)    , \quad & 0< \theta < 1/2q  ,
\end{cases}
\end{align*}
and
\begin{equation*}
X_{\theta,p} =
\begin{cases}
\left\{v  \in B^{2\theta}_{qp,\bar{\sigma}} (\sN;T\Sigma) : \, v=0 \text{ on }\Sigma_b  \text{ and }  \partial_r v=0 \text{ on }\Sigma_u \right\}, & 1/2+1/2q<\theta \leq 1 ;\\
\left\{v  \in B^{2\theta}_{qp,\bar{\sigma}} (\sN;T\Sigma) : \, v=0 \text{ on }\Sigma_b   \right\}, & 1/2q<\theta < 1/2+1/2q  ;\\
  B^{2\theta}_{qp,\bar{\sigma}} (\sN;T\Sigma)    , \quad & 0< \theta < 1/2q  ,
\end{cases}
\end{equation*}
where $ B^s_{qp,\bar{\sigma}} (\sN;T\Sigma) {:=}  B^s_{qp } (\sN;T\Sigma) \cap L_{q,\bar{\sigma}}(\sN;T\Sigma)  $ for $s\geq 0$ with
\begin{equation}
\label{Def:Besov}
B^s_{qp} (\sN;T\Sigma) :=
\begin{cases}
\left( H^{k}_q   (\sN;T\Sigma) , H^{k+1}_q   (\sN;T\Sigma)\right)_{s-k,p} ,\quad &\text{when } k<s<k+1;\\
\left( H^{k-1}_q   (\sN;T\Sigma) , H^{k+1}_q   (\sN;T\Sigma)\right)_{1/2,p}  , & \text{when } s=k\in \bN
\end{cases}
\end{equation}
for $s\in  (0,\infty)$.
}
\end{proposition}
\begin{proof}
By using Proposition~\ref{prop: interpolation} {and} Proposition~\ref{pro: properties-Ph}, the above interpolation results can be obtained by following the proofs of \cite[Propositions~C.3 and C.4]{SSW25}.
\end{proof}

Theorem~\ref{thm: H-infinity calculus} implies that for any $T>0$ and $\mu \in (1/p,1]$, {it holds that}
\begin{equation}
\label{def-S}
\left[v \mapsto (\partial_t v + \sA v , \gamma v ) \right] \in \Lis (\bE_{1,\mu}(T), \bE_{0,\mu}(T) \times X_{\mu-1/p,p}),
\end{equation}
where $\gamma v = v(0)$ is the {trace operator at $t=0$, see for instance \cite{DHP03, PrSi16}.
We then set
\begin{equation}
\label{sA-MR}
\sS:=  (\partial_t + \sA  , \gamma )^{-1}\in \cL(\bE_{0,\mu}(T) \times X_{\mu-1/p,p} , \bE_{1,\mu}(T)).
\end{equation}
In order to estimate the nonlinear term $\sF$, we first establish the following auxiliary result, which is also of independent interest.

\begin{lemma}\label{lem:anisotropic space}
Let $s,\tau\ge 0$ and $q\in (1,\infty)$. Then
\begin{align*}
H^{\tau+s}_q(\sN; T\Sigma)  \hookrightarrow  H^{\tau}_q((-h,0); H^s_q(\Sigma; T\Sigma))  .
\end{align*}
\end{lemma}
\begin{proof}
By \cite[Theorem 6.1.8]{PrSi16} the operator $A:=I-\partial_r^2$ has a bounded $H^\infty$-calculus with $H^\infty$-angle $\phi_A^\infty=0$ in $L_q(\bR;E)$, where $E$ is any UMD-space.

Moreover, by \cite[Proposition 3.2]{SSW25}, (for the case $\partial\Sigma=\emptyset$), the operator $B:=I-\Delta_\Sigma$ has a bounded $H^\infty$-calculus with $H^\infty$-angle $\phi_B^\infty<\pi/2$ in $L_q(\Sigma;T\Sigma)$.
As in the proof of Proposition~\ref{prop: H calculus of A} one can show that the operator
$$C:=2I-\partial_r^2-\Delta_\Sigma=A+B$$
has a bounded $H^\infty$-calculus with $H^\infty$-angle $\phi_C^\infty\le\phi_B^\infty$ in $L_q(\bR\times \Sigma;T\Sigma)$.
Since $A,B,C$ have bounded imaginary powers, we obtain from \cite[Theorem 1.15.3]{Tri78} that
$$D(A^{\tau/2})=H_q^\tau(\bR;E),\quad D(B^{s/2})=H_q^s(\Sigma;T\Sigma),\quad D(C^{\sigma/2})=H_p^\sigma(\bR\times\Sigma;T\Sigma)$$
for any $\tau,s,\sigma\ge 0$. Indeed, for each $k\in\bN$ it holds that
$$D(A^k)=H_q^{2k}(\bR;E),\quad D(B^k)=H_q^{2k}(\Sigma;T\Sigma),\quad D(C^{k})=H_p^{2k}(\bR\times\Sigma;T\Sigma),$$
which follows from a localization procedure as in the proof of Lemma \ref{lem:highOrd}.

Let $\tilde{u}\in H_q^{\tau+s}(\bR\times\Sigma;T\Sigma)$. From the above considerations and with $E:=H_q^s(\Sigma;T\Sigma)$ we then obtain
\begin{equation}
\label{MW}
\begin{aligned}
\| \tilde{u} \|_{H^{\tau}_q(\bR; H^s_q(\Sigma;T\Sigma))}&\lesssim \| A^{\tau/2} \tilde{u} \|_{L_q(\bR;H_q^s(\Sigma;T\Sigma))}\\
&\lesssim \| A^{\tau/2}   B^{s/2} \tilde{u} \|_{L_q(\bR\times\Sigma;T\Sigma)} \\
&\lesssim \| C^{(\tau+s)/2} \tilde{u} \|_{L_q(\bR\times\Sigma;T\Sigma)} \\
&\lesssim \|\tilde{u}\|_{H_q^{\tau+s}(\bR\times\Sigma;T\Sigma)}.
\end{aligned}
\end{equation}
In the above estimate~\eqref{MW},  we have furthermore used the fact that the operator $$A^{\tau/2}B^{s/2}[C^{(\tau+s)/2}]^{-1}$$
is bounded in $L_q(\bR \times \Sigma; T\Sigma)$,
 which follows for instance  from the joint functional calculus, see \cite[Corollary 4.5.8]{PrSi16}, and by extending $B$ canonically to $L_q(\bR\times \Sigma;T\Sigma)$. Let
$$
S:H^{\tau+s}_q((-h,0)\times \Sigma; T\Sigma)\to H^{\tau+s}_q(\bR\times \Sigma; T\Sigma)
$$
be a linear and bounded extension operator, whose existence may be seen as in \cite[Theorem 5.19]{Ada03} or \cite[Lemma 2.5]{MeSc12}.
For any $u\in H^{\tau+s}_q((-h,0)\times \Sigma; T\Sigma)$ we set $\tilde{u}:=Su$. Then it holds that
$$
  \| u \|_{H^{\tau}_q((-h,0); H^s_q(\Sigma;T\Sigma))}  \leq \| \tilde{u} \|_{H^{\tau}_q(\bR; H^s_q(\Sigma;T\Sigma))}
  \lesssim \|\tilde{u}\|_{H_q^{\tau+s}(\bR\times\Sigma;T\Sigma)}
  \lesssim \|{u}\|_{H_q^{\tau+s}((-h,0)\times\Sigma;T\Sigma)}.
$$
\end{proof}

To prove  local wellposedness of \eqref{PE-abstract}, we will need the following estimate for the  term {$\sF(v)$}.
\begin{lemma}
\label{lem:bilinear}
Let $q\in (1,\infty)$ and $s\ge 0$. Then there exists some $C>0$ such that
\begin{itemize}
\item[{\em (a)}] 
\begin{equation}
\label{eq:estF}
\| F (v_1,v_2) \|_{H_q^{s}(\sN)} \leq C \| v_1\|_{H^{s+1+1/q}_{q} (\sN)}\| v_2\|_{H^{s+1+1/q}_{q}  (\sN)},
\end{equation}
for all $v_1,v_2\in H^{s+1+1/q}_{q,\bar{\sigma}}  (\sN;T\Sigma)$. Moreover, $\sF\in C^\infty(H_{q,\bar{\sigma}}^{s+1+1/q}(\sN;T\Sigma);H_{q,\bar{\sigma}}^{s}(\sN;T\Sigma))$.
\vspace{1mm}
\item[{\em (b)}]
$
\| \sF (v_1) - \sF(v_2) \|_{X_0}
 \leq C \big( \| v_1\|_{X_{\frac{1}{2} + \frac{1}{2q}}} + \| v_2\|_{X_{\frac{1}{2} + \frac{1}{2q}}} \big) \| v_1-v_2\|_{X_{\frac{1}{2} + \frac{1}{2q}}} ,
$
for all \\
$v_1,v_2\in X_{\frac{1}{2} + \frac{1}{2q}}$.
\end{itemize}
\end{lemma}
\begin{proof}
(a) It follows from \eqref{eq:bdd_hydrHP_improved} that
\begin{align*}
\| F (v_1,v_2) \|_{L_q (\sN)}& \lesssim   \| \nabla_{v_1} v_2  \|_{L_q (\sN)} +  \| w (v_1) \partial_r  v_2 \|_{L_q (\sN)}.
\end{align*}
By \eqref{pointwise-estimate}, H\"{o}lder's inequality and  the Sobolev embedding theorem \cite[Theorems 14.1 and 14.2]{Ama13}, it holds that
\begin{align*}
\| \nabla_{v_1} v_2  \|_{L_q (\sN)}  \leq \|v_1\|_{L_{3   q}(\sN)} \|   v_2 \|_{H^1_{3 q / 2}(\sN)} \leq C \| v_1\|_{H^{1+1/q}_{q}  (\sN)}\| v_2\|_{H^{1+1/q}_{q}  (\sN)}.
\end{align*}
It follows from Lemma~\ref{lem:anisotropic space} and \cite[Theorem~14.2]{Ama13} that
\begin{align*}
\| w (v_1) \partial_r  v_2 \|_{L_q (\sN)} & \leq  \| w (v_1) \|_{L_\infty((-h,0);L_{2q}(\Sigma))} \| \partial_r v_2 \|_{L_q((-h,0);L_{2q}(\Sigma))}\\
&\leq C \| \wdv v_1 \|_{L_q((-h,0);L_{2q}(\Sigma))} \|   v_2 \|_{H^1_q((-h,0);L_{2q}(\Sigma))}\\
& \leq C \|   v_1 \|_{L_q((-h,0);H^{1+1/q}_{q} (\Sigma))} \|   v_2 \|_{H^1_q((-h,0);H^{1/q}_{q}(\Sigma))}\\
& \leq C \| v_1\|_{H^{1+1/q}_{q}  (\sN)}\| v_2\|_{H^{1+1/q}_{q}  (\sN)}.
\end{align*}
This proves \eqref{eq:estF} in case $s=0$. For $s=k\in\bN$, \eqref{eq:estF} can be proven inductively, using \eqref{eq:bdd_hydrHP_improved} and Lemma \ref{lem:anisotropic space}. The case of non-integer $s>0$ is a consequence of complex interpolation for bilinear operators, see e.g. \cite[Section 1.19.5]{Tri78}. Finally, the smoothness of $\sF(v)=F(v,v)$ follows from the bilinear structure of $F$ and \eqref{eq:estF}.

\medskip\noindent
(b) The difference estimate is a direct consequence of part (a) for $s=0$ and the fact that the norm in the spaces $X_\beta$, $\beta\in [0,1]$, is given by $\|\cdot\|_{H_{q}^{2\beta}(\sN)}$.
\end{proof}

The following lemma will be useful for showing additional regularity of solutions to \eqref{PE-abstract}.

\begin{lemma}
\label{lem:F-cont}
{Let $p,q\in (1,\infty)$ such that $\frac{1}{p}+ \frac{1}{q}\leq 1$, $\mu\in [1/p+1/q,1]$} and $T>0$. Then  \\$\sF\in C^\infty(\bE_{1,\mu}(T), \bE_{0,\mu}(T))$ with
$$
\partial \sF (v)u= F(u,v) + F(u,v)
$$
and
\begin{equation}
\label{bilinear-est-time}
 \| F(u,v) \|_{\bE_{0,\mu}(T)} \leq C \| u \|_{\bE_{1,\mu}(T)} \| v \|_{\bE_{1,\mu}(T)} .
\end{equation}
Moreover, given any $v\in \bE_{1,\mu}(T)$, {it holds that}
\begin{equation}
\label{solution map}
\left[ u \mapsto  (\partial_t u + \sA u  - { \partial \sF(v)} u , \gamma u ) \right] \in \Lis (\bE_{1,\mu}(T), \bE_{0,\mu}(T) \times X_{\mu-1/p,p}).
\end{equation}
\end{lemma}
\begin{proof}
In this proof, for any Banach space $X$, and a function space $\mathfrak{F}((0,T);X)$,
the space $\prescript{}{0}{\mathfrak{F}} ((0,T);X)$ denotes the subspace of functions with zero initial trace (whenever it is well-defined). {For example,}
$$
\zbE_{1,\mu} (T) = \{ u\in \bE_{1,\mu} (T) : \,  \gamma u =0 \}.
$$
Based on \cite[Corollary~1.4]{MeVe12}, it was observed in \cite{PrWi17, PSW18} that for $2\beta -1 =\mu-1/p$ and $\eta= 1-\beta-1/2p + \mu $, it holds that
\begin{equation}
\label{emb-L2pmu}
\zbE_{1,\mu}(T) \hookrightarrow L_{2p,\eta}((0,T); X_\beta)
\end{equation}
with an embedding constant independent of $T$.\\
It follows from Lemma~\ref{lem:bilinear}(a)  and  H\"older's inequality  that
\begin{equation}
\label{est-L2pmu}
\| F(u,v) \|_{\bE_{0,\mu}(T)} \leq C \| u \|_{L_{2p,\eta}((0,T);X_\beta)}\|v\|_{L_{2p,\eta}((0,T);X_\beta)},
\end{equation}
since $X_\beta\hookrightarrow H_q^{1+1/q}(\sN;T\Sigma)$, provided $\mu\in [1/p+1/q,1]$.

The constant $C$ again is independent of $T$.
The bilinear estimate~\eqref{bilinear-est-time} then follows.
The smoothness of $\sF$ is a direct consequence of \eqref{bilinear-est-time} and its bilinear structure.

\smallskip
Next, we consider solvability of the system
\begin{equation}
\label{abstract-nonliner}
\left\{\begin{aligned}
  \partial_t u + \sA u - \partial \sF(v) u   &= f    ,\\
u(0) & = u_0 ,
\end{aligned}\right.
\end{equation}
with $f\in \bE_{0,\mu}(T)$ and $u_0 \in X_{\mu-1/p,p}$.
For some $T_0\in (0,T]$, to be determined later, we define the set
$$
W_{T_0}:= \left\{ \phi\in \bE_{1,\mu}(T_0) : \, \phi(0)=u_0  \right\};
$$
and for $\phi\in W_{T_0}$ we define $\sG (\phi):= u$
 iff $u$ solves
\begin{equation*}
\left\{\begin{aligned}
  \partial_t u + \sA u - F (\phi, v) - F(v ,\phi)   &= f ,\\
u(0) & = u_0 ,
\end{aligned}\right.
\end{equation*}
that is,
$$
\sG (\phi) = \sS(f + F (\phi, v) + F(v ,\phi) , u_0).  
$$
 Here we remind that $v\in \bE_{1,\mu}(T)$ is given and $\partial \sF(v)\phi = F (\phi, v) + F(v ,\phi).$
It  follows from~\eqref{sA-MR}  that
$$
\sG : W_{T_0}  \to W_{T_0}
$$
is well-defined.
In virtue of  \eqref{def-S},  \eqref{emb-L2pmu}, and \eqref{est-L2pmu}
\begin{equation}
\label{Lip est}
\begin{split}
  \| \sG (\phi_1)  - \sG (\phi_2) \|_{\bE_{1,\mu}(T_0)} & = \| \sS(  F (\phi_1 - \phi_2, v) + F(v ,\phi_1-\phi_2) , 0) \|_{\bE_{1,\mu}(T_0)} \\
  & \leq C \| F (\phi_1 - \phi_2, v) + F(v ,\phi_1-\phi_2)\|_{\bE_{0,\mu}(T_0)} \\
  & \leq C'  \| v \|_{L_{2p,\eta}((0,T);X_\beta)}        \| \phi_1 - \phi_2 \|_{\bE_{1,\mu}(T_0)} ,
\end{split}
\end{equation}
where the constant $C'$ is independent of $T_0$.
Since $\| v \|_{L_{2p,\eta}((0,T_0);X_\beta)} <\infty$, we have
$$
 \| v \|_{L_{2p,\eta}((t_0,t_1);X_\beta)} \to 0
$$
for $0\leq t_0 < t_1 \leq T$ as $|t_1-t_0|\to 0$.
By choosing $T_0>0$ so small that $\| v \|_{L_{2p,\eta}((0,T_0);X_\beta)}  < 1/2C'$, \eqref{Lip est} implies that
$\sG : W_{T_0 }\to  W_{T_0}$
is a contraction.

The contraction mapping principle yields the existence of a unique solution $u\in \bE_{1,\mu}(T_0)$ to \eqref{abstract-nonliner} on $[0,T_0]$.
Note that the step size $T_0$ is independent of {$(f,u_0)$} and only depends on $v$.
Therefore, by repeating the above argument for finitely many steps and following
the argument in \cite[Proposition~4.2]{DSS24}, we obtain a solution  $u\in \bE_{1,\mu}(T)$ to \eqref{abstract-nonliner}.

If \eqref{abstract-nonliner} admits two  solutions $u_1,u_2\in \bE_{1,\mu}(T_0)$, then $\hat{u}=u_1-u_2$ solves
\begin{equation*}
\left\{\begin{aligned}
  \partial_t \hat{u} + \sA \hat{u}   &=  F (\hat{u}, v) + F(v ,\hat{u})      ,\\
\hat{u} (0) & = 0 .
\end{aligned}\right.
\end{equation*}
Note that for any $0<a<\infty$
\begin{align*}
\| \hat u\|_{\bE_{1,\mu}(a)} & = \| \sS (F (\hat{u}, v) + F(v ,\hat{u}) , 0)\|_{\bE_{1,\mu}(a)}\\
&\leq C  \|F (\hat{u}, v) + F(v ,\hat{u}) \|_{\bE_{0,\mu}(a)}  \\
&\leq C'' \| v\|_{\bE_{1,\mu}(a)}  \|\hat{u} \|_{\bE_{1,\mu}(a)},
\end{align*}
where the constant $C'$ only depends on the norm of $\sS$ and the constant in \eqref{bilinear-est-time}.
For sufficiently small $a$ we have $\| v\|_{\bE_{1,\mu}(a)}<1/2C''$, and this implies $\hat{u}=0$ on $[0,a]$. By repeating this argument for finitely many times, we can conclude that $u_1=u_2$ on $[0,T]$. 
This establishes the assertion in~\eqref{solution map}.
\end{proof}

Now our main result for local well-posedness of \eqref{PE} is a direct consequence of \cite[Theorem~2.1]{PSW18} and Theorem~\ref{thm: H-infinity calculus}.
\begin{proposition}
\label{prop:local existence}
Let $p,q\in (1,\infty)$ such that $\frac{1}{p}+ \frac{1}{q}\leq 1$, 
and $\mu \in \left[ \frac{1}{p}+ \frac{1}{q}, 1 \right]$.
\begin{itemize}
\item[\em (a)]  For any
$
v_0\in X_{\mu-1/p,p}  ,
$
there exists a number $a=a(v_0)\in (0,\infty)$ and a unique solution $v \in \bE_{1,\mu}(a)$ to \eqref{PE-abstract} such that
$$
v\in C([0,a]; X_{\mu-1/p,p}) \cap C((0,a]; X_{1-1/p,p}).
$$
\vspace{-2mm}
\item[\em (b)]
 Let $T_+=T_+(v_0)=\sup \left\{a\in (0,\infty): \text{\eqref{PE-abstract} has a solution in }\bE_{1,\mu}(a) \right\}$.  Let  $0<T<T_+$.
Then  there exists $\rho>0$ such that for any  $u_0 \in B_{X_{\mu-1/p,p}}(v_0, \rho)$, \eqref{PE-abstract} with initial value $u_0$ has a unique solution
$u=u(\cdot, u_0) \in \bE_{1,\mu}(T)$.
 Moreover, it holds that
 $$
 [u_0 \mapsto  u(\cdot, u_0)]\in C^\infty( B_{X_{\mu-1/p,p}}(v_0, \rho), \bE_{1,\mu}(T)).
 $$
\item[\em (c)] For any $T\in (0,T_+)$ it holds  that
 $$[t \mapsto t^k \partial^k_t v]\in \bE_{1,\mu}(T),\quad k\in \bN,$$
and  hence $v\in C^\infty((0,T); X_1).$
\vspace{2mm}
\item[\em (d)]
For any $T\in (0,T_+)$ it holds  that
$$v\in C^\infty((0,T)\times\sN;T\Sigma).$$
\item[\em (e)]
Suppose that for  some $\delta\in (0,T_+)$ and $\bar{\mu} \in (\mu, 1],$
$
v \in BC \left([\delta, T_+(v_0)); X_{\bar{\mu}-1/p,p} \right).
$
\\
Then $T_+=\infty$.

\end{itemize}
\end{proposition}
\begin{proof}
(a)  The assertion is a direct consequence of Theorem~\ref{thm: H-infinity calculus}, Lemma~\ref{lem:bilinear}, and \cite[Theorem~1.2]{PrWi17},
see also \cite[Theorem~2.1]{PSW18}.

\medskip\noindent
(b) Let $T\in (0,T_+)$ be given and let $v\in \bE_{1,\mu}(T)$ be the solution of ~\eqref{PE-abstract}.
We define the map
\begin{align*}
&\Phi : \bE_{1,\mu}(T)\times  X_{\mu-1/p,p}  \to \bE_{1,\mu}(T), \quad
\Phi (u, u_0)= (\partial_t u +\sA u -\sF(u), \gamma u -u_0).
\end{align*}
It follows that
$\Phi (u,u_0)=0$ iff $u\in \bE_{1,\mu}(T)$ is a solution of the problem
\begin{equation}
\label{PE-abstract-2}
\left\{\begin{aligned}
  \partial_t u  + \sA  u    &= \sF(u)      ,\\
u(0) & = u_0.
\end{aligned}\right.
\end{equation}
We notice that $\Phi(v,v_0)=0$, as $v$ solves~\eqref{PE-abstract-2} with initial value $v_0$.
It follows from Lemma~\ref{lem:F-cont} that
$$\Phi\in C^\infty \left( \bE_{1,\mu}(T)\times  X_{\mu-1/p,p}  , \bE_{0,\mu}(T)  \right).$$
Next, we observe that
$$
\partial_1 \Phi(v, v_0)= (\partial_t +\sA -\partial \sF(v), \gamma)\in \Lis( \bE_{1,\mu}(T), \bE_{0,\mu}(T)\times X_{1-1/p, p}),
$$
where we employed Lemma~\ref{lem:F-cont} for the last assertion.
The implicit function theorem implies that there exists some $\rho>0$ and a function
$
\psi\in C^\infty( B_{X_{\mu-1/p,p}}(  0,\rho) , \bE_{1,\mu}(T) ),
$
such that
$$
\Phi(\psi(u_0), u_0)=0,\quad u_0 \in B_{X_{\mu-1/p,p}}(  0,\rho).
$$
Hence, $\psi(u_0)=u(\cdot, u_0)$ is the unique solution of~\eqref{PE-abstract-2}, and the assertion follows.

\medskip\noindent
(c) Let $T\in (0,T_+)$ be fixed and choose   $\varepsilon>0$ so small that $(1+\varepsilon)T < T_+$.

\smallskip\noindent
For a function $u\in C([0,T_+); X)$ we set  $u_\lambda(t)= u(\lambda t), \ t\in [0,T], $
and we define
\begin{equation*}
\begin{aligned}
&\Psi: (1-\varepsilon, 1+\varepsilon) \times \bE_{1,\mu}(T) \to  \bE_{0,\mu}(T) \times X_{\mu-1/p,p} , \\
&\Psi(\lambda, u) := \left( \partial_t u - \lambda \sA u -\lambda \sF(u)   ,  u(0)- v_0 \right).
\end{aligned}
\end{equation*}
Thanks to Lemma~\ref{lem:F-cont} one readily verifies that
$$
\Psi\in C^\infty \left( (1-\varepsilon, 1+\varepsilon) \times \bE_{1,\mu}(T) ,  \bE_{0,\mu}(T) \times X_{\mu-1/p,p} \right).
$$

\medskip\noindent
By uniqueness of solutions, we have $\Psi(\lambda, v_\lambda)=(0,0)$, where $v$ is the solution to \eqref{PE-abstract}, defined on $[0, T_+)$.
It follows from Lemma~\ref{lem:F-cont} that
$$
\partial_2 \Psi ( 1, v) \in \Lis (\bE_{1,\mu}(T), \bE_{0,\mu}(T) \times X_{\mu-1/p,p}).
$$
The implicit function theorem then implies that there exists some $\hat{\varepsilon} \in (0,\varepsilon)$ and a function
$$\psi\in C^1((1- \hat{\varepsilon} , 1+ \hat{\varepsilon}), \bE_{1,\mu}(T))$$
such that $\psi(\lambda)= v_\lambda$. Further, the differentiability of $\psi$ implies
$$
\Big(\frac{d}{d\lambda} \psi \Big|_{\lambda=1}\Big)^k = t^k \partial^k_t v \in \bE_{1,\mu}(T), \quad k\in \bN.
$$
Let $\delta\in (0,T)$ be given. It follows that $v\in H^k_p((\delta,T); X_1)$ for any $k\in\bN$. The last assertion in (c) is then a consequence of  Sobolev embedding and the fact
that $\delta\in (0,T)$ can be chosen arbitrarily.


\noindent
\medskip\noindent
(d) We rewrite \eqref{PE-abstract} pointwise for any $t\in (0,T)$ as
$$v=(\lambda+\sA)^{-1}(\sF(v)+\lambda v-\partial_t v),$$
where $\lambda>0$ is sufficiently large so that Theorem \ref{thm:highReg_hydrStokes} is available. By assertion (c) and by Lemma \ref{lem:bilinear} (a) this yields
\begin{equation}
\label{eq:spatial_reg}
\partial_t^nv=(\lambda+\sA)^{-1}(\partial_t^n\sF(v)+\lambda \partial_t^nv-\partial_t^{n+1} v),
\end{equation}
for any $n\in\bN$. Since, by assertion (c), $\partial_t^k v\in H_{q,\bar{\sigma}}^{2}(\sN;T\Sigma)$ for any $k\in\bN$, Lemma \ref{lem:bilinear} (a) with $s=2-(1+1/q)$ implies
$$\partial_t^n \sF(v)\in H_{q,\bar{\sigma}}^{2-(1+1/q)}(\sN;T\Sigma)=H_{q,\bar{\sigma}}^{1-1/q}(\sN;T\Sigma).$$
This yields
$$(\partial_t^n\sF(v)+\lambda \partial_t^nv-\partial_t^{n+1} v)\in H_{q,\bar{\sigma}}^{1-1/q}(\sN;T\Sigma),$$
hence
$$\partial_t^n v\in H_{q,\bar{\sigma}}^{3-1/q}(\sN;T\Sigma),$$
for any $n\in\bN$, by \eqref{eq:spatial_reg} and Theorem \ref{thm:highReg_hydrStokes}. This in turn implies
$$\partial_t^n \sF(v)\in H_{q,\bar{\sigma}}^{2-2/q}(\sN;T\Sigma),$$
by Lemma \ref{lem:bilinear} (a) with $s=2-2/q$, hence
$$\partial_t^n v\in H_{q,\bar{\sigma}}^{4-2/q}(\sN;T\Sigma),$$
for any $n\in\bN$, by \eqref{eq:spatial_reg} and Theorem \ref{thm:highReg_hydrStokes}. Iteratively, this yields
$$\partial_t^n v\in H_{q,\bar{\sigma}}^{k+2-k/q}(\sN;T\Sigma),$$
for all $k,n\in\bN$. For arbitrary $n,\ell\in\mathbb{N}$, it follows from Sobolev embeddings that the mapping
$$(0,T)\times\sN\ni(t,x)\mapsto (\nabla^{\sN})^\ell \partial_t^nv$$
is continuous, provided that $k+2-(k+3)/q>\ell$ or equivalently $k>\frac{1}{q-1}\left(q(\ell-2)+3\right)$. This readily implies $v\in C^\infty((0,T)\times\sN;T\Sigma)$.

\medskip\noindent
(e)
Assume {on the contrary that} $T_+(v_0)<\infty$.
Because of compactness of the embedding
$$
X_{\bar{\mu}-1/p,p} \hookrightarrow X_{\mu-1/p,p} ,
$$
the closure of $\{v(t)\}_{t\in [\delta,T_+(v_0))}$ in $X_{\mu-1/p,p}$, denoted by $\cV$, is compact in $X_{\mu-1/p,p}$.
It follows from parts (a) and  (b) that for any $z_0\in \cV$, there exists some $\rho(z_0)>0$ and $\tau(z_0)>0$ such that for every $u_0\in B_{X_{\mu-1/p,p}}(v_0, \rho (z_0))$,
the solution of \eqref{PE-abstract} with initial value  $u_0$ exists on $[0,\tau(z_0)]$. By  compactness of $\cV$, we can find finitely many $z^k_0\in \cV$, {$k=1,2,\ldots, N$}, such that
$$
\cV \subset \bigcup\limits_{k=1}^N  B_k, \quad B_k=B_{X_{\mu-1/p,p}}(z^k_0, \rho_k), \quad \rho_k=\rho(z^k_0),
$$
and the solution to \eqref{PE-abstract} with initial value $u_0\in B_k$ exists on $[0,\tau_k]$ with $\tau_k =\tau(z^k_0)$.

Let $\tau=\min\{\tau_k :\, k=1,\ldots , N \}$.
Pick $s \in [\delta, T_+(v_0))$ such that $T_+(v_0) - s< \tau$ and consider the problem

\begin{equation}
\label{PE-abstract-4}
\left\{\begin{aligned}
  \partial_t u + \sA u    &= \sF (u)       ,\\
u(0) & = v(s) .
\end{aligned}\right.
\end{equation}
There exists $k\in \{1,\ldots, N\}$ such that $v(s)\in B_k$. Therefore, the solution of \eqref{PE-abstract-4} exists at least on the interval $[0,\tau]$.
The uniqueness part in (a) shows that $u(t)= v(s+t)$ for $t\in [0,\tau]\cap[0, T_+(v_0)-s)$.
However, in view of the condition $T_+(v_0) - s< \tau $,  $v$ can be extended beyond $T_+(v_0)$. A contradiction.
Therefore, $v$ exists globally, i.e., $T_+(v_0)=\infty$.

\end{proof}

\begin{remark}
\label{rem: pressure}
The pressure $\pi_s$ in \eqref{PE} may be reproduced as follows. We first observe that if $v\in C^\infty((0,T_+)\times\sN;T\Sigma)$ is a solution of \eqref{PE-abstract}, then, in particular,
$$\partial_t v - \Ph ( \Delta_\sN  +	 \Ric  )v     = - \Ph \left( \nabla_{v} v    + w (v) \partial_r  v  \right),$$
by definition of $(\sA,\sF)$. With
$$G(v):=\nabla_{v} v    + w (v) \partial_r  v-(\Delta_\sN  +	 \Ric  )v\in C^\infty((0,T_+)\times\sN;T\Sigma)$$
it follows that 
$$0=\partial_t v+\Ph G(v)=\partial_tv+G(v)+(\Ph-I)G(v)=\partial_tv+G(v)+(\PH-I)\overline{G(v)}$$
by \eqref{I-Ph}. Then, by \eqref{def-PH} it holds that
$$(\PH-I)\overline{G(v)}=-\wgd\psi_{\overline{G(v)}},$$
where $\psi_{\overline{G(v)}}$ solves
$$\Delta_B\psi_{\overline{G(v)}}=\wdv \overline{G(v)}\in C^\infty((0,T_+)\times\Sigma),$$
since $\wdv\PH\overline{G(v)}=0$.
Setting $\pi_s:=-\psi_{\overline{G(v)}}$ shows that $(v,\pi_s)$ solves \eqref{PE}. In addition $\pi_s$ is smooth. Indeed, without loss of generality, we may assume that $\int_\Sigma\pi_s d\mu_g=0$, hence $\|\Delta_B\cdot\|_{\hat{H}_q^s(\Sigma)}$ is an equivalent norm in $\hat{H}_q^{2+s}(\Sigma)$, where
$$\hat{H}_q^s(\Sigma):=\left\{u\in H_q^s(\Sigma):\int_\Sigma u d\mu_g=0\right\},\quad q\in (1,\infty),\ s\ge 0.$$
Since
$$-\Delta_B\partial_t^n\pi_s=\wdv\partial_t^n \overline{G(v)}$$
and $\int_\Sigma\partial_t^n\pi_s d\mu_g=0$ for any $n\in\bN$, it follows readily that $\pi_s\in C^\infty((0,T_+)\times\Sigma)$.
\end{remark}

\section{Global Existence}\label{S: global existence}
\noindent
 In this section, we show that the primitive equations \eqref{PE} admit global strong solutions.
Our approach follows closely the strategy employed in \cite{HiKa16,GHK17, GGHHK20} for a Euclidean setting.
The more complicated geometry of $\Sigma$ considered in this manuscript introduces substantial additional difficulties into the analysis.

\medskip\noindent
We remind that  given $f\in C(\sN;T\Sigma)$,  $\tilde{f}=f- \overline{f}$.
Assume that $v$ solves \eqref{PE}. Then one shows that   $\overline{v}$ and $\tilde{v}$ satisfy
\begin{equation}
\label{PE-avg}
\begin{split}
  \partial_t \overline{v}
-  \big(\Delta_\Sigma   +  \Ric_\Sigma   \big)  \overline{v}   + \gd_\Sigma  \pi_s    = -\nabla_{\overline{v}} \overline{v} - \avint_{-h}^0 \left[\nabla_{\tilde{v}} \tilde{v}  + \tilde{v}\, \wdv \tilde{v}  \right]\, dr  - \frac{1}{h} \partial_r v |_{\Sigma_b}     ,
\end{split}
\end{equation}
and
\begin{equation}
\label{PE-oscillation}
\begin{split}
\partial_t \tilde{v} +     w \partial_r \tilde{ v} -  \big( \Delta_\sN    +  \Ric   \big) \tilde{v}    &+   \nabla_{\tilde{v}} \tilde{v}  +\nabla_{\overline{v}} \tilde{v}
 =     \avint_{-h}^0 \left[\nabla_{\tilde{v}} \tilde{v}  + \tilde{v} \,\wdv \tilde{v}  \right]\, dr  - \nabla_{\tilde{v}} \overline{v}     + \frac{1}{h} \partial_r v |_{\Sigma_b}     ,
\end{split}
\end{equation}
respectively.

\medskip\noindent
Throughout this section, $ B_j:  \overline{(\bR_+)^2} \to \bR_+$  denote continuous functions with the property that
$$ \lim\limits_{x,y\to 0^+} B_j(x,y)=0, \quad j\in \bN.$$

\subsection{A priori Estimates}\label{S:a priori est}

\begin{lemma}\label{lem:Stokes-H}
For $T\in  (0,\infty)$   and  $f\in L_2((0,T); H^1_2(\Sigma; T\Sigma))$,  let $(v,\pi)$ be a solution  of
\begin{equation}
\label{horizontal-eq}
\left\{\begin{aligned}
  \partial_t v  - \Delta_\Sigma    v  - \Ric  v  + \gd_\Sigma  \pi     &= f     &&\text{in}&&\Sigma ,\\
\dv_\Sigma v & = 0 &&\text{in}&&\Sigma.\\
\end{aligned}\right.
\end{equation}
Then there exists a positive constant $C_\Sigma$ such that
$$
   20\,  \partial_t \| \nabla v \|_{L_2(\Sigma)}^2 +   \| \Delta_\Sigma v \|_{L_2(\Sigma)}^2 +  \| \partial_t v \|_{L_2(\Sigma)}^2+  \| \wgd \pi \|_{L_2(\Sigma)}^2  \leq C_\Sigma  \left( \| f \|_{L_2(\Sigma)}^2 +   \| v  \|_{L_2(\Sigma)}^2 \right)
$$
for a.e. $t\in (0,T)$.
\end{lemma}
\begin{proof}
Multiplying \eqref{horizontal-eq}$_1$ by $\Delta_\Sigma v$ and employing ~\eqref{div-scalar-Sigma}, \eqref{Green-5} and \eqref{commutator-divergence},
we obtain
\begin{align*}
\frac{1}{2} \partial_t \| \nabla v \|_{L_2(\Sigma)}^2 + \| \Delta_\Sigma    v  \|_{L_2(\Sigma)}^2
& = -( \Ric v + f , \Delta_\Sigma v )_\Sigma  - \int_\Sigma  \pi \wdv \Delta_\Sigma v \, d\mu_g     \\
&=  -( \Ric v + f , \Delta_\Sigma v )_\Sigma - \int_\Sigma  \pi \left[ \Delta_B \wdv v + \wdv (\Ric  v) \right]\, d \mu_g   \\
&= -( \Ric v + f , \Delta_\Sigma v )_\Sigma +  \int_\Sigma (\wgd \pi) \Ric  v \, d \mu_g    \\
&\leq \frac{\varepsilon}{2} \| \wgd \pi \|_{L_2(\Sigma)}^2 + \frac{1}{2}\| \Delta_\Sigma    v  \|_{L_2(\Sigma)}^2  \\
&\qquad+ {\frac{C_\Sigma(\varepsilon)}{2}} \| v \|_{L_2(\Sigma)}^2 + {\frac{3}{4} }\| f \|_{L_2(\Sigma)}^2,
\end{align*}
where $\varepsilon>0$ is a number that will be determined later.  Absorbing the term $\frac{1}{2}\| \Delta_\Sigma    v  \|_{L_2(\Sigma)}^2$ and multiplying  by 2 yields
\begin{equation}
\label{eq:est1}
\partial_t \| \nabla v \|_{L_2(\Sigma)}^2 + \| \Delta_\Sigma    v  \|_{L_2(\Sigma)}^2\le \varepsilon \| \wgd \pi \|_{L_2(\Sigma)}^2 + C_\Sigma(\varepsilon) \| v \|_{L_2(\Sigma)}^2 + {\frac{3}{2} }\| f \|_{L_2(\Sigma)}^2.
\end{equation}
Next, we multiply \eqref{horizontal-eq}$_1$ by $v_t$ and integrate over $\Sigma$ to obtain
\goodbreak
\begin{align*}
   \| v_t \|_{L_2(\Sigma)}^2 + \frac{1}{2} \partial_t \| \nabla v \|_{L_2(\Sigma)}^2 & = ( f+\Ric v , v_t )_\Sigma \\
&\leq \frac{1}{2}  \| v_t \|_{L_2(\Sigma)}^2 +  {\frac{C_\Sigma (\varepsilon) } {2} } \| v \|_{L_2(\Sigma)}^2 + {\frac{3}{4} } \| f \|_{L_2(\Sigma)}^2,
\end{align*}
which implies
\begin{align}
\label{eq:est2}
   \| v_t \|_{L_2(\Sigma)}^2 +\partial_t \| \nabla v \|_{L_2(\Sigma)}^2  \leq    C_\Sigma (\varepsilon) \| v \|_{L_2(\Sigma)}^2 
   + {\frac{3}{2} } \| f \|_{L_2(\Sigma)}^2 .
\end{align}
Adding \eqref{eq:est1} and \eqref{eq:est2}, we arrive at
\begin{equation}
\label{est-v-avg-int}
 2\,\partial_t \| \nabla v \|_{L_2(\Sigma)}^2 + \| \Delta_\Sigma    v  \|_{L_2(\Sigma)}^2  +  \| v_t \|_{L_2(\Sigma)}^2   \leq  \varepsilon \| \wgd \pi \|_{L_2(\Sigma)}^2  +  C_\Sigma(\varepsilon) \| v \|_{L_2(\Sigma)}^2 + 3 \| f \|_{L_2(\Sigma)}^2 .
\end{equation}
We can use \eqref{horizontal-eq}$_1$  to estimate $\wgd \pi$ as follows:
\begin{align*}
 \| \wgd \pi \|_{L_2(\Sigma)}^2  \leq 4\left(\| v_t \|_{L_2(\Sigma)}^2 + \| \Delta_\Sigma    v  \|_{L_2(\Sigma)}^2 + C_\Sigma\| v \|_{L_2(\Sigma)}^2 +   \| f \|_{L_2(\Sigma)}^2\right).
\end{align*}
Combining the last estimate with \eqref{est-v-avg-int} yields
\begin{equation*}
\begin{aligned}
 10\,\partial_t \| \nabla v \|_{L_2(\Sigma)}^2 + \| \Delta_\Sigma    v  \|_{L_2(\Sigma)}^2  & +  \| v_t \|_{L_2(\Sigma)}^2  +(1-5\varepsilon)\| \wgd \pi \|_{L_2(\Sigma)}^2 \\
 & \le  C_\Sigma (\varepsilon)( \| v \|_{L_2(\Sigma)}^2 +  \| f \|_{L_2(\Sigma)}^2) .
\end{aligned}
\end{equation*}
Choosing  $5\varepsilon=1/2$ yields the assertion.
\end{proof}


\begin{lemma}\label{lem:basic-L2-est}
 Suppose $v_0\in H^1_{2, \bar{\sigma}} (\sN;T\Sigma)$ satisfies $v_0=0$  on $\Sigma_b$.
Let $v\in \bE_{1,1}(T)$ be the solution to \eqref{PE} with $p=q=2$ asserted by Proposition~\ref{prop:local existence},
defined on an interval $[0,T]$.  Then
$$
\| v (t) \|_{L_2(\sN)}^2 + 2\int_0^t   \|  \nabla^\sN v (\tau) \|_{L_2(\sN)}^2 \, d\tau \leq B_0( \|v_0 \|_{L_2(\sN)} , T),
$$
for each $ t\in [0,T]$.
\end{lemma}
\begin{proof}
Multiplying both sides of \eqref{PE}$_1$ by $v$ and using \eqref{div-scalar-Sigma},  \eqref{Green-2},
the fact that $\pi_s$ does not depend on $r$, and $\eqref{PE}_2$  yields
\begin{align*}
{ \frac{1}{2} \partial_t \| v\|_{L_2(\sN)}^2 + \|  \nabla^\sN  v \|_{L_2(\sN)}^2 } & =  ( \Ric \, v, v)_\sN +  \int_{\Sigma} \pi_s \int_{-h}^0 \wdv v  \, dr  \, d\mu_g \\
& =  ( \Ric \, v, v)_\sN \ \leq C  \| v\|_{L_2(\sN)}^2 .
\end{align*}
Here we also utilized Proposition~\ref{pro: divergence-free} with $z=v$, $q=2$.
Then Gronwall's inequality implies
\begin{align*}
 \| v (t) \|_{L_2(\sN)}^2 + 2\int_0^t \|  \nabla^\sN  v (\tau) \|_{L_2(\sN)}^2 \, d\tau &\leq    e^{CT} \|v_0 \|_{L_2(\sN)}^2   ,
\end{align*}
for each $t\in [0,T]$. This completes the proof.
\end{proof}

}Before stating the next result, which addresses the a priori estimates for the solutions of \eqref{PE}, let us include the following remark. If $v_0\in  H^1_{2, \bar{\sigma}} (\sN;T\Sigma)$ satisfies $v_0=0$ on $\Sigma_b$, then, by Proposition \ref{prop:local existence} with $p=q=2$, the unique solution $v$ to \eqref{PE} satisfies $v \in C^\infty((0,T)\times\sN;T\Sigma)$ for each $T\in (0,T_+(v_0))$. In particular, for any $\tau\in (0,T)$, it holds that $v(\tau)\in H_{2,\bar{\sigma}}^{2}(\sN;T\Sigma)$. By applying the time-shift $t\mapsto v(t+\tau)$, we may therefore always assume without loss of generality that $v_0\in H_{2,\bar{\sigma}}^2(\sN;T\Sigma)$ with the compatibility conditions $v_0=0$ on $\Sigma_b$ and $\partial_r v_0=0$ on $\Sigma_u$. Moreover, $v$ and $\pi_s$ in \eqref{PE} may be assumed to be as smooth as desired on $[0,T)\times \sN$ and $[0,T)\times\Sigma$, respectively, including the initial time $t=0$.

\begin{theorem}
\label{thm: a priori est}
Suppose  $v_0\in  H^2_{2, \bar{\sigma}} (\sN;T\Sigma)$ satisfies  $v_0=0$  on $\Sigma_b$ and $\partial_r v_0=0$ on $\Sigma_u$. 
Then the unique solution $v$ to \eqref{PE} with $p=q=2$ asserted by Proposition~\ref{prop:local existence}, satisfies
$$
\|v\|_{  L_\infty((0,T);H^2_2(\sN))} \leq B_1 ( \| v_0\|_{H^2_2(\sN)},   T).
$$
\end{theorem}
\begin{proof}
The proof follows the approach of  \cite{GHK17, GGHHK20,  HiKa16} and is divided into eight steps.

\medskip\noindent
{\bf Step 1:} $L_\infty((0,T); L_2(\Sigma))$-estimate for $\nabla \overline{v}:$

\medskip\noindent
We apply  Lemma~\ref{lem:Stokes-H} to \eqref{PE-avg} and use Proposition~\ref{Prop:trace} {as well as Young's inequality} to obtain
\begin{align*}
&\quad  20  \partial_t \| \nabla \overline{v} \|_{L_2(\Sigma)}^2 +   \| \Delta_\Sigma \overline{v} \|_{L_2(\Sigma)}^2 +  \| \wgd \pi_s \|_{L_2(\Sigma)}^2  \\
&\leq   C_\Sigma  \left( \| |\overline{v}|_g |\nabla \overline{v} |_g  \|_{L_2(\Sigma)}^2 +  \| |\tilde{v}|_g |\nabla^\sN\tilde{v} |_g  \|_{L_2(\sN)}^2  +     \| \overline{v} \|_{L_2(\Sigma)}^2 + \| v_r   \|_{L_2(\Sigma_b)}^2   \right) \\
&\leq C_\Sigma  \left( \| |\overline{v}|_g |\nabla \overline{v} |_g  \|_{L_2(\Sigma)}^2 +  \| |\tilde{v}|_g |\nabla^\sN\tilde{v} |_g  \|_{L_2(\sN)}^2  + \| \overline{v} \|_{L_2(\Sigma)}^2
 + \| v_r   \|_{L_2(\sN)}^2  \right)   +  \frac{1}{4}\|  \nabla^\sN  v_r   \|_{L_2(\sN)}^2 .
\end{align*}
The first   term  on the RHS can be estimated as follows:
\goodbreak
\begin{align*}
& \| |\overline{v}|_g |\nabla \overline{v} |_g  \|_{L_2(\Sigma)}^2 \\
& \leq \|\overline{v} \|_{L_4(\Sigma)}^2 \| \nabla  \overline{v} \|_{L_4(\Sigma)}^2 \\
&\leq C \left( \|\overline{v} \|_{L_2(\Sigma)}^2 + \|\overline{v} \|_{L_2(\Sigma)} \| \nabla  \overline{v} \|_{L_2(\Sigma)}  \right)
\left( \| \nabla\overline{v} \|_{L_2(\Sigma)}^2 + \| \nabla \overline{v} \|_{L_2(\Sigma)} \| \nabla^2  \overline{v} \|_{L_2(\Sigma)}  \right) \\
&\leq C \left( \|\overline{v} \|_{L_2(\Sigma)}^2 + \|\overline{v} \|_{L_2(\Sigma)} \| \nabla  \overline{v} \|_{L_2(\Sigma)}  \right)
\left( \| \nabla\overline{v} \|_{L_2(\Sigma)}^2 + \| \nabla \overline{v} \|_{L_2(\Sigma)} ( \|\overline{v} \|_{L_2(\Sigma)}+ \| \Delta_\Sigma   \overline{v} \|_{L_2(\Sigma)}  ) \right)\\
& \leq C \|\overline{v} \|_{L_2(\Sigma)}^2 \| \nabla\overline{v} \|_{L_2(\Sigma)}^2 + C \|\overline{v} \|_{L_2(\Sigma)}^3 \| \nabla \overline{v} \|_{L_2(\Sigma)} + C \|\overline{v} \|_{L_2(\Sigma)} \| \nabla \overline{v} \|_{L_2(\Sigma)}^3 \\
&\qquad + C \|\overline{v} \|_{L_2(\Sigma)}^2 \| \nabla \overline{v} \|_{L_2(\Sigma)} \| \Delta_\Sigma   \overline{v} \|_{L_2(\Sigma)}
+ C \|\overline{v} \|_{L_2(\Sigma)}  \| \nabla \overline{v} \|_{L_2(\Sigma)}^2 \| \Delta_\Sigma   \overline{v} \|_{L_2(\Sigma)}  \\
&\leq  C \left(1 +   \|   v \|_{L_2(\sN)}^4 +  \|   \nabla^\sN  v \|_{L_2(\sN)}^2 + \|    v \|_{L_2(\sN)}^2\|    \nabla^\sN  v \|_{L_2(\sN)}^2 \right) \| \nabla\overline{v}  \|_{L_2(\Sigma)}^2  \\
&\qquad \qquad + C \|  v \|_{L_2(\sN)}^6   + \frac{1}{C_\Sigma} \| \Delta_\Sigma  \overline{v} \|_{L_2(\Sigma)}^2
\end{align*}
where we have used \eqref{bdd-avg}, \eqref{bdd-avg-gd}, and \eqref{G-N ineq}$_4$.
In addition, we used  elliptic regularity theory to obtain the estimate
\begin{equation}
\label{elliptic-estimate}
\| \nabla^2 \bar{v}\|_{L_2(\Sigma)} \leq C \| \bar{v}\|_{H^2(\Sigma)} \leq  C \left( \| \bar{v} \|_{L_2(\Sigma)} + \| \Delta_\Sigma  \bar{v} \|_{L_2(\Sigma)} \right),
\end{equation}
 see  for instance \cite[Theorem 1.30]{Ama16}.
We thus infer 
\begin{equation}
\label{est-v-avg-H1}
\begin{split}
&  20  \partial_t \| \nabla \overline{v} \|_{L_2(\Sigma)}^2   +  \| \wgd \pi_s \|_{L_2(\Sigma)}^2  \\
&\leq C  \left(1 +   \|   v \|_{L_2(\sN)}^4 +  \|    \nabla^\sN  v \|_{L_2(\sN)}^2 + \|    v \|_{L_2(\sN)}^2\|    \nabla^\sN  v \|_{L_2(\sN)}^2 \right) \| \nabla\overline{v}  \|_{L_2(\Sigma)}^2  \\
&\quad +C \left( \| v_r   \|_{L_2(\sN)}^2 + \|  v \|_{L_2(\sN)}^2  + \|  v \|_{L_2(\sN)}^6   \right)  + C_\Sigma   \| |\tilde{v}|_g |\nabla \tilde{v} |_g  \|_{L_2(\sN)}^2   +    \frac{1}{4}\|   \nabla^\sN  v_r   \|_{L_2(\sN)}^2 .
\end{split}
\end{equation}

\medskip\noindent
{\bf Step 2:} $L_\infty((0,T); L_2(\sN))$-estimate for $\partial_r v$. For this, we test \eqref{PE} by $-\partial_r^2 v$.

\medskip\noindent
We first consider the term $(\Delta_{\sN}v  , \partial^2_r v)_\sN$.
Employing Proposition \ref{prop:local existence} (d), \eqref{Green-1} and integrating by parts with respect to $r$ yields
\begin{equation*}
\begin{aligned}
( \Delta_\sN v , \partial^2_r v)_\sN
& = -(  \nabla^\sN  v,   \nabla^\sN  \partial^2_r v)_\sN + ( \nabla^\sN_{\nu}  v, \partial^2_r v)_{\pa\sN} \\
& =(  \nabla^\sN  \partial_r v ,   \nabla^\sN  \partial_r v)_\sN
-\int_\Sigma \Big[ \la   \nabla^\sN  v ,   \nabla^\sN  \pa_r v\ra_{g_\sN}-  \la  \partial_r  v, \partial^2_r v\ra_g \Big]^{r=0}_{r=-h} \\
& = (  \nabla^\sN  v_r ,   \nabla^\sN  v_r)_\sN.
\end{aligned}
\end{equation*}
Here we used the identity
\begin{align*}
\la  \nabla^\sN v,  \nabla^\sN  \partial_r v\rangle_g - \la \partial_r v, \partial^2_r v \ra_g
 = \la  \nabla v, \nabla \partial_r v\ra_g ,
\end{align*}
see \eqref{splitting-2},
and the boundary conditions  imposed on $v$ to deduce that
$$
\int_\Sigma \Big[ \la   \nabla^\sN  v ,   \nabla^\sN  \pa_r v\ra_{g}-  \la  \partial_r  v, \partial^2_r v\ra_g \Big]^{r=0}_{r=-h} \,d\mu_g=0.
$$
We now turn our attention to the term  $ -(\nabla_v v+ w v_r , \partial_r  v_r )_\sN$.

Integrating by parts with respect to $r$ and using the boundary conditions for $v$ and $w$,
it follows from Proposition~\ref{pro: divergence-free} with $z= v_r$ and $q=2,$
and from $\eqref{PE}_3$  that
\begin{align*}
 -(\nabla_v v+ w v_r , \partial_r  v_r)_\sN
&= - \int_\Sigma \Big[ \la  \nabla_v v + w v_r,   v_r\ra_{g_\sN} \Big]^{r=0}_{r=-h }  \,d\mu_g\\
& \quad + (\nabla_v v_r + w \pa_r v_r ,    v_r )_\sN  +  (\nabla _{v_r} v + (\pa_r w) v_r ,    v_r )_\sN  \\
& =  (\nabla _{v_r} v ,    v_r )_\sN  -   \int_{\sN}   (\wdv v) |v_r|^2_g    \, d\mu_{g_\sN}.
\end{align*}
Taking into account that  $\pi_s$ and $\Ric$ do not depend on $r$ we obtain
$$
(\gd_\Sigma \pi_s, \partial^2_r v)_\sN = \int_\Sigma \Big[\la \gd_\Sigma  \pi_s, \partial_r v \ra_g \Big]^{r=0}_{r=-h}\, d\mu_g
= -\int_\Sigma \la \gd_\Sigma  \pi_s, \partial_r v \ra_g \Big|_{r=-h}  \,d\mu_g
$$
and
$$
(\Ric v, \partial^2_r v)_{\sN}
= \int_{\Sigma} \Big[ \la \Ric v, \partial_r v \ra_g \Big]^{r=0}_{r=-h}\, d\mu_g - (\Ric \partial_r v, \partial_r v)_\sN,
= - (\Ric \partial_r v, \partial_r v)_\sN
$$
where we again used the boundary conditions for $v$.
In summary,
\begin{align*}
&\frac{1}{2} \partial_t \| v_r   \|_{L_2(\sN)}^2 + \|   \nabla^\sN  v_r \|_{L_2(\sN)}^2 \\
& = -\int_{\Sigma_b} \la \wgd \pi_s , v_r\ra_g \, d\mu_g  +   \int_{\sN} (\wdv v) |v_r|_g^2\, d\mu_{g_\sN}   -   \left(\nabla_ {v_r} \overline{v} , v_r \right)_\sN  -   \left(\nabla_ {v_r} \tilde{v} , v_r \right)_\sN
+  (\Ric  v_r , v_r )_\sN  \\
& =:{ I_1+\cdots +I_5}.
\end{align*}
\medskip\noindent
The six terms on the RHS can be estimated as follows.
An application of Proposition~\ref{Prop:trace} yields
\begin{align*}
I_1 &\leq \| \wgd \pi_s \|_{L_2(\Sigma_b)} \|v_r   \|_{L_2(\sN)}^{1/2}  \|   \nabla^\sN  v_r \|_{L_2(\sN)}^{1/2}  \\
&\leq \frac{1}{4} \| \wgd \pi_s \|_{L_2(\Sigma)}^2 + \frac{1}{8} \|   \nabla^\sN  v_r \|_{L_2(\sN)}^2 + C  \|v_r \|_{L_2(\sN)}^2,	
\end{align*}
{since $\| \wgd \pi_s \|_{L_2(\Sigma_b)}=\| \wgd \pi_s \|_{L_2(\Sigma)}$.}

\medskip\noindent
It follows from \eqref{bdd-avg-gd} and \eqref{G-N ineq}$_4$ that
\begin{align*}
I_3  & \leq \int_{\Sigma} | \nabla \overline{v}|_g \int_{-h}^0 | v_r|_g^2\, dr \, d\mu_g \leq \| \nabla \overline{v} \|_{L_2(\Sigma)} \int_{-h}^0 \left\|   v_r  \right\|_{L_4(\Sigma)}^2 \, dr  \\
&\leq C \| \nabla \overline{v} \|_{L_2(\Sigma)} \int_{-h}^0 \left(  \| v_r\|_{L_2(\Sigma)} \|\nabla v_r\|_{L_2(\Sigma)} +   \| v_r\|_{L_2(\Sigma)}^2    \right) \, dr   \\
&\leq C \left(\|   \nabla^\sN   v \|_{L_2(\sN)}  +  \|   \nabla^\sN   v \|_{L_2(\sN)}^2  \right) \| v_r\|_{L_2(\sN)}^2     +\frac{1}{8} \|  \nabla^\sN  v_r\|_{L_2(\sN)}^2.
\end{align*}
In view of \eqref{PE}$_2$, we have { $\wdv {v}=\wdv \tilde{v}$, hence}
\begin{align*}
I_2 &= \int_{\sN} \wdv \tilde{v} |v_r|_g^2\, d\mu_{g_\sN}  = - 2 \left( \nabla_{\tilde{v}} v_r , v_r  \right)_\sN \leq C \int_{\sN} |\tilde{v}|_g | \tilde{v}_r|_g | \nabla v_r|_g \, d\mu_{g_\sN} \\
&\leq {\frac{C_1}{4}} \left\|  |\tilde{v}|_g | \tilde{v}_r|_g\right\|_{L_2(\sN)}^2  + {\frac{1}{8}} \| \nabla^\sN  v_r\|_{L_2(\sN)}^2
\end{align*}
for some constant $C_1$,
and similarly
\begin{align*}
I_4 & = \int_{\sN} \nabla_{v_r} \la \tilde{v} , v_r \ra_g \, d\mu_{g_\sN} - \int_{\sN}  \la \tilde{v} , \nabla_{v_r} v_r \ra_g \, d\mu_{g_\sN} \\
&=  - \int_{\sN} \wdv v_r \la \tilde{v} , v_r \ra_g \, d\mu_{g_\sN} - \int_{\sN}  \la \tilde{v} , \nabla_{v_r} v_r \ra_g \, d\mu_{g_\sN} \\
& \leq  2 \int_{\sN} | \nabla v_r|_g |  \tilde{v}|_g | v_r|_g \, d\mu_{g_\sN}  \\
&\leq \frac{C_1}{4} \left\|  |\tilde{v}|_g | \tilde{v}_r|_g\right\|_{L_2(\sN)}^2  + {\frac{1}{8} } \| \nabla^\sN  v_r\|_{L_2(\sN)}^2.
\end{align*}
Moreover,
$ I_5\leq C  \|v_r \|_{L_2(\sN)}^2. $ 
Combining the estimates for $I_1$-$I_5$, we arrive at
\begin{equation}
\label{est-vr-L2}
\begin{split}
& \partial_t \| v_r  \|_{L_2(\sN)}^2 + \|   \nabla^\sN  v_r \|_{L_2(\sN)}^2 \\
& \leq C_1\left\|  |\tilde{v}|_g | \tilde{v}_r|_g\right\|_{L_2(\sN)}^2 +  C \left(1 +  \|   \nabla^\sN   v \|_{L_2(\sN)}^2  \right) \|v_r \|_{L_2(\sN)}^2 + \frac{1}{2} \| \wgd \pi_s \|_{L_2(\Sigma)}^2
\end{split}
\end{equation}
 where we also used the estimate $(\|   \nabla^\sN  v \|_{L_2(N)} + \|   \nabla^\sN  v \|^2_{L_2(N)}) \le  C (1 +  \|   \nabla^\sN   v \|_{L_2(\sN)}^2  ). $

\medskip\noindent
{\bf Step 3:} $L_\infty((0,T); L_4(\sN))$-estimate for $\tilde{v}$. We test  \eqref{PE-oscillation} by $|\tilde{v}|_g^2 \tilde{v}$.

\medskip\noindent
By virtue of \eqref{Green-2}, \eqref{computation-2} and the boundary condition $\partial_r v=\partial_r \tilde v=0$ on $\Sigma_u$ we obtain
\begin{align*}
-(\Delta_\sN  \tilde v, |\tilde v|^2_g \tilde v)_\sN
& =  (  \nabla^\sN  \tilde v,   \nabla^\sN  ( | \tilde v|^2_g \tilde v))_\sN
- \int_\Sigma \Big[\la \partial_r  \tilde v,  |\tilde v|^2_g \tilde v\ra_g \Big]^{r=0}_{r=-h}\,d\mu_g \\
& = \left\| | \tilde{v}|_g |   \nabla^\sN  \tilde{v}|_g \right\|_{L_2(\sN)}^2 +  \frac{1}{2} \left\|\gd_\sN |\tilde{v}|^2_g \right\|_{L_2(\sN)}^2
+ \int_\Sigma \la \partial_r  \tilde v,  |\tilde v|^2_g \tilde v\ra_g \Big|_{r=-h}\,d\mu_g.
\end{align*}
By Proposition~\ref{pro: divergence-free} with $z=\tilde v$ and $q=4$ we have
\begin{align*}
( \nabla_{\tilde{v}} \tilde{v}  +\nabla_{\overline{v}} \tilde{v} + w \partial_r \tilde{ v}  , |\tilde v|^2_g \tilde v)_\sN
= ( \nabla_v \tilde{v}  + w \partial_r \tilde{ v}   , |\tilde v|^2_g \tilde v)_\sN =0.
\end{align*}
Combining, this results in
\begin{align*}
 \frac{1}{4}\partial_t& \| \tilde{v} \|_{L_4(\sN)}^4 + \frac{1}{2} \left\| \gd_\sN |\tilde{v}|^2_g \right\|_{L_2(\sN)}^2 + \left\| | \tilde{v}|_g |   \nabla^\sN  \tilde{v}|_g \right\|_{L_2(\sN)}^2 \\
&=
  \left(\avint_{-h}^0 \left[\nabla_{\tilde{v}} \tilde{v}  + (\wdv \tilde{v})  \tilde{v}  \right]\, dr  ,  |\tilde{v}|_g^2 \tilde{v} \right)_{\sN}
   + \frac{1}{h}  \left( \partial_r v |_{\Sigma_b} ,  |\tilde{v}|_g^2 \tilde{v}  \right)_\sN  \\
 &\qquad  \qquad  + \int_\Sigma \la \partial_r  \tilde v,  |\tilde v|^2_g \tilde v\ra_g \Big|_{r=-h}\,d\mu_g  -  \left(   \nabla_{\tilde{v}} \overline{v}   , |\tilde{v}|_g^2 \tilde{v} \right)_\sN  +  \left( \Ric \tilde{v} , |\tilde{v}|_g^2 \tilde{v} \right)_\sN \\
&  = : I \! I_1 + \cdots +  I \! I_5 .
\end{align*}
We will estimate the terms on the  RHS one by one.
First, we can apply   Lemma~\ref{lem:interpolation} to derive
\begin{align*}
II_1\, & \leq C \int_{\Sigma} \left(  \int_{-h}^0 |\tilde{v}|_g | \nabla^\sN \tilde{v}|_g \, dr  \right)\left(  \int_{-h}^0 |\tilde{v}|_g^3 \, dr  \right)\, d\mu_g \\
&\leq C \left(  \int_{-h}^0 \left\| |\tilde{v}|_g | \nabla^\sN \tilde{v}|_g \right\|_{L_{4/3}(\Sigma)} \, dr  \right)\left(  \int_{-h}^0 \left\| |\tilde{v}|_g^3 \right\|_{L_4(\Sigma)} \, dr  \right) \\
&\leq C \left(  \int_{-h}^0 \left\|  \tilde{v} \right\|_{L_4(\Sigma)} \left\| \nabla^\sN \tilde{v}  \right\|_{L_2(\Sigma)} \, dr  \right)
\left(  \int_{-h}^0 \left(  \left\|  \tilde{v}  \right\|_{L_4(\Sigma)}^3  + \left\|  \tilde{v}  \right\|_{L_4(\Sigma)} \left\| \wgd |\tilde{v}|_g^2 \right\|_{L_2(\Sigma)} \right) \, dr  \right) \\
&\leq C     \left\|  \tilde{v} \right\|_{L_4(\sN)} \left\| \nabla \tilde{v}  \right\|_{L_2(\sN)}
\left(    \left\|  \tilde{v}  \right\|_{L_4(\sN)}^3  + \left\|  \tilde{v}  \right\|_{L_4(\sN)} \left\|  \gd   |\tilde{v}|_g^2 \right\|_{L_2(\sN)} \right)    \\
&\leq C  \left(  \|   \nabla^\sN  v \|_{L_2(\sN)}  + \|   \nabla^\sN  v \|_{L_2(\sN)}^2 \right) \left\|  \tilde{v}  \right\|_{L_4(\sN)}^4   +  \frac{1}{8}\left\| \gd |\tilde{v}|_g^2 \right\|_{L_2(\sN)}^2.
\end{align*}
Here we have used the fact that $\|\nabla \tilde{v}\|_{L_2(\sN)} \leq \| \nabla v \|_{L_2(\sN)}  + \|\nabla \overline{v} \|_{L_2(\sN)} \leq 2\| \nabla v \|_{L_2(\sN)} $ due to   \eqref{bdd-avg-gd}.

\medskip\noindent
Next, it follows from Lemma~\ref{lem:interpolation} and Proposition~\ref{Prop:trace}  that
\begin{align*}
II_2\, &   \leq C \int_\Sigma \left|  \partial_r v |_{\Sigma_b} \right|_g \int_{-h}^0  | \tilde{v}|_g^3\, dr  \, d\mu_g \leq C \| v_r \|_{L_2(\Sigma_b)}\int_{-h}^0  \left\| | \tilde{v}|_g^3 \right\|_{L_2(\Sigma)}\, dr   \\
&\leq C \| v_r \|_{L_2(\Sigma_b)} \int_{-h}^0 \left( \| \tilde{v}\|_{L_4(\Sigma)}^3 + \| \tilde{v}\|_{L_4(\Sigma)} \left\| \wgd |\tilde{v}|_g^2 \right\|_{L_2(\Sigma)}  \right) \, dr  \\
&\leq C \| v_r\|_{L_2(\sN)}^{1/2}\|   \nabla^\sN  v_r\|_{L_2(\sN)}^{1/2} \| \tilde{v}\|_{L_4(\sN)}^3
+ C \| v_r\|_{L_2(\sN)}^{1/2}\|   \nabla^\sN  v_r\|_{L_2(\sN)}^{1/2}\| \tilde{v}\|_{L_4(\sN)} \left\| \gd |\tilde{v}|_g^2 \right\|_{L_2(\sN)} \\
&\leq   \frac{1}{8}\left\| \gd |\tilde{v}|_g^2 \right\|_{L_2(\sN)}^2 + \frac{1}{8 C_2}\|   \nabla^\sN  v_r \|_{L_2(\sN)}^2 + C \left( \| v_r\|_{L_2(\sN)}^{2/3} + \| v_r\|_{L_2(\sN)}^2 \right) \| \tilde{v}\|_{L_4(\sN)}^4,
\end{align*}
where {$C_2= 2(C_\Sigma +C_1)$}.
By the relation $\tilde{v}=- \bar{v}$ on $\Sigma_b$, for $p=\frac{2}{1-\theta}$ and $q=\frac{4}{1+2\theta}$ with $\theta>0$ sufficiently small we obtain
\begin{equation}
\begin{split}\label{est-additional bd}
II_3 & =  \int_\Sigma \la \partial_r  v,  |\tilde v|^2_g \bar v\ra_g \Big|_{r=-h}\,d\mu_g \le \| \partial_r v \|_{L_4(\Sigma_b)} \left\| |\tilde v|^2_g \right\|_{L_p(\Sigma_b)} \| \bar{v} \|_{L_q(\Sigma)} \\
&\le C \|   v_r \|_{H^1_2(\sN)} \left\| |\tilde v|^2_g \right\|_{H^{\frac{1}{2}+\theta}_2(\sN)} \left( \|\nabla \bar{v}\|_{L_2(\Sigma)}^{\frac{1-2\theta}{2}} \|  \bar{v}\|_{L_2(\Sigma)}^{\frac{1+2\theta}{2}}  + \|  \bar{v}\|_{L_2(\Sigma)} \right) \\
& \leq C \| \nabla^{\sN} v_r \|_{L_2(\sN)} \left\| |\tilde v|^2_g \right\|_{L_2(\sN)}^{\frac{1}{2}-\theta} \left(  \left\|  |\tilde v|^2_g \right\|_{L_2(\sN)}^{\frac{1}{2}+\theta} + \left\| \gd |\tilde v|^2_g \right\|_{L_2(\sN)}^{\frac{1}{2}+\theta} \right)\\
\quad & \times
\left( \|\nabla^{\sN} v\|_{L_2(\sN)}^{\frac{1-2\theta}{2}} \|  v\|_{L_2(\sN)}^{\frac{1+2\theta}{2}}  + \|  v\|_{L_2(\sN)} \right)  \\
& \leq C \| \nabla^{\sN} v_r \|_{L_2(\sN)} \left\| |\tilde v|^2_g \right\|_{L_2(\sN)} \left( \|\nabla^{\sN} v\|_{L_2(\sN)}^{\frac{1-2\theta}{2}} \|  v\|_{L_2(\sN)}^{\frac{1+2\theta}{2}}  + \|  v\|_{L_2(\sN)} \right)  \\
 &\quad+ C  \| \nabla^{\sN} v_r \|_{L_2(\sN)} \left\| |\tilde v|^2_g \right\|_{L_2(\sN)}^{\frac{1}{2}-\theta}   \left\| \gd |\tilde v|^2_g \right\|_{L_2(\sN)}^{\frac{1}{2}+\theta}
\left( \|\nabla^{\sN} v\|_{L_2(\sN)}^{\frac{1-2\theta}{2}} \|  v\|_{L_2(\sN)}^{\frac{1+2\theta}{2}}  + \|  v\|_{L_2(\sN)} \right)  \\
&\le \frac{1}{8 C_2} \| \nabla^{\sN} v_r \|_{L_2(\sN)}^2 + \frac{1}{8}\left\|\gd |\tilde{v}|^2_g \right\|_{L_2(\sN)}^2  \\
& \quad + C  \left\|  \tilde v  \right\|_{L_4(\sN)}^4   \left( \|\nabla^{\sN} v\|_{L_2(\sN)}^2 \|  v\|_{L_2(\sN)}^{\frac{2+4\theta}{1-2\theta}}  + \|  v\|_{L_2(\sN)}^{\frac{4}{1-2\theta}}  + \|\nabla^{\sN} v\|_{L_2(\sN)}^{ 1-2\theta } \|  v\|_{L_2(\sN)}^{ 1+2\theta }  + \|  v\|_{L_2(\sN)}^2 \right).
\end{split}
\end{equation}
In \eqref{est-additional bd}$_2$, we have used the Gagliardo-Nirenberg   inequality
\[
\|u\|_{L_q(\Sigma)} \leq C \| \nabla u\|_{L_2(\Sigma)}^{1-\frac{2}{q}} \|   u\|_{L_2(\Sigma)}^{ \frac{2}{q}} + \|   u\|_{L_2(\Sigma)}.
\]
and the embeddings
\[
H^1_2(\sN) \hookrightarrow L_4(\Sigma_b), \qquad
H^{\frac{1}{2}+\theta}_2(\sN)  \hookrightarrow L_p(\Sigma_b),
\]
which follows from the trace theorem and Sobolev embedding, , cf. \cite[Theorems 10.1 and 14.2]{Ama13}.
In \eqref{est-additional bd}$_3$, we applied   interpolation theory, see  \cite[Theorem~13.1]{Ama13}.
\\
Similarly, {by Lemma~\ref{lem:interpolation}},
\begin{align*}
II_4 & \leq  \int_\Sigma | \nabla \overline{v} |_g \int_{-h}^0 |\tilde{v}|_g^4 \, dr \, d\mu_g     \leq C \|\nabla \overline{v}\|_{L_2(\Sigma ) }  \int_{-h}^0 \left\| |\tilde{v}|_g^4 \right\|_{L_2(\Sigma)} \, dr  \\
&\leq C \|\nabla \overline{v}\|_{L_2(\Sigma ) }  \int_{-h}^0 \left(  \left\|  \tilde{v}  \right\|_{L_4(\Sigma)}^4 + \left\|  \tilde{v}  \right\|_{L_4(\Sigma)}^2  \left\| \wgd | \tilde{v} |_g^2 \right\|_{L_2(\Sigma)}   \right) \, dr  \\
&\leq C \left( \|\nabla \overline{v}\|_{L_2(\Sigma ) } +  \|\nabla \overline{v}\|_{L_2(\Sigma ) }^2 \right) \| \tilde{v} \|_{L_4(\sN)}^4 +  \frac{1}{8}\left\|\gd |\tilde{v}|^2_g \right\|_{L_2(\sN)}^2 \\
&\leq C  \left(  \|   \nabla^\sN  v \|_{L_2(\sN)}  + \|   \nabla^\sN  v \|_{L_2(\sN)}^2 \right) \| \tilde{v} \|_{L_4(\sN)}^4 +  \frac{1}{8}\left\|\gd |\tilde{v}|^2_g \right\|_{L_2(\sN)}^2 ,
\end{align*}
where we have used \eqref{bdd-avg-gd}.

Direct computations show that $II_5 \leq C \| \tilde{v} \|_{L_4(\sN)}^4. $
Combining the above estimates for $II_1$-$II_5$ yields
\begin{equation}
\label{est-v-osc-L4}
\begin{split}
&\frac{1}{4}\partial_t \| \tilde{v} \|_{L_4(\sN)}^4  + \left\| | \tilde{v}|_g |   \nabla^\sN  \tilde{v}|_g \right\|_{L_2(\sN)}^2 \\
& \leq C  \Big(  1  + \|   \nabla^\sN  v \|_{L_2(\sN)}^2 +  \|\nabla^{\sN} v\|_{L_2(\sN)}^2 \|  v\|_{L_2(\sN)}^{\frac{2+4\theta}{1-2\theta}}  + \|  v\|_{L_2(\sN)}^{\frac{4}{1-2\theta}}   \Big) \| \tilde{v} \|_{L_4(\sN)}^4 +    {\frac{1}{4 C_2}}\|   \nabla^\sN  v_r \|_{L_2(\sN)}^2  .
\end{split}
\end{equation}

\medskip\noindent
{\bf Step 4:}
{Multiplying \eqref{est-v-osc-L4} with $C_2$ and adding the result with \eqref{est-v-avg-H1} and \eqref{est-vr-L2}, yields
\begin{align*}
&\partial_t \left( \frac{C_2}{4}  \| \tilde{v} \|_{L_4(\sN)}^4  +   \| v_r  \|_{L_2(\sN)}^2  +  20  \| \nabla \overline{v} \|_{L_2(\Sigma)}^2  \right)\\
&\quad  +  \frac{1}{2} \left(C_2\left\| | \tilde{v}|_g |   \nabla^\sN  \tilde{v}|_g \right\|_{L_2(\sN)}^2 +   \|   \nabla^\sN  v_r \|_{L_2(\sN)}^2 +   \| \wgd \pi_s \|_{L_2(\Sigma)}^2\right)\\
&\leq C  \left(  1 +    \|  v\|_{L_2(\sN)}^{\frac{4}{1-2\theta}}  + \|   \nabla^\sN  v \|_{L_2(\sN)}^2 + \|    v \|_{L_2(\sN)}^2\|    \nabla^\sN  v \|_{L_2(\sN)}^2 +\|\nabla^{\sN} v\|_{L_2(\sN)}^2 \|  v\|_{L_2(\sN)}^{\frac{2+4\theta}{1-2\theta}}   \right) \times \\
& \qquad \left(  \frac{C_2}{4}  \| \tilde{v} \|_{L_4(\sN)}^4  +   \| v_r  \|_{L_2(\sN)}^2  + 20 \| \nabla \overline{v} \|_{L_2(\Sigma)}^2   \right)  +C \left(  \|  v \|_{L_2(\sN)}^2  + \|  v \|_{L_2(\sN)}^6  \right)     .
\end{align*}
Let
\begin{align*}
\varphi(t):& = \frac{C_2}{4}  \| \tilde{v} (t)\|_{L_4(\sN)}^4  +   \| v_r (t) \|_{L_2(\sN)}^2  +20  \| \nabla \overline{v}(t) \|_{L_2(\Sigma)}^2 \\
& \quad + \frac{1}{2}\int_0^t  \big( C_2\left\| | \tilde{v}|_g |   \nabla^\sN  \tilde{v}|_g \right\|_{L_2(\sN)}^2
 +   \| \wgd \pi_s \|_{L_2(\Sigma)}^2     +     \|   \nabla^\sN  v_r \|_{L_2(\sN)}^2 \big)(\tau) \, d\tau , \\
a_1(t) : & =  \frac{C_2}{4}  \| \tilde{v} (0)\|_{L_4(\sN)}^4  +   \| v_r (0) \|_{L_2(\sN)}^2  + 20 \| \nabla \overline{v}(0) \|_{L_2(\Sigma)}^2  \\
&\qquad+  C \int_0^t \left(  \|  v (\tau)\|_{L_2(\sN)}^2  + \|  v (\tau)\|_{L_2(\sN)}^6  \right)   \, d\tau \\
a_2(t) : &= C \left(  1 +    \|  v\|_{L_2(\sN)}^{\frac{4}{1-2\theta}}   + \|   \nabla^\sN  v(t) \|_{L_2(\sN)}^2 + \|    v (t)\|_{L_2(\sN)}^2\|    \nabla^\sN  v(T) \|_{L_2(\sN)}^2  \right. \\
&\quad \left. +\|\nabla^{\sN} v\|_{L_2(\sN)}^2 \|  v\|_{L_2(\sN)}^{\frac{2+4\theta}{1-2\theta}}   \right)
\end{align*}
}
After integrating the above inequality from $0$ to $t$, we infer that
\begin{align*}
0\le \varphi(t) \le a_1(t) + \int_0^t  a_2(t) \varphi(t),\quad t\in [0,T].
\end{align*}
An application of Gronwall's inequality together with Lemma \ref{lem:basic-L2-est} yields
\begin{equation}
\label{combine-est-1}
\begin{split}
&  \frac{C_2}{4}  \| \tilde{v} (t)\|_{L_4(\sN)}^4  +   \| v_r (t) \|_{L_2(\sN)}^2  +20  \| \nabla \overline{v}(t) \|_{L_2(\Sigma)}^2 \\
& \quad + \frac{1}{2}\int_0^t  \left( C_2\left\| | \tilde{v}|_g |   \nabla^\sN  \tilde{v}|_g \right\|_{L_2(\sN)}^2  +        \| \wgd \pi_s \|_{L_2(\Sigma)}^2     +     \|   \nabla^\sN  v_r \|_{L_2(\sN)}^2 \right)(\tau) \, d\tau \\
&\leq  B_2( \|v_0 \|_{ H^1_2(\sN) },   T)
\end{split}
\end{equation}
for each $t\in [0,T]$.

\goodbreak
\medskip\noindent
{\bf Step 5:} {$L_\infty((0,T); H^1_2(\sN))$}-estimate for $v$. Here we test \eqref{PE} with $- \Delta_\sN v$.

\medskip\noindent
  Employing~\eqref{Green-1} and the boundary conditions for $v$ results in
\begin{align*}
-( \Delta_{\sN} v, \partial_t v)_\sN = (  \nabla^\sN  v,   \nabla^\sN  \partial_t v)_\sN - ( \nabla^\sN_{\nu}  v, \partial_t v)_{\pa\sN }= \frac{1}{2}   \partial_t \|   \nabla^\sN  v \|_{L_2(\sN)}^2.
\end{align*}
An  absorbing argument then yields

\begin{align*}
  \partial_t \|   \nabla^\sN  v \|_{L_2(\sN)}^2 + \|  \Delta _\sN v \|_{L_2(\sN)}^2  &
   \leq C\big( \| \Ric   v \|_{L_2(\sN)}^2 + \| \wgd \pi_s   \|_{L_2(\sN)}^2 + \left\| \nabla_{\overline{v}} \overline{v} \right\|_{L_2(\sN)}^2
   + \left\| \nabla_{\overline{v}} \tilde{v} \right\|_{L_2(\sN)}^2 \\
  & \quad  + \left\| \nabla_{\tilde{v}} \overline{v} \right\|_{L_2(\sN)}^2  + \left\| \nabla_{\tilde{v}} \tilde{v} \right\|_{L_2(\sN)}^2 + \left\| w \partial_r v \right\|_{L_2(\sN)}^2 \big) \\
  &=: III_1 + \cdots + III_7.
\end{align*}
It is clear that $III_1 \leq C \|  v \|_{L_2(\sN)}^2$.
The  terms $III_3$-$III_7$   can be estimated by using \eqref{bdd-avg}, \eqref{bdd-avg-gd}, and \eqref{G-N ineq}$_2$, \eqref{G-N ineq}$_4$  as follows.
\begin{align*}
III_3\, &\leq   C \| \overline{v} \|_{L_4(\Sigma)}^2 \|  \nabla \overline{v} \|_{L_4(\Sigma)}^2 \\
& \leq C \left( \| \overline{v} \|_{L_2(\Sigma)}   \|  \nabla \overline{v} \|_{L_2(\Sigma)}  +   \| \overline{v} \|_{L_2(\Sigma)}^2  \right)  \left( \|  \nabla \overline{v} \|_{L_2(\Sigma)}   \|  \nabla^2   \overline{v} \|_{L_2(\Sigma)}  +   {\|\nabla \overline{v} \|_{L_2(\Sigma)}^2 } \right) \\
& \leq C \left( \| \overline{v} \|_{L_2(\Sigma)}   \|  \nabla \overline{v} \|_{L_2(\Sigma)}  +   \| \overline{v} \|_{L_2(\Sigma)}^2  \right) \\
&\quad  \times  \left( \|  \nabla \overline{v} \|_{L_2(\Sigma)} \left(  \| \overline{v} \|_{L_2(\Sigma)}  +  \|  \Delta_{\Sigma} \overline{v} \|_{L_2(\Sigma)} \right)  +  { \| \nabla\overline{v} \|_{L_2(\Sigma)}^2 } \right) \\
&{\leq C  \| v \|_{L_2(\sN)}^6 +C  \left(1 +  \| v \|_{L_2(\sN)}^2   \|  \nabla \overline{v} \|_{L_2(\Sigma)}^2  + \| v \|_{L_2(\sN)}^4    \right) \|    \nabla^\sN  v \|_{L_2(\sN)}^2}   +   \frac{1}{8}  \|  \Delta_\sN  v \|_{L_2(\sN)}^2  \\
III_4\, &\leq    C \| \overline{v} \|_{L_4(\Sigma)}^2 \|  \nabla \tilde{v} \|_{L_4(\sN)}^2  \leq    C \| \overline{v} \|_{L_4(\Sigma)}^2  \|  \nabla v \|_{L_4(\sN)}^2  \\
&\leq C \left(  \| \overline{v} \|_{L_2(\Sigma)}\|\nabla \overline{v} \|_{L_2(\Sigma)} +  \| \overline{v} \|_{L_2(\Sigma)}^2 \right)
  \|   \nabla^\sN  v \|_{L_2(\sN)}^{1/2}  \|  (\nabla^{\sN})^2 v \|_{L_2(\sN)}^{3/2}    \\
&\leq C \left(  \| \overline{v} \|_{L_2(\Sigma)}\|\nabla \overline{v} \|_{L_2(\Sigma)} +  \| \overline{v} \|_{L_2(\Sigma)}^2 \right) \|   \nabla^\sN  v \|_{L_2(\sN)}^{1/2} {\| \Delta_\sN  v \|_{L_2(\sN)}^{3/2} }   \\
& \leq   { C \| v \|_{L_2(\sN)}\|\nabla \overline{v} \|_{L_2(\Sigma)}\|   \nabla^\sN  v \|_{L_2(\sN)}^{1/2}\| \Delta_\sN  v \|_{L_2(\sN)}^{3/2} +C \| v \|_{L_2(\sN)}^2 \|  \nabla^\sN  v \|_{L_2(\sN)}^{1/2}\| \Delta_\sN  v \|_{L_2(\sN)}^{3/2} }\\
&\leq  C \left( 1 + {\| v \|_{L_2(\sN)}^8 } + \| v \|_{L_2(\sN)}^4 \|\nabla \overline{v} \|_{L_2(\Sigma)}^4  \right) \|   \nabla^\sN  v \|_{L_2(\sN)}^2  +  \frac{1}{8}  \|  \Delta_\sN v \|_{L_2(\sN)}^2    \\
III_5\, &\leq  C \| \tilde{v} \|_{L_4(\sN)}^2 \|  \nabla \overline{v} \|_{L_4(\Sigma)}^2 \leq  C \| \tilde{v} \|_{L_4(\sN)}^2  \left( \|  \nabla  \overline{v} \|_{L_2(\Sigma)} \|  \nabla^2 \overline{v}   \|_{L_2(\Sigma)} + {\|\nabla \overline{v} \|_{L_2(\Sigma)}^2} \right)  \\
&\leq C \| \tilde{v} \|_{L_4(\sN)}^2 \left[ \|  \nabla  \overline{v} \|_{L_2(\Sigma)}  \left(  \| \overline{v} \|_{L_2(\Sigma)}  +  \|  \Delta_\Sigma \overline{v} \|_{L_2(\Sigma)} \right)     + {\|\nabla \overline{v} \|_{L_2(\Sigma)}^2}   \right]\\
&\leq  C  \| \tilde{v} \|_{L_4(\sN)}^2 \| v \|_{L_2(\sN)}    \|    \nabla^\sN  v \|_{L_2(\sN)}   \\
&\hspace{3cm} + C  \| \tilde{v} \|_{L_4(\sN)}^2 {\|   \nabla^\sN  v \|_{L_2(\sN)}^2}  +  C  \| \tilde{v} \|_{L_4(\sN)}^2  \|    \nabla^\sN  v \|_{L_2(\sN)}  \|  \Delta_\sN v \|_{L_2(\sN)}    \\
&\leq    C \| \tilde{v} \|_{L_4(\sN)}^4 \| v \|_{L_2(\sN)}^2  +    C \left( 1  +  \| \tilde{v} \|_{L_4(\sN)}^4  \right)   \|    \nabla^\sN  v \|_{L_2(\sN)}^2  +  \frac{1}{8}  \|  \Delta_\sN v \|_{L_2(\sN)}^2   \\
III_6\, &\leq  \left\|  |\tilde{v} |_g |   \nabla^\sN  \tilde{v} |_g\right\|_{L_2(\sN)}^2 \\
III_7\, & {\leq \| w \|_{L_\infty(L_4)}^2 \| v_r \|_{L_2(L_4)}^2 \leq C \| \wdv v \|_{L_2(L_4)}^2 \| v_r \|_{L_2(L_4)}^2} \\
&{\leq C   \| \wdv v \|_{L_2(\sN)} \| \wdv v \|_{ L_2( {H^1 }) }  \| v_r \|_{L_2(\sN)} \| v_r \|_{L_2({H^1})}} \\
&{\leq C \|   \nabla^\sN  v \|_{L_2(\sN)} \|  v \|_{{H^2}(\sN)}  \| v_r \|_{L_2(\sN)} \|   \nabla^\sN  v_r \|_{L_2(\sN)}}\\
&{\leq C \|   \nabla^\sN  v \|_{L_2(\sN)}\| \Delta_\sN v \|_{L_2(\sN)}  \| v_r \|_{L_2(\sN)} \|   \nabla^\sN  v_r \|_{L_2(\sN)}}\\
&{\leq C\| v_r \|_{L_2(\sN)}^2 \|   \nabla^\sN  v_r \|_{L_2(\sN)}^2\|   \nabla^\sN  v \|_{L_2(\sN)}^2+\frac{1}{8}  \|  \Delta_\sN  v \|_{L_2(\sN)}^2},
\end{align*}
where,  in $III_5$ we used \eqref{elliptic-estimate},
while in  $III_4$ and $III_7$, we have used the fact that
$$\|v\|_{H^2_2(\sN)}\leq C \| \Delta_\sN v \|_{L_2(\sN)},$$
see Lemma~\ref{lem: N-invertible}.
Additionally, in the estimate for $III_7$, we used the interpolation inequality
$$\|u\|_{L_4(\Sigma)}^2\le C\|u\|_{L_2(\Sigma)}\|u\|_{H^1_2(\Sigma)}$$
as well as the Poincar\'{e} inequality for $v_r$,
see Lemma~\ref{lem: Poincare}.
Finally, we have set $L_r(L_q):= L_r((-h,0);L_q(\Sigma))$.

\medskip\noindent
Combining the estimates for $III_1$-$III_7$, we have
$$
  \partial_t \|   \nabla^\sN  v \|_{L_2(\sN)}^2 +  \frac{1}{2} \| \Delta_\sN v \|_{L_2(\sN)}^2  \leq C a_1(t)    + Ca_2(t)  \|   \nabla^\sN  v \|_{L_2(\sN)}^2,
$$
with
\begin{align*}
a_1(t): &=1 + \| \tilde{v} \|_{L_4(\sN)}^4 \| v \|_{L_2(\sN)}^2  
+ \| v \|_{L_2(\sN)}^6
+ \| \wgd \pi_s   \|_{L_2(\sN)}^2 +  \left\|  |\tilde{v} |_g |   \nabla^\sN  \tilde{v} |_g\right\|_{L_2(\sN)}^2
\end{align*}
and
$$a_2(t):= 1 + \| \tilde{v} \|_{L_4(\sN)}^4 + {\| v \|_{L_2(\sN)}^8 } + \| v \|_{L_2(\sN)}^4 \|\nabla \overline{v} \|_{L_2(\Sigma)}^4+\| v_r \|_{L_2(\sN)}^2 \|   \nabla^\sN  v_r \|_{L_2(\sN)}^2.$$
{Then, Gronwall's inequality implies
\begin{equation}
\label{est-v-H1}
\begin{split}
  \|   \nabla^\sN  v (t) \|_{L_2(\sN)}^2 + \int_0^t \| \Delta_\sN v (\tau) \|_{L_2(\sN)}^2 \, d\tau
 \le B_3( \|v_0 \|_{ H^1_2(\sN) } , T),
\end{split}
\end{equation}
for each $t\in [0,T]$.
}

\goodbreak
\medskip\noindent
{\bf Step 6:} $L_\infty((0,T); L_2(\sN))$-estimate for $\partial_t v$:

\medskip\noindent
Multiplying \eqref{PE}$_1$ by $\partial_t v$ and integrating over $\sN$, we can derive
\begin{align*}
& \| v_t \|_{L_2(\sN)}^2 \\
&=  \left( (\Delta_\sN +\Ric)v , v_t \right)_\sN -   \left( \nabla_v v , v_t \right)_\sN -    \left( w\partial_r v , v_t \right)_\sN \\
&\leq  C \left( \| \Delta_\sN v \|_{L_2(\sN)} + \| v \|_{L_2(\sN)} +  \| \nabla v \|_{L_4(\sN)}  \| v \|_{L_4(\sN)} +  \| \wdv v \|_{L_4(\sN)}  \| \partial_r v \|_{L_4(\sN)}   \right) \| v_t \|_{L_2(\sN)}  \\
& \leq C \left( \| v\|_{H^2_2(\sN)} +  \| v\|_{H^2_2(\sN)}^2      \right)  \| v_t \|_{L_2(\sN)}
\end{align*}
in view of  \eqref{PE}$_3$ and the Sobolev embedding  $H^2_2  (\sN)\hookrightarrow H_4^1 (\sN)$, cf. \cite[Theorem~14.2]{Ama13}.
Evaluating the above inequality at $t=0$, which gives
\begin{equation}
\label{est-vt-0}
{\|  v_t  (0)\|_{L_2(\sN)} \leq C  \left(  \| v_0\|_{H^2_2 (\sN)} +  \| v_0\|_{H^2_2 (\sN)}^2    \right). }
\end{equation}
Next, taking the temporal derivative of \eqref{PE}$_1$ and then multiplying by $\partial_t v$ yields
\begin{align*}
 \frac{1}{2} \partial_t \| v_t \|_{L_2(\sN)}^2 + \|   \nabla^\sN  v_t \|_{L_2(\sN)}^2   & =  \int_\Sigma \partial_t \pi_s \partial_t \left( \int_{-h}^0  \wdv v  \, dr  \right) \, d\mu_g + \int_{\sN} \Ric | v_t |^2 \, d\mu_{g_\sN}    \\
 &\quad  -   \left[  ( \nabla_{v_t} v  ,  v_t)_\sN  + (w_t \partial_r v , v_t)_\sN  \right]
  -     \left(\left(\nabla_v + w\partial_r  \right) v_t , v_t \right)_\sN   \\
 & \leq  C \| v_t \|_{L_2(\sN)}^2  -   \left[  ( \nabla_{v_t} v  ,  v_t)_\sN  + (w_t \partial_r v , v_t)_\sN \right]  ,
\end{align*}
where we have used Proposition~\ref{pro: divergence-free}  with $z=v_t$ and $q=2$,
\eqref{PE}$_3$, the boundary conditions for $w$.
The last two terms on the RHS can be estimated by using \eqref{G-N ineq}$_2$, \eqref{G-N ineq}$_5$, and \eqref{G-N ineq}$_6$ as  follows:
\begin{align*}
\int_\sN   \la \nabla_{v_t} v  , v_t\ra_{g_\sN}   \, d\mu_{g_\sN}  & \leq  C \|   \nabla^\sN  v \|_{L_2(\sN)} \| v_t \|_{L_4 (\sN)}^2  \leq C \|   \nabla^\sN  v \|_{L_2(\sN)} \| v_t \|_{L_2 (\sN)}^{1/2} \|   \nabla^\sN  v_t \|_{L_2 (\sN)}^{3/2} \\
&\leq  C \|   \nabla^\sN  v \|_{L_2(\sN)}^4 \| v_t \|_{L_2 (\sN)}^2       + \frac{1}{4} \|   \nabla^\sN  v_t \|_{L_2(\sN))}^2
\end{align*}
and
\begin{align*}
& \int_\sN   \la w_t \partial_r v ,  v_t\ra_{g_\sN}   \, d\mu_{g_\sN} \\
& \leq       \left(\int_{-h}^0   \| \wdv v_t (\cdot,z)\|_{L_2(\Sigma)} \, dr    \right) \| v_r \|_{L_2((-h,0);L_3(\Sigma))} \| v_t \|_ {L_2((-h,0);L_6(\Sigma))}\\
&\leq C \| \wdv v_t  \|_{L_2(\sN)} \left(  \| v_r \|_{L_2(\sN)}^{2/3} \| \nabla v_r \|_{L_2(\sN)}^{1/3} +  \| v_r \|_{L_2(\sN)} \right) \left(  \| v_t \|_{L_2(\sN)}^{1/3} \| \nabla v_t \|_{L_2(\sN)}^{2/3} +  \| v_t \|_{L_2(\sN)} \right)  \\
&\leq C \|   \nabla^\sN  v_t  \|_{L_2(\sN)}   \|   v_r \|_{L_2(\sN)}   \| v_t \|_{L_2(\sN)}  + C \|  \nabla^\sN  v_t  \|_{L_2(\sN)}^{5/3}  \| v_r \|_{L_2(\sN)}^{2/3} \|   \nabla^\sN  v_r \|_{L_2(\sN)}^{1/3} \| v_t \|_{L_2(\sN)}^{1/3}   \\
& \quad + C \|   \nabla^\sN  v_t  \|_{L_2(\sN)}   \| v_r \|_{L_2(\sN)}^{2/3} \|   \nabla^\sN  v_r \|_{L_2(\sN)}^{1/3} \| v_t \|_{L_2(\sN)}  + C \|   \nabla^\sN  v_t  \|_{L_2(\sN)}^{5/3}   \| v_r \|_{L_2(\sN)}  \| v_t \|_{L_2(\sN)}^{1/3}   \\
&\leq C \left( \|   v_r \|_{L_2(\sN)}^2   +  \| v_r \|_{L_2(\sN)}^4 \|   \nabla^\sN  v_r \|_{L_2(\sN)}^2 +  \| v_r \|_{L_2(\sN)}^{4/3} \|   \nabla^\sN  v_r \|_{L_2(\sN)}^{2/3}   + \| v_r \|_{L_2(\sN)}^6 \right) \| v_t \|_{L_2(\sN)}^2 \\
&\quad    +  \frac{1}{4} \|   \nabla^\sN  v_t \|_{L_2(\sN))}^2.
\end{align*}
Combing these estimates yields
\begin{align*}
&\| v_t (t) \|_{L_2(\sN)}^2 + \int_0^t \|   \nabla^\sN  v_t (\tau) \|_{L_2(\sN)}^2 \, d\tau  \leq \| v_t (0) \|_{L_2(\sN)}^2   \\
& \quad  + C \int_0^t \left( 1+\|   \nabla^\sN  v (\tau)\|_{L_2(\sN)}^4 +  \| v_r (\tau)\|_{L_2(\sN)}^4 \|   \nabla^\sN  v_r (\tau)\|_{L_2(\sN)}^2  + \| v_r (\tau)\|_{L_2(\sN)}^6  \right) \| v_t (\tau) \|_{L_2(\sN)}^2\, d\tau .
\end{align*}
Then, {Gronwall's} inequality,  \eqref{combine-est-1}, and \eqref{est-vt-0} imply
\begin{equation}
\label{est-vt-L2}
\begin{split}
& \| v_t (t) \|_{L_2(\sN)}^2 + \int_0^t \|   \nabla^\sN  v_t (\tau) \|_{L_2(\sN)}^2 \, d\tau  \\
& \leq C   \| v_t (0) \|_{L_2(\sN)}^2
 \exp \big( \! \!\int_0^T \!\!\! \big( 1+\|   \nabla^\sN  v  \|_{L_2(\sN)}^4  +  \| v_r  \|_{L_2(\sN)}^4 \|   \nabla^\sN  v_r  \|_{L_2(\sN)}^2  + \| v_r \|_{L_2(\sN)}^6  \big)  (\tau) \, d\tau \big) \\
&=: B_4( \|v_0 \|_{ H^2_2(\sN) },    T),
\end{split}
\end{equation}
for each $t\in [0,T]$.

\medskip\noindent
{\bf Step 7:} $L_\infty((0,T); {L_3(\sN) } )$-estimate for $v_r =\partial_r v$. We test  \eqref{PE}$_1$ with  $-\partial_r (|v_r|_g v_r)$,

\medskip\noindent
We first observe that  Green's identity~\eqref{Green-1} and integration by parts with respect to $r$  yields
\begin{align*}
 (\Delta_\sN v, \partial_r(|v_r| v_r))_\sN &=  -(  \nabla^\sN  v,    \nabla^\sN  \partial_r (|v_r| v_r) )_\sN
 + (  \nabla^\sN_{\nu} v,  \partial_r(|v_r| v_r))_{\pa\sN}\\
& =   (  \nabla^\sN   v_r ,   \nabla^\sN   (|v_r| v_r))_\sN  \\
&\quad -\int_\Sigma \Big[ \la   \nabla^\sN  v ,    \nabla^\sN ( |v_r| v_r) \ra_{g_\sN} - \la \partial_r v,  \partial_r(|v_r| v_r)\ra_\sN
 \Big]^{r=0}_{r=-h} \, d\mu_g \\
& =   (  \nabla^\sN   v_r ,   \nabla^\sN   (|v_r| v_r))_\sN .
\end{align*}
In the derivation above we used \eqref{splitting-2} to obtain
\begin{align*}
\la   \nabla^\sN  v ,    \nabla^\sN ( |v_r| v_r) \ra_{g_\sN}
& =  \la \nabla  v ,  \nabla ( |v_r| v_r)  \ra_{\sN} + \la  \partial_r v, \partial_r ( |v_r| v_r) \ra_{\sN}.
\end{align*}
Moreover,  we used the boundary conditions to deduce
$$
\int_\Sigma\Big[ \la \nabla  v ,  \nabla ( |v_r| v_r)  \ra_{\sN} \Big]^{r=0}_{r=-h} \, d\mu_g =0.
$$
Finally, invoking \eqref{computation} yields
$$
 (\Delta_\sN v, \partial_r(|v_r| v_r))_\sN =
  (  \nabla^\sN   v_r ,   \nabla^\sN   (|v_r| v_r))_\sN = (  \nabla^\sN  v_r , | v_r |   \nabla^\sN  v_r)_\sN
  +  \frac{4}{9} \left\| \gd | v_r|_g^{3/2} \right\|_{L_2(\sN)}^2 .
$$
Integrating by parts with respect to $r$ and using the boundary conditions for $v$ and $w$,
it follows from Proposition~\ref{pro: divergence-free}  with $z= v_r$ and $q=3,$
and from $\eqref{PE}_3$  that
\goodbreak
\begin{align*}
 -(\nabla_v v+ w v_r , \partial_r (|v_r|_g v_r))_\sN
&= - \int_\Sigma \Big[ \la  \nabla_v v + w v_r, |v_r|_g v_r\ra_{g_\sN} \Big]^{r=0}_{r=-h } \\
& \quad + (\nabla_v v_r + w \pa_r v_r ,  |v_r|_g v_r )_\sN  +  (\nabla _{v_r} v + (\pa_r w) v_r ,  |v_r|_g v_r )_\sN  \\
& =  (\nabla _{v_r} v ,  |v_r|_g v_r )_\sN  -   \int_{\sN}   \wdv v \,  |v_r|_g^3    \, d\mu_{g_\sN}.
\end{align*}

Combining, we have
\begin{align*}
& \frac{1}{3}  \partial_t \| v_r \|_{L_3(\sN)}^3  + \frac{4}{9} \left\| \gd | v_r|_g^{3/2} \right\|_{L_2(\sN)}^2
+ \left\|  | v_r|_g^{1/2}   \nabla^\sN  v_r \right\|_{L_2(\sN)}^2\\
&=  \int_{\sN}   \wdv v   |v_r|_g^3    \, d\mu_{g_\sN}
      -   \left(  \nabla_{v_r} v  , v_r |v_r|_g   \right)_\sN +\int_{\sN}  \la \Ric \, v_r  , v_r |v_r |_g\ra_g  \, d\mu_{g_\sN}    \\
&   \hspace{6cm}    - \int_{\Sigma_b} \la \wgd \pi  ,  v_r |v_r|_g  \ra_g   \, d\mu_{g_\sN} \\
&  = :{IV_1 + \cdots + IV_4.}
\end{align*}

It follows from $\eqref{G-N ineq}_2$ that
\begin{align*}
{IV_1+IV_2}\, & \le  { C} \int_{\sN}   |\nabla v|_g   |v_r|_g^3    \, d\mu_{g_\sN}  \leq  { C} \| \nabla v \|_{L_2(\sN)} \left\| |v_r|_g^3 \right\|_{L_2(\sN)}
=   { C} \| \nabla v \|_{L_2(\sN)} \left\| |v_r|_g^{3/2} \right\|_{L_4(\sN)}^2 \\
&\leq  C \| \nabla v \|_{L_2(\sN)} \left\| |v_r|_g^{3/2} \right\|_{L_2(\sN)}^{1/2}  \left\| \gd |v_r|_g^{3/2} \right\|_{L_2(\sN)}^{3/2}  \\
&\leq C \| \nabla v \|_{L_2(\sN)}^4 \left\| |v_r|_g^{3/2} \right\|_{L_2(\sN)}^2 + \frac{2}{9}   \left\| \gd |v_r|_g^{3/2} \right\|_{L_2(\sN)}^2 \\
&= C \| \nabla v \|_{L_2(\sN)}^4   \| v_r \|_{L_3(\sN)}^3 + \frac{2}{9}   \left\| \gd |v_r|_g^{3/2} \right\|_{L_2(\sN)}^2.
\end{align*}
Direct computations show
\begin{align*}
IV_3 \leq C \| v_r \|_{L_3(\sN)}^3 .
\end{align*}
By the trace theorem and Sobolev embedding theorem, cf. \cite[Theorems 10.1 and 14.2]{Ama13}, we can conclude that
$$
H^{3/4}_2(\sN) \hookrightarrow L_{8/3}(\Sigma_b) .
$$
Then applying interpolation theory, cf. \cite[Theorem 13.1]{Ama13}, one derives
\begin{align*}
IV_4\, &\leq \|  \gd_\Sigma \pi_s \|_{L_2(\Sigma_b)} \left\| |v_r|_g^2 \right\|_{L_2(\Sigma_b)} = \| \gd_\Sigma \pi_s \|_{L_2(\Sigma_b)} \left\| |v_r|_g^{3/2} \right\|_{L_{8/3}(\Sigma_b)}^{4/3} \\
&\leq C \| \gd_\Sigma \pi_s \|_{L_2(\Sigma_b)} \left\| |v_r|_g^{3/2} \right\|_{H^{3/4}_2(\sN)}^{4/3} \\
&\leq  C \| \gd_\Sigma \pi_s \|_{L_2(\Sigma_b)} \left\| |v_r|_g^{3/2} \right\|_{L_2(\sN)}^{1/3} \left\| \gd |v_r|_g^{3/2} \right\|_{L_2(\sN)} \\
&\leq  C \| \gd_\Sigma \pi_s \|_{L_2(\Sigma_b)}^2 \left\|  v_r  \right\|_{L_3(\sN)}  + \frac{2}{9}  \left\| \gd |v_r|_g^{3/2} \right\|_{L_2(\sN)}^2 \\
&\leq C \| \gd_\Sigma \pi_s \|_{L_2(\Sigma_b)}^2 \left( \left\|  v_r  \right\|_{L_3(\sN)}^3 + 1 \right)  + \frac{2}{9}  \left\| \gd |v_r|_g^{3/2} \right\|_{L_2(\sN)}^2.
\end{align*}
Combining the estimates for $IV_1-IV_4$ gives
\begin{align*}
  \partial_t \| v_r \|_{L_3(\sN)}^3
   \leq C   \| \gd_\Sigma \pi_s \|_{L_2(\Sigma_b)}^2     + C \left( 1 +   \| \gd_\Sigma \pi_s \|_{L_2(\Sigma_b)}^2  + { \|   \nabla^\sN  v  \|_{L_2(\sN)}^4}  \right)\| v_r \|_{L_3(\sN)}^3 .
\end{align*}
Then {Gronwall's inequality} implies
\begin{equation}
\label{est-vr-L3}
\begin{split}
 \| v_r (t)\|_{L_3(\sN)}^3
& \leq C  \left(\| v_r (0)\|_{L_3(\sN)}^3 + \int^T_0  \| \gd_\Sigma \pi_s (\tau)\|_{L_2(\Sigma_b)}^2    \, d\tau \right)  \times \\
&  \exp \left( C   \int_0^T \left( 1 +   \| \gd_\Sigma \pi_s \|_{L_2(\Sigma_b)}^2   + { \|   \nabla^\sN  v  \|_{L_2(\sN)}^4} \right)(\tau) \, d\tau \right) \\
&=: B_5( \|v_0 \|_{ H^2_2(\sN) } , T),
\end{split}
\end{equation}
{for each $t\in [0,T]$.}

\medskip\noindent
{\bf Step 8:} $L_\infty((0,T); { H^2_2(\sN)} )$-estimate for $v$.

\medskip\noindent
It follows from Theorem~\ref{thm: H-infinity calculus} that there exists some $\lambda>0$ such that
\begin{align*}
\|v\|_{H^2_2(\sN)} \leq \| (\lambda + \sA) v \|_{L_2(\sN)} \leq \lambda \| v \|_{L_2(\sN)}  + \|\sA v \|_{L_2(\sN)} .
\end{align*}
Then \eqref{PE-abstract} implies
\begin{align*}
\| v \|_{H^2_2(\sN)} & \leq \lambda \| v \|_{L_2(\sN)}  + \|\sA v \|_{L_2(\sN)} \\
& \leq   \lambda \left\| v \right\|_{L_2(\sN)}   + \left\| \partial_t v \right\|_{L_2(\sN)}  + \left\| \nabla_v v \right\|_{L_2(\sN)} + \left\| w \partial_r v \right\|_{L_2(\sN)} .
\end{align*}
Most of the terms on the RHS have been evaluated previously. We will only estimate the third and fourth terms.
It follows from the Sobolev embedding theorem, cf. \cite[Theorem 14.2]{Ama13},  that
\begin{align*}
\left\| \nabla_v v \right\|_{L_2(\sN)}  & \leq  \| v \|_{L_6(\sN)} \|   \nabla^\sN  v \|_{L_3(\sN)}  \leq  C \| v \|_{H^1_2(\sN)}  \| v \|_{H^2_2(\sN)} \leq C \| v \|_{H^1_2(\sN)}^2 + \frac{1}{4} \| v \|_{H^2_2(\sN)}^2  \\
\left\| w \partial_r v \right\|_{L_2(\sN)} &\leq \| w \|_{L_6(\sN)} \| v_r \|_{L_3(\sN)} \leq C \| v \|_{H^2_2(\sN)} \| v_r \|_{L_3(\sN)} \leq C \| v_r \|_{L_3(\sN)}^2 + \frac{1}{4} \| v \|_{H^2_2(\sN)}^2 .
\end{align*}
Combining Lemma~\ref{lem:basic-L2-est},  \eqref{est-v-H1},  \eqref{est-vt-L2}, and \eqref{est-vr-L3}, we can infer that
\begin{equation}
\label{est-v-H2}
\begin{split}
\| v  (t)\|_{H^2_2(\sN)}    \leq  B_6( \|v_0 \|_{ H^2_2(\sN) },    T),
\end{split}
\end{equation}
{for each $t\in [0,T]$.}
\end{proof}

\medskip\noindent
\subsection{Global existence}
Now we are ready to state and prove the main result of this article.

\begin{theorem}\label{thm: global}
{
Let   $1<p,q<\infty$ with $1/p+1/q\leq 1$, and $\mu \in [1/p+1/q,1]$.
Assume that
$$
v_0 \in X_{\mu-1/p,p} .
$$
Let $v\in\mathbb{E}_{1,\mu}(a)$ be the solution to \eqref{PE} asserted by Proposition~\ref{prop:local existence}(a) on some maximal interval of interval of existence $[0,T_+(v_0))$.
Then  $v$   exists globally, i.e, $T_+(v_0)=\infty$.
}
\end{theorem}
\begin{proof}
By Proposition \ref{prop:local existence}(d) we have in particular $v\in C([\delta,a]; H^2_2(\sN; T\Sigma))$
for all $0<\delta<a < T_+ (v_0)$. Then, Theorem~\ref{thm: a priori est} shows that
$v\in BC([ \delta ,T_+) ; H^2_2(\sN; T\Sigma))$, hence
$$v\in BC([\delta ,T_+) ; H^s_q(\sN; T\Sigma)), \quad s\leq \frac{1}{2} + \frac{3}{q},$$
by Sobolev embedding, see \cite[Theorem~14.2]{Ama13}. It follows from \eqref{Def:Besov} and  interpolation theory \cite[Theorem~13.1]{Ama13}  that
$$
H^{s_1}_q (\sN; T\Sigma)) \hookrightarrow B^{s_0}_{qp}(\sN; T\Sigma), \quad 0<s_0<s_1<\infty,
$$
and therefore
$$
v\in    BC([\delta ,T_+) ; B^{2/q + \eta}_{qp}(\sN; T\Sigma))
$$
for each $\eta\in (0,1/2+1/q)$.
Setting $\bar \mu = 1/p+1/q +\eta/2$, we see that
$$X_{\bar \mu-1/p,p} \subset B^{2/q+\eta}_{qp}(\sN; T\Sigma). $$
We can now apply Proposition~\ref{prop:local existence}(e) with $\mu_c = 1/p +1/q$
to conclude that $T_+(v_0)=\infty$.
Indeed, we  just need to observe that the assumption $v_0\in X_{\mu-1/p,p}$ implies
$v_0\in X_{\mu_c-1/p, p}$, so that Proposition~\ref{prop:local existence}(e) applies with $\mu= \mu_c$.
\end{proof}

\appendix

\section{Some fundamental results for Riemannian manifolds}
\label{S: Appendix A}

\subsection{Riemannian manifolds}
\label{sub: Riemannian}
In this section, we consider a general Riemannian manifold \((M,g)\) equipped with its Levi--Civita connection \(\nabla\).
We adopt notation that is standard in differential geometry and collect several definitions, identities, and auxiliary results that will be used throughout the manuscript. For the convenience of the reader, we also briefly recall some basic properties of the geometric objects and differential operators appearing in our analysis.

Let $(M,g)$ be a smooth $n$-dimensional  Riemannian manifold, and let $\la \cdot, \cdot \ra_g$  denote the inner product.
We write  $T{M}$ and $T^{\ast}{M}$ for the tangent and the cotangent bundles of $M$, respectively, and define
$$
T^1_1{M}:=TM\otimes T^*M,
$$
the bundle of $(1,1)$-tensors on $M$.

Furthermore,  $\Gamma(M ;TM)$,  $\Gamma (M; T^*M)$, and  $\Gamma(M;  T^{1}_{1}{M})$  denote the spaces of  smooth sections of
 $TM$, $T^*M$,  and $T^1_1(M)$, respectively.

 Given a local coordinate system, we denote the associated local bases of
 the tangent and cotangent bundles by
$$\{\partial_i : 1\le i\le n\} \quad\text{and}\quad   \{dx^j : 1\le j\le n\},$$
respectively.
For  $X\in \Gamma(M; TM)$ and $S\in \Gamma(M; T^1_1(M))$
the corresponding  local representations are
$$
X= X^i \partial_i,\qquad S=S^i_j \partial_i \otimes dx^j.
$$
The pointwise inner products of vector fields \(X,Y\in\Gamma(M;TM)\) and \((1,1)\)-tensors \(S,T\in\Gamma(M;T^1_1M)\) are given in local coordinates by
\begin{equation}
\label{(1,1)-inner}
\la X, Y\ra_g = g_{ij} X^i Y^j, \qquad \la S, T \ra_g = g^{jk} g_{i\ell}  S^i_j T^\ell_k,
\end{equation}
respectively, where $g_{ij}= \la \partial_i, \partial_j \ra_g$ and  $(g^{ij} )$ denotes the inverse of the metric tensor $(g_{ij})$.

\medskip
Every tensor $S\in\Gamma(M;T^1_1M)$ induces a linear map
$
S:\Gamma(M;TM)\to\Gamma(M;TM)
$
via
$$
SX=
(S^i_j\,\partial_i\otimes dx^j)X = S^i_jX^j\,\partial_i,
\qquad
X=X^j\partial_j\in\Gamma(M;TM).
$$
Moreover, the pointwise estimate
$$
|SX|_g \le |S|_g\,|X|_g
$$
holds, where
$|S|_g = \sqrt{\langle S,S\rangle_g }$.

\medskip
Let \(\nabla\) denote the Levi--Civita connection on \(M\). For a vector field
$
X=X^i\partial_i\in \Gamma(M;TM),
$
the covariant derivative of \(X\) is given in local coordinates by
\begin{equation}
\label{covariant-local}
\nabla_j X= \bigl(\partial_j X^i+\Gamma^i_{jk}X^k\bigr)\partial_i =:X^i_{| j}\,\partial_i\, ,
\end{equation}
where \(\nabla_j=\nabla_{\partial_j}\) and \(\Gamma^i_{jk}\) denote the Christoffel symbols of \(\nabla\).
The covariant derivative
$$
\nabla:C^1(M;TM)\to C(M;T^1_1M)
$$
is therefore represented locally as
\begin{equation}
\label{covariant-derivative}
\nabla X
=\nabla_jX\otimes dx^j = \bigl(\partial_jX^i+\Gamma^i_{jk}X^k\bigr) \partial_i\otimes dx^j
= X^i_{|j}\,\partial_i\otimes dx^j.
\end{equation}
For $X,Y\in C^1(M;TM)$, we have $\nabla_X Y= (\nabla Y) X$ and  the pointwise estimate
\begin{equation}
\label{pointwise-estimate}
|\nabla_XY|_g \le |\nabla Y|_g\,|X|_g
\end{equation}
holds.
For a vector field
$X=X^i\partial_i\in C^1(M;TM),$
the divergence of $X$ is defined by
$$
\dv X:=
X^i{}_{|i} = \partial_iX^i+\Gamma^i_{ik}X^k.
$$
For a scalar function \(\phi\in C^1(M;\bR)\), the gradient
$
\operatorname{grad}_M\phi\in C(M;TM)
$
is defined by
$$
\langle \operatorname{grad}_M\phi,X\rangle_g = \langle \nabla\phi,X\rangle_g
= \nabla_X\phi,
\quad X\in C(M;TM),
$$
where
$\nabla\phi\in C(M;T^*M) $
denotes the covariant derivative of $\phi$.
In local coordinates, the gradient is given by
$$
\operatorname{grad}_M\phi
=g^{ij}\,\partial_j\phi\,\partial_i\, .
$$

For the curvature tensor $R(X,Y)Z:= [\nabla_X \nabla_Y]Z - \nabla_{[X,Y]}Z$, with $X,Y,Z \in \Gamma(M; TM)$,
we use the following convention (as in \cite{Pet06, SaTu20}, for instance)
$$
R\left(\partial_i, \partial_j\right) \partial_k = R^\ell_{ijk } \partial_\ell.
$$
Finally, the Ricci $(2,0)$-tensor is  defined by ${\rm Ric}_{jk}=R^i_{ijk}.$

\bigskip\noindent
Suppose $X$ a vector field on $M$.
We show that
\begin{equation}
\label{computation}
\la \nabla X,\nabla(| X |_g X) \ra_g =|X|_g\,|\nabla X|_g^2+ \frac{4}{9}\,\big|\gd_M |X|_g^{3/2}\big|_g^2,
\end{equation}
where  $|X|_g  = \la X,X \ra^{1/2}_g.$

In local coordinates,
$X = X^i \partial_i$ and  $ |X|_g^2 = g_{ij} X^i X^j.$
By \eqref{covariant-derivative}, the pointwise norm of the  $(1,1)$-tensor $\nabla X$ is given by
$$
|\nabla X|^2_g= g^{jk} g_{i\ell} X^i_{| j} X^\ell_{| k}.
$$
Employing the product rule  $\nabla_j  (|X|_g  X)=( \partial_j |X|_g  )X + |X|_g  \nabla_j X $ results in
\begin{align*}
\la \nabla X,\nabla(|X|_g  X) \ra_g
=g^{jk} g_{i\ell} X^i_{|j} \left[ |X|_g  X^\ell_{| k} + (\partial_k |X|_g  ) X^\ell  \right]
= |X|_g  |\nabla X|^2_g + g^{jk} g_{i\ell} X^i_{|j}   X^\ell \partial_k |X|_g  .
\end{align*}
Using the metric compatibility of the Levi-Civita connection, we obtain
\begin{align*}
 g^{jk} g_{i\ell} X^i_{| j}  X^\ell \, \partial_k |X|_g
 =  g^{jk} \la \nabla_j X, X\ra_g \, \partial_k |X |
 = |X|_g  \,  g^{jk} \, \partial_j |X|_g  \,  \partial_k |X|_g
 = |X|_g  \big|\gd_M|X|_g  \big|^2 .
\end{align*}
Hence
$$
\la \nabla X,\nabla(| X |_g X) \ra_g =|X|_g  \,|\nabla X|_g^2 +  |X|_g  \big|\gd_M|X|_g  \big|^2 .
$$
Observing that
$$
\gd (|X|_g^{3/2}) =\frac{3}{2}|X|_g^{1/2} \gd_M\,|X|_g
$$
yields the assertion in \eqref{computation}.

An inspection of the arguments presented above also shows that
\begin{equation}
\label{computation-2}
\la \nabla X,\nabla(| X |^2_g X) \ra_g =|X|^2_g\,|\nabla X|^2_g+ \frac{1}{2}\,\big|\gd_M|X|^2_g\big|^2_g.
\end{equation}
Indeed, following the proof given above, one observes that by the metric compatibility
$$
  g^{jk} \la \nabla_j X, X\ra_g \, \partial_k |X |^2_g
 = \frac{1}{2}  g^{jk} \, \partial_j |X|^2_g \,  \partial_k |X |^2_g
 = \frac{1}{2}\,\big|\gd_M|X|^2_g\big|^2_g.
$$


\subsection{The product manifold $\Sigma \times [-h,0]$.}
\label{subsec: product manifolds}
In this section, for the reader's convenience we summarize a number of results
that are mostly well-known and have been used throughout the manuscript.

We consider the product manifold $(\sN,g_{\sN})$ with  product metric $g_\sN$, where
$$
\sN=\Sigma\times[-h,0],
\qquad
g_{\sN}=g\oplus g_{\bR},
$$
and $\Sigma$  is a closed hypersurface of $\bR^3$. As before, $g=g_\Sigma$ is the Riemannian metric on $\Sigma$
induced by the ambient Euclidean metric.
At a point \((p,r)\in \sN=\Sigma\times[-h,0]\), the tangent space decomposes as
$$
T_{(p,r)}\sN=T_p\Sigma \oplus T_r[-h,0] =T_p\Sigma \oplus \bR \cong T_p\Sigma \times \bR.
$$
Thus, every tangent vector \(u\in T_{(p,r)}\sN\) admits a unique decomposition
$$
u(p,r)= v(p,r) + w(p,r)\partial_r \cong (v(p,r), w(p,r)),
$$
where $v(p,r)\in T_p\Sigma$ and $w(p,r)\in \bR$. We write in local coordinates
\begin{equation}
\label{tangent-local}
u (p,r)= v^i(p,r) \tau_i + w(p,r)  \tau_3 =  v^i(p,r) \partial_i + w(p,r)  \partial_r .
\end{equation}
Since the metric tensor is block diagonal,
$$
[g_{\sN}]=
\begin{bmatrix}
g_{ij} & 0\\
0 & 1
\end{bmatrix},
$$
a direct computation shows that the  Christoffel symbols of $\sN$ satisfy
\begin{equation}
\label{Chris-N}
\Gamma^k_{ij}(\sN)=\Gamma^k_{ij}(\Sigma),\qquad
\Gamma^3_{ij}=\Gamma^k_{i3}=\Gamma^3_{i3}=\Gamma^k_{33}=0.
\end{equation}
Hence, the only nonvanishing Christoffel symbols are those inherited from $(\Sigma,g_\Sigma)$;
all Christoffel symbols involving the $r$-direction vanish.
Let
$$u_i = v_i + w_i \partial_r \cong (v_i,w_i), \qquad i=1,2,
$$
be   vector fields on \(\sN\) as in~\eqref{tangent-local}, where
$$
v_i=v_i(p,r)\in T_p\Sigma,
\quad
w_i=w_i(p,r)\in\mathbb R.
$$
Then employing \eqref{Chris-N}  one shows that the Levi--Civita connection $\nabla^\sN$  of  $\sN$ is given by
\begin{equation}
\label{covariant-N}
\nabla^{\sN}_{(v_1+w_1\partial_r)}(v_2+w_2\partial_r)
=\nabla_{v_1} v_2+w_1\partial_r v_2 + \big(v_1(w_2)+w_1\partial_r w_2\big)\partial_r,
\end{equation}
or equivalently,
\begin{equation*}
\nabla^{\sN}_{(v_1,w_1)}(v_2,w_2)
= \bigl( \nabla_{v_1}v_2+w_1\partial_r v_2,\, v_1(w_2)+w_1\partial_r w_2 \bigr).
\end{equation*}
Here, $\nabla$ stands for the covariant derivative on $\Sigma$, and
$v(w)$ denotes the directional derivative of $w$ in the direction of $v$, given in local coordinates by
$v(w)= v^j\partial_j w. $
In particular, is follows from~\eqref{covariant-N} that
\begin{equation*}
\nabla^\sN_v v= \nabla_v v, \quad
\nabla^\sN_{\partial_r}(v+w\partial_r)=\partial_r v + \partial_rw \partial_r, \quad
\nabla^\sN_{\partial_r} \partial_r =0, \quad
\nabla^\sN_v \partial_r =0.
\end{equation*}
Suppose $u,v: \sN\to T\Sigma$ are tangential to $\Sigma$. Then it follows readily from~\eqref {(1,1)-inner} and~\eqref{Chris-N}   that
\begin{equation}
\label{splitting-2}
\langle \nabla^N u,\nabla^N v\rangle_{g_\sN}
= \langle \nabla u,\nabla  v\rangle_g + \langle \partial_r u,\partial_r v\rangle_g.
\end{equation}

\noindent
Let
$$
\Delta_{\sN}:=\operatorname{tr}\bigl((\nabla^\sN)^2\bigr)
$$
be the connection Laplacian associated with the connection $\nabla^\sN$.
Let
$ u=v+w\,\partial_r, $
where
$$
v=v(p,r)\in T_p\Sigma, \qquad w=w(p,r)\in\bR.
$$
Since  the only nonvanishing Christoffel symbols are those inherited from $(\Sigma,g_\Sigma)$,
 $\Delta_\sN$ decomposes into its tangential and vertical parts.
More precisely, one obtains
\begin{equation*}
\Delta_{\sN}u =\bigl( \Delta_\Sigma v+\partial_r^2 v \bigr)
+ \bigl( \Delta_B w+\partial_r^2 w \bigr)\partial_r.
\end{equation*}
Here,
\begin{itemize}
\item $\Delta_\Sigma$ denotes the connection Laplacian on \((\Sigma,g_\Sigma)\), acting on  tangential vector fields  $v(\cdot,r)$ for each fixed $r$,
\vspace{1mm}
\item $\Delta_B$ denotes the Laplace--Beltrami operator on \((\Sigma,g_\Sigma)\), acting on scalar functions $w(\cdot,r)$,
\vspace{-2mm}
\item $\partial_r^2$ denotes the second derivative with respect to the vertical variable \(r\).
\end{itemize}
Let $ f :\sN \to \bR$  be a an integrable  function. Then the integral of \( f \) over the product manifold is computed using Fubini's theorem:
\begin{equation*}
\int_{\sN} f(p, r) \, d\mu_{g_\sN}
= \int_\Sigma \left( \int^0_{-h} f(p, r) \, dr \right)\,d\mu_g
= \int^0_{-h} \left( \int_\Sigma f(p, q) \,d\mu_g \right)dr,
\end{equation*}
where
$d\mu_{g_\sN} $ and $d\mu_g$ denote the Riemannian volume element of $(\sN, g_\sN)$ and $(\Sigma, g)$, respectively.

\bigskip\noindent
We are now ready to state the divergence theorem for $\sN$.
For this, we first observe that the boundary $\partial \sN$ of $\sN$ is given by
$$
\partial \sN = (\Sigma \times\{-h\}) \cup (\Sigma \times\{0\}) = \Sigma_b \cup \Sigma_u.
$$
The outward unit normal vector field $\nu=\nu_{\partial\sN} $ on the boundary $\partial \sN$ is given as follows.

\smallskip
(i) On the upper boundary $\Sigma_u$,  \quad$\nu = \partial_r$.

\smallskip
(ii) On the lower boundary $\Sigma_u$, \quad $\nu = -\partial_r.$

\medskip\noindent
Suppose $\phi$ is a scalar function defined on $\sN$.
Then we  have the following integration by parts formula (the divergence theorem):
\begin{equation*}
\begin{aligned}
 \int_{\sN} (\dv_{\sN} u)\phi \,d\mu_{g_\sN} &= -\int_{\sN} \la  u,\gd_{\sN}\phi\ra_{g_\sN} \, d\mu_{g_{\sN}}
 + \int_{\partial \sN} \la u , \nu \ra_{g_\sN} \phi\,  dS  \\
& = - \int_{\sN} \nabla^\sN_u \phi\,d\mu_{g_\sN}  +  \int_\Sigma \Big[\la u , \nu \ra_{g_\sN} \phi \Big]^{r=0}_{r=-h}\,  d\mu_g.
\end{aligned}
\end{equation*}
For the closed manifold $\Sigma$, the divergence theorem states
\begin{equation}
\label{div-scalar-Sigma}
\begin{aligned}
 \int_{\Sigma} (\dv_{\Sigma} u)\phi \,d\mu_{g} = -\int_{\Sigma} \la  u,\gd_\Sigma\phi\ra_g \, d\mu_{g}
 = - \int_{\Sigma} \nabla_u \phi\,d\mu_g .
\end{aligned}
\end{equation}
We are now ready to state and prove the following result which is used repeatedly in Section~\ref{S: global existence}.
\begin{proposition}
\label{pro: divergence-free}
Let $q\in (1,\infty)$,   $z\in H^1_2(\sN; T\Sigma)\cap L_{2q-2}(\sN; T\Sigma)$ and suppose that  $(v,w)$ satisfy \eqref{PE}.
 Then
\begin{equation*}
\int_\sN \la  \nabla_v z + w \partial_r z  , |z|^{q-2} z \ra_g \, d\mu_{g_\sN} = 0
\end{equation*}
and
\begin{equation*}
\int_\sN \la  \nabla_{\overline{v}} z  , |z|^{q-2} z \ra_g \, d\mu_{g_\sN} = 0.
\end{equation*}
\end{proposition}
\begin{proof}
We remind here that  $|z|=|z|_g= \sqrt{ \la z, z\ra_g} $.
By the metric compatibility of the Levi-Civita connection,
\begin{equation*}
\begin{aligned}
\la \nabla_v z, |z|^{q-2}z\ra_g
= \nabla_v \la z, |z|^{q-2} z\ra_g - \la z, \nabla_v (|z|^{q-2} z)\ra_g
=\nabla_v |z|^q - \la z, \nabla_v (|z|^{q-2} z)\ra_g . \\
\end{aligned}
\end{equation*}
A short computation, using properties of the covariant derivative and, once more, the metric compatibility of the Levi-Civita connection, yields
$$
\la z, \nabla_v (|z|^{q-2} z)\ra_g = (q-1) \la \nabla_v z, |z|^{q-2} z\ra_g.
$$
Combining the above identities, we obtain
$$
q\, \la \nabla_v z, |z|^{q-2}z\ra_g =\nabla_v |z|^q  = \la \gd_\Sigma |z|^q, v \ra_g.
$$
Employing the divergence theorem~\eqref{div-scalar-Sigma} for $\Sigma$ yields
\begin{equation}
\label{AB}
\begin{aligned}
\int_\sN \la\nabla_v  z  , z |z |^{q-2}  \ra_g \, d\mu_{g_\sN} & = \frac{1}{q}  \int_{-h}^0 \int_{\Sigma} \la \wgd |z|^q , v \ra_g \, d\mu_g \, dr  \\
&= - \frac{1}{q}  \int_{-h}^0 \int_{\Sigma}   |z|^q\,  \wdv v   \, d\mu_g \, dr.
\end{aligned}
\end{equation}
By similar arguments as above and employing the boundary conditions $w(\cdot, -h)=w(\cdot, 0)=0$ we obtain
\begin{align*}
\int_\sN  \la  w \partial_r z  , |z|^{q-2} z \ra_g \, d\mu_{g_\sN} &= \frac{1}{q}  \int_{-h}^0 \int_{\Sigma} w \partial_r |z|^q   \, d\mu_g \, dr  \\
&=  - \frac{1}{q}  \int_{-h}^0 \int_\Sigma  \partial_r w |z|^q   \, d\mu_g \, dr  = \frac{1}{q}  \int_{-h}^0 \int_{\Sigma}   |z|^q  \wdv v   \, d\mu_g \, dr .
\end{align*}
For the last assertion, we used the identity
$$
\partial_r w+\operatorname{div}_\Sigma v=0,
$$
which follows from \(\eqref{PE}_3\). This completes the proof of the first part of the lemma.

\smallskip
For the second part, we apply formula \eqref{AB} with \(v\) replaced by \(\bar v\). Since
$\dv_\Sigma \bar v=0, $
by $\eqref{PE}_2$, the desired conclusion follows at once.
\end{proof}
The integration by parts formula (Green's identity) for the connection Laplacian is
\begin{equation}
\label{Green-1}
\int_\sN \langle \Delta_{\sN} u, v \rangle_{g_\sN} \, d\mu_{g_\sN}
=-\int_\sN \langle \nabla^\sN u, \nabla^\sN v \rangle_{g_\sN} \, d\mu_{g_\sN} + \int_{\partial N} \langle \nabla_\nu u, v \rangle_{g_\sN} \, dS,
\end{equation}
where $\nu$ is the outward unit normal, see for instance  \cite[Lemma~B.1]{SSW25}  for a proof.
Therefore,
\begin{equation}
\label{Green-2}
\int_\sN \langle \Delta_{\sN} u, v \rangle_{g_\sN} \, d\mu_{g_\sN}
= -  \int_\sN \langle \nabla^\sN u, \nabla^\sN v \rangle_{g_\sN} \, d\mu_{g_\sN}
 + \int_\Sigma \Big[ \la \partial_r u, v \rangle \Big]^{r=0}_{r=-h} \, d\mu_g .
\end{equation}

\medskip\noindent
For the closed manifold $\Sigma$, we have for $u(p,r), v(p,r) \in T_{p}\Sigma $
\begin{equation}
\label{Green-5}
\int_\Sigma \langle \Delta_\Sigma u, v \rangle_{g} \, d\mu_{g}
= \int_\Sigma  \langle \nabla u, \nabla  v \rangle_g\, d\mu_g .
\end{equation}

\begin{lemma}[Poincar\'e inequality]
\label{lem: Poincare}
Let \(1\le q<\infty\) and suppose that
$ u\in W^{1,q}(\sN; T\sN)$
satisfies  $u=0$ on $\Sigma_b$ or $\Sigma_u$.
Then
$$
\|u\|_{L^q(\sN)} \le h\,\|\partial_r u\|_{L^q(\sN)},\quad\text{or equivalently}\quad \|u\|_{L^q(\sN)} \le h\,\|\nabla^\sN_{\partial_r}u\|_{L^q(\sN)}.
$$
Hence,
$$
\|u\|_{L^q(\sN)} \le h\| \nabla^\sN u\| _{L^q(\sN)}.
$$
\end{lemma}
\begin{proof}
 We  will only consider the case  $u=0$ on $\Sigma_b$. The proof of the case  $u=0$ on $\Sigma_u$  is entirely analogous.
As $u(\cdot,-h)=0$, the fundamental theorem of calculus yields
$$
u(p,r)= \int_{-h}^{r}\partial_su(p,s)\,ds
$$
for every \((p,r)\in \sN\). Consequently,
$$
|u(p,r)|_{g_\sN} \le \int_{-h}^{0} |\partial_su(p,s)|_{g_\sN}\,ds.
$$
By H\"older's inequality,
$$
|u(p,r)|_{g_\sN}^q \le \left( \int_{-h}^{0} |\partial_su(p,s)|_{g_\sN}\,ds \right)^q \le h^{q-1} \int_{-h}^{0} |\partial_su(p,s)|_{g_\sN}^q\,ds.
$$
Integrating with respect to \(r\in[-h,0]\), we obtain
$$
\int_{-h}^{0} |u(p,r)|_{g_\sN}^q\,dr \le h^q \int_{-h}^{0} |\partial_ru(p,r)|_{g_\sN}^q\,dr.
$$
Integrating over \(\Sigma\) yields
$$
\int_\Sigma\int_{-h}^{0} |u(p,r)|_{g_\sN}^q\,dr\,d\mu_g \le h^q \int_\Sigma\int_{-h}^{0} |\partial_ru(p,r)|_{g_\sN}^q\,dr\,d\mu_g.
$$
Hence
$$
\|u\|_{L^q(\sN)} \le h\,\|\partial_ru\|_{L^q(\sN)}.
$$
Finally, since $\sN$ is a product Riemannian manifold, we have $\nabla^\sN_{\partial_r}u = \partial_r u$.
\end{proof}
As an immediate consequence of Poincar\'e's inequality we can show invertibility of the connection Laplacian $\Delta_\sN$.
\begin{lemma}
\label{lem: N-invertible}
Let $D(\Delta_\sN )=\{v  \in H^2_{q} (\sN;T\Sigma) : \, v=0 \text{ on }\Sigma_b, \ \ \partial_r v=0 \text{ on }\Sigma_u\}.$ \\
Then
$$\Delta_\sN \in \Lis (D( \Delta_\sN), L_q(\sN; T\Sigma)).$$
In particular, there exists a number $C>0$ such that
$$
\|  v \|_{H^2_q(\sN)} \le C \| \Delta_\sN v \|_{L_q(\sN)},\quad v\in D(\Delta_\sN).
$$

\end{lemma}
\begin{proof}
Since $\Delta_{\sN}$ has compact resolvent, it suffices to show that $0$ is not an eigenvalue.
Suppose, to the contrary, that $v\in D(\Delta_\sN)\setminus\{0\}$ and $\Delta_\sN v=0$.
It then follows from \eqref{Green-2} that
$$
0 = \int_\sN \langle \Delta_{\sN} v, v \rangle_{g_\sN} \, d\mu_{g_\sN}
= -  \int_\sN | \nabla^\sN v |^2_{g_\sN} d\mu_{g_\sN}.
$$
Hence $\nabla^\sN v=0$, and by Lemma~\ref{lem: Poincare} we conclude that $v=0$, leading to a contradiction.
\end{proof}
\begin{lemma}
\label{lem:interpolation}
Let $p,q,r\in (1,\infty)$ satisfying $\frac{1}{p}\left( \frac{1}{2}+ \frac{1}{q} \right) \geq \frac{1}{r}$.
Then there exists some constant $C>0$ such that
\begin{align*}
\| |z|^p \|_{L_q(\Sigma)}\leq C \left( \|  z  \|_{L_r(\Sigma)}^p +  \|  z  \|_{L_r(\Sigma)}^{p-r/2}  \| \nabla | z|^{r/2}  \|_{L_2(\Sigma)} \right),
\end{align*}
provided the terms on the right hand side are finite.
\end{lemma}
\begin{proof}
The proof follows exactly the same argument in \cite[Lemma~6.3(b)]{HiKa16}.
The corresponding embedding and interpolation inequalities on manifolds has been obtained in  \cite{Ama13} and \cite{AubinBook}.
\end{proof}

The following Gagliardo-Nirenberge inequalities have been utilized in Section~\ref{S: global existence}.
They can be  proved by using the Poincar\'e inequality, see Lemma~\ref{lem: Poincare}, the Sobolev embedding theorem and interpolation theory, see for instance \cite[Theorems 13.1 and 14.2]{Ama13}.
\begin{equation}
\label{G-N ineq}
\begin{split}
\| u \|_{L_3(\sN)} & \leq C \| u \|_{L_2(\sN)}^{1/2} \| \nabla^\sN u\|_{L_2(\sN)}^{1/2}   \\
\| u \|_{L_4(\sN)} & \leq C \| u \|_{L_2(\sN)}^{1/4} \| \nabla^\sN u\|_{L_2(\sN)}^{3/4}   \\
\| u \|_{L_4(\sN)} & \leq C \| u \|_{L_3(\sN)}^{1/2} \| \nabla^\sN u\|_{L_2(\sN)}^{1/2}   \\
\| u \|_{L_4(\Sigma)} & \leq C \| u \|_{L_2(\Sigma)}^{1/2} \| \nabla u\|_{L_2(\Sigma)}^{1/2} + \| u \|_{L_2(\Sigma)} \\
\| u \|_{L_3(\Sigma)} & \leq C \| u \|_{L_2(\Sigma)}^{2/3} \| \nabla u\|_{L_2(\Sigma)}^{1/3} + \| u \|_{L_2(\Sigma)} \\
\| u \|_{L_6(\Sigma)} & \leq C \| u \|_{L_2(\Sigma)}^{1/3} \| \nabla u\|_{L_2(\Sigma)}^{2/3}  + \| u \|_{L_2(\Sigma)}.
\end{split}
\end{equation}
The first three inequalities hold for all $u\in H^1_2(\sN; T\sN)$ such that $u=0$ on $\Sigma_b$ or $\Sigma_u$.
\begin{proposition}\label{Prop:trace}
Suppose $u\in H^1_2(\sN; T\sN)$ and $u=0$ on $\Sigma_u$.
There exists a constant $C$ such that
$$
\| u \|_{L_2(\Sigma_b)} \leq C \|u\|_{L_2(\sN)}^{1/2}\| \nabla^\sN u\|_{L_2(\sN)}^{1/2}.
$$
\end{proposition}
\begin{proof}
Let $\nu$ be  the unit outward pointing normal of $\Sigma_b$ in the ambient  manifold $\sN$.
Note that  -$\nu$ can be extended to a vector field $V\in BC^1(\sN; T\sN)$ such that $V=0$ on $\Sigma_u$.
 By the divergence theorem  \eqref{Green-1},
\begin{align*}
\int_{\sN} \dv_{\sN} \left( |u|^2 V  \right) \, d\mu_{g_\sN} & =  -\int_\Sigma \left(|u|^2 \la V, \nu\ra_{g_\sN}\right)\Big|_{r=-h} 
 = \int_{\Sigma_b} |u|^2 \, d\mu_g  = \|u\|_{L_2(\Sigma_b )}^2,
\end{align*}
where $|u|= |u|_{g_\sN}$.
On the other hand,
\begin{align*}
\int_{\sN} \dv_{\sN} \left( |u|^2 V  \right) \, d\mu_{g_\sN} &
= 2 \int_{\sN} \la  \nabla_V^{\sN} u , u \ra_{g_\sN} \, d\mu_{g_\sN} +  \int_{\sN}   |u|^2 \,\dv_{\sN} V   \, d\mu_{g_\sN} \\
&\leq C \| u \|_{L_2(\sN)} \|  \nabla^{\sN}  u \|_{L_2(\sN)}  + C \| u \|_{L_2(\sN)}^2 \\
& = C \| u \|_{L_2(\sN)} \| \nabla^{\sN}  u \|_{L_2(\sN)}  + C \| u \|_{L_2(\sN)} \| u \|_{L_2(\sN)} \\
&\leq C \| u \|_{L_2(\sN)} \| \nabla^{\sN}  u \|_{L_2(\sN)}
\end{align*}
in virtue of  Lemma~\ref{lem: Poincare}. This completes the proof.
\end{proof}

\section{Derivation of the primitive equations}\label{S: Appendix B}
In this section, we derive the primitive equations~\eqref{PE}.

\medskip\noindent
Suppose  $\Sigma$ is a smooth and closed hypersurface  in  $\bR^3$,
and let $S$ be a collar neighborhood of $\Sigma$ given by
$$
S=
\left\{x \in \bR^3 :x = \p + r \nu_\Sigma(p),
\quad\p \in \Sigma,\quad -h < r < 0\right\},
$$
where $h>0$ is small and $\nu_\Sigma$ denotes the outward pointing unit normal field  on $\Sigma$.
We assume that $S$ is occupied by an incompressible viscous fluid, whose motion is modeled
by the  incompressible (three dimensional Euclidean) Navier-Stokes equations,
\begin{equation}
\label{N-S eq}
\left\{\begin{aligned}
\varrho \left( \partial_t u + u\cdot\nabla u \right) - \mu_s \Delta u + \nabla  \pi  & =f &&\text{in}&&S ,\\
\dv u &=0 &&\text{in}  &&S , \\
\end{aligned}\right.
\end{equation}
where $\varrho$ and $\mu_s$ are the fluid density and viscosity, respectively.
Hence, we think of an ocean of water of constant depth $h$ that occupies the region $S$ 
 and is subject to an external force $f$. 

After introducing canonical normal coordinates for  $S$, we rewrite the Navier--Stokes equations with respect to the
  \emph{collar metric}. In the course of this reformulation, we perform approximations that exploit the smallness of the parameter $h$.

\medskip
In the following, we consider $(\Sigma,g)$ as a Riemannian manifold, where $g=g_\Sigma$ is the Riemannian metric induced on $\Sigma$ by the ambient Euclidean metric
$g_{\bR^3}$ of $\bR^3$.
Let
\begin{equation*}
\Phi: \Sigma\times (-h,0) \to \bR^3,\quad \Phi(p,r) := p+ r\nu_\Sigma(p). \\
\end{equation*}
If $h$ is sufficiently small, it is well-known, see for instance  \cite[Section 2.3]{PrSi16}, that there exist  two smooth mappings
\begin{equation}
\label{P-d}
\Pi_\Sigma: S \to \Sigma\quad \text{and} \quad \sd_\Sigma: S \to (-h,0)
\end{equation}
such that for $x\in S$, $\Pi_\Sigma(x)  $ is the nearest point to $x$ on $\Sigma$
and $\sd_\Sigma(x) $ is the signed distance of $x$ to $\Sigma$.
With this, we have
$$
S= {\rm Im}(\Phi)= \{x \in \bR^3 : -h < \sd_\Sigma (x) < 0 \}.
$$
Hence,
\begin{equation*}
\Phi:  \sM \to S,\qquad \sM = \Sigma \times (-h,0),
\end{equation*}
is a smooth diffeomorphism.

\medskip
Let $\{U_\kappa : 1\le \kappa\le N\}$ be an open cover for $\Sigma$ and let
$$
\varphi_\kappa : B_r(0)\subset \bR^2 \to U_\kappa,\quad \kappa =1,\ldots, N,
$$
be a local parameterization of $U_\kappa$.
For the rest of this article, we will omit the subscript $\kappa$ when the choice of the local patch is immaterial.
In local coordinates, the  metric $g$ of $\Sigma$ is given by
$$
[g]_{ij}= [\la\partial_i \varphi, \partial_j \varphi\ra _{g_{\bR^3}}],\quad 1\le i,j\le 2,
$$
where  $\la \cdot\, , \cdot\ra_{g_{\bR^3}}$ denotes the Euclidean inner product.
A basis for the tangent space $T\Sigma $ of $\Sigma$ can be defined by
$$
\tau_i =\frac{\partial}{\partial x_i}  = \partial_i= \partial_i \varphi, \quad i=1,2.
$$
For the dual basis  we use the notation $\{\tau^i\}_{i=1}^2$.
Hence, $\tau^i= dx^i$, where $dx^i$ is the notation commonly used in differential geometry.
We recall that
\begin{equation}
\label{diff-nu}
\partial_i \nu_\Sigma = -L_\Sigma \tau_i = - l^k_i \tau_k,
\end{equation}
where $L_\Sigma=l^i_j \tau_i \otimes \tau^j$ is the  Weingarten tensor of $\Sigma$.
In the following, we use the convention
\begin{equation}
\label{lij}
l_{ij}= g_{ik} l^k_j,\quad  l^{i}_j  = g^{ik}l_{jk}.
\end{equation}
The (two-fold) mean curvature of $\Sigma$ is defined by
\begin{equation*}
H_\Sigma = {\rm tr} L_\Sigma .
\end{equation*}
The gradient,  the divergence, the Levi-Civita connection, and  the connection Laplacian  on $(\Sigma, g)$ are denoted by
$$
\gd_\Sigma, \   \dv_\Sigma, \ \nabla=\nabla^\Sigma,\   \Delta_\Sigma,
$$
 respectively.
Next we observe that
\begin{equation}
\label{phi-parameterization}
\phi_\kappa : B_r(0) \times (-h,0)\to \bR^3,\quad  \phi_\kappa(x_1,x_2,r)=\varphi_\kappa(x_1,x_2)+ r \nu_{\Sigma}(\varphi_\kappa(x_1,x_2)),
\ \kappa=1,\ldots, N,
\end{equation}
is a parameterization of the collar patch $\Phi( U_\kappa \times (-h,0))$,
 the parametrization by \emph{canonical normal coordinates}.
Setting
$$
\boldsymbol{\tau}_i=\boldsymbol{\tau}_i(\cdot, r):=\partial_i \phi,    \quad \boldsymbol{\tau}_3:= \nu_\Sigma,
$$
we obtain the \emph{collar metric}
$$
g_\sM= [\la \boldsymbol{\tau}_\alpha, \boldsymbol{\tau}_\beta\ra_{g_{\bR^3 }}],\quad 1\le \alpha, \beta \le 3.
$$
 As before, we omit the subscript $\kappa$ whenever the choice of a local patch is immaterial.
In the following, we use the Einstein summation convention, indicating that terms with repeated indices are added.
Lowercase greek indices vary from 1 to 3, latin indices vary from~1~to~2.

Setting $\sM = \Sigma\times (-h,0)$,
$$
 ( \sM, g_{\sM}) \quad\text{becomes a Riemannian manifold},
$$
and
$\{\boldsymbol{\tau}_\alpha\}_{\alpha=1}^3$
forms a basis  of  the tangent space of $\sM$  at a point
 $z=\phi_\kappa(x_1,x_2,r)\in\sM$, with corresponding dual basis $\{\boldsymbol{\tau} ^\alpha \}_{\alpha=1}^3$.
In the following, we suppress the coordinates $(x_1,x_2)$ in our notation, while still keeping~$r$.
With \eqref{diff-nu}, we obtain
$$
\boldsymbol{\tau}_i =  \partial_i \phi = \tau_i -r l^j_i \tau_j .
$$
Employing  \eqref{lij},
we then obtain for  the metric of $\sM$ the expression
 \begin{equation*}
 [g_\sM (r)]= \begin{bmatrix}
&\!\![g_{ij} + d_{ij}(r)] & 0 \\
&0  & 1  \\
\end{bmatrix} ,
\quad\text{with}\quad d_{ij}(r)= -2r l_{ij} +r^2 l_{ik} l^k_j , \quad i,j\in\{1,2\}.
\end{equation*}
It is important to notice that the Christoffel symbols
$\Gamma^\gamma_{\alpha\beta}(r)=(\Gamma_\sM)^\gamma_{\alpha\beta}(r)$ of $\sM$ satisfy
\begin{equation}
\label{Chris-1}
\begin{aligned}
\Gamma^3_{ij}(r) &= l_{ij} -r l_{ik} l^k_j,     \quad && \Gamma^3_{\alpha 3}(r)=0, \\
\Gamma^i_{3j}(r) &=g^{i\ell}(r) (-l_{j \ell }+ r l_{\ell k} l^k_j) , \quad &&\Gamma^i_{33}(r)=0 ,\\
\Gamma^m_{ij}(r) &= (\Gamma_\Sigma)^m_{ij} + \Lambda^m_{ij}(r), \quad &&\text{with } \Lambda^m_{ij}(0)=0
\end{aligned}
\end{equation}
for $i,j,m\in\{1,2\}$, where $(\Gamma_\Sigma)^m_{ij}$ are the Christoffel symbols of $\Sigma$.
Note that
\begin{equation*}
\Gamma^m_{ij} (0) = (\Gamma_\Sigma)^m_{ij},\quad i,j,m\in\{1,2\} .
\end{equation*}
Because the vertical scale of the ocean is much smaller than its horizontal scale, it is sometimes reasonable to set $r=0$
for certain geometric quantities
in the model derivation.
In the sequel, $A(r) \approx B$ means that we  set $r=0$ in some components of $A(r)$ to obtain an approximation $B$.
For instance, we have  $\Gamma^m_{ij} (r)\approx \Gamma^m_{ij} (0)=(\Gamma_\Sigma)^m_{ij}$
for $i,j,m\in\{1,2\}$.

\medskip\noindent
The gradient, the divergence,  the Levi-Civita connection, and the connection Laplace  on $(\sM, g_\sM)$ are denoted by
 $$
 \gd_\sM,\quad \dv_\sM \quad   \nabla^\sM,\quad    \text{and} \quad  \Delta_\sM,
 $$
 respectively. Now we observe that
 \begin{equation*}
 \Phi : (\sM, g_{\sM})  \to (S, g_{\bR^3}) \quad\text{is an isometry},
 \end{equation*}
 as  $g_\sM$ is the pull-back metric  $\Phi^*(g_{\bR^3})$  of the Euclidean metric $g_{\bR^3}$.
Since the Levi–Civita connection is preserved under isometries, we can conclude that
\begin{equation}
\label{quantiities-collar}
 \gd_\sM = \Phi^* \gd_{\bR^3},\quad \dv_\sM = \Phi^*\dv_{\bR^3},
 \quad \nabla^\sM =\Phi^* \nabla^{\bR^3},
 \quad \Delta_{\sM}= \Phi^* \Delta_{\bR^n},
\end{equation}
where $\Phi^*$ denotes the pullback  by $\Phi$. In~\eqref{quantiities-collar}, we use slightly imprecise language
as the mathematically correct statements would, for instance,  read $\gd_\sM = \Phi^* (\gd_{\bR^3}\circ \Phi_*)$.

\medskip\noindent
Employing~\eqref{quantiities-collar}, we  are now ready to reformulate the  Navier-Stokes equations \eqref{N-S eq} by means of the collar metric $g_{\sM}$.
We use the orthogonal decomposition
$$
u(p,r) = v(p,r) + w(p,r) \nu_\Sigma(p) \in T_{p}\Sigma \oplus T^\perp_p\Sigma, \quad (p,r)\in \Sigma\times (0,h),
$$
and for short
$$
u(p,r)=(v(p,r), w(p,r) )\in T_p\Sigma \times \bR .
$$
Hence, $v$ is  the tangential component and $w$ normal component of for the velocity field $u: S\to \bR^3$, respectively.
Here we use, in slight abuse of notation, the same symbol for $u$ and  $u\circ \Phi$.
We have in local coordinates
\begin{equation*}
u(\phi(x_1,x_2,r)),\quad (v(\phi(x_1,x_2,r), w(\phi(x_1,x_2,r))),\quad (x_1,x_2,r)\in B_r(0)\times (-h,0),
\end{equation*}
where $\phi$ is defined in~\eqref{phi-parameterization}.
Moreover,  we write in local coordinates

\begin{equation}
\label{u-basis}
u  = \sum_{\alpha =1}^3 u^\alpha \boldsymbol{\tau}_\alpha = \sum_{i=1}^2 v^i \boldsymbol{\tau}_i + w  \nu_\Sigma.
\end{equation}
In the collar metric, the transport term  $\Phi^* (u\cdot \nabla u) = \Phi^*(\nabla^{\bR^3}_u u)$ is given by
$$
 \nabla^{\sM}_u u = u^\alpha \nabla^\sM_\alpha u = u^j \nabla^\sM_j u+ u^3\nabla^\sM_3 u = v^j \nabla^{\sM}_j u + w\nabla^{\sM}_3 u,
$$
where $\nabla^\sM_\alpha = \nabla^\sM_{\boldsymbol{\tau}_\alpha}$.
Employing  \eqref{u-basis} and \eqref{Chris-1}, a straightforward calculation yields
\begin{equation*}
\begin{aligned}
(\nabla^\sM_u u)
&= v^j  \big({\partial_j} v^i + v^k\Gamma^i_{jk}(r) +w \Gamma^i_{j3}(r) \big)\boldsymbol{\tau}_i + w \big( \partial_r v^i  + v^k\Gamma^i_{k3} (r) \big)\boldsymbol{\tau}_i \\
&\qquad + \big( v^i v^j \Gamma^3_{ij}(r)  + v^j\partial_j w + w\partial_3 w\big)\boldsymbol{\tau}_3.
\end{aligned}
\end{equation*}
The tangential part is then given by
\begin{equation}
\label{transport-M}
\begin{aligned}
(\nabla^\sM_u u)_{\rm tan }
&= v^j  \big({\partial_j} v^i + v^k\Gamma^i_{jk}(r) +w \Gamma^i_{j3}(r) \big)\boldsymbol{\tau}_i + w \big( \partial_r v^i  + v^k\Gamma^i_{k3} (r) \big)\boldsymbol{\tau}_i \\
& \approx v^j \big({\partial_j} v^i +\Gamma^i_{jk}(0) v^k)\tau_i + w( \partial_r v^i  + 2 v^j  \Gamma^i_{j3}(0)  \big){\tau}_i \\
& = \nabla_v v  + w\partial_r v - 2w L_\Sigma v .
\end{aligned}
\end{equation}
We point out that the partial derivatives $\partial_j u$ are taken with respect to the parameterization $\phi$ in ~\eqref{phi-parameterization}.
In case $r=0$, $\partial_j u$ coincides with the partial derivatives taken with respect to the parameterization $\varphi$.
We also note that $\boldsymbol{\tau}_j = \tau_j$ in case $r=0$.

For the last assertion in~\eqref{transport-M}, we have used \eqref{Chris-1} to conclude that
$$
2w \Gamma^i_{j3}(0)v^j \tau_i  = -2w l^i_j v^j \tau_i = -2w L_\Sigma (v^j \tau_j)= -2w L_\Sigma v.
$$
Taking  once more \eqref{Chris-1} into account, the divergence $\dv_\sM  \Phi^*(\dv_{\bR^3})$ is given by
\begin{equation}
\label{div-M}
\begin{aligned}
 \dv_\sM  u = u^\alpha_{| \alpha}
&= {\partial}_j v^j + \Gamma^j_{jk}(r) v^k + \Gamma^j_{j3}(r) w + \partial_r w + \Gamma^3_{3j}(r) v^j     \\
     &={\partial}_j v^j + \Gamma^j_{jk}(r) v^k + \Gamma^j_{j3}(r) w + \partial_r w   \\
     &\approx \partial_j v^j + \Gamma^j_{jk}(0) v^k  + \Gamma^j_{j3}(0) w +\partial_r w \\
     & = {\rm div}_\Sigma v - l^j_j w +\partial_r w \\
     & =  {\rm div}_\Sigma v - H_\Sigma w +\partial_r w.
\end{aligned}
\end{equation}
The connection Laplacian $\Delta_\sM =\Phi^* (\Delta_{\bR^3})$ is given in local coordinates by
\begin{equation}
\label{Bochner-greek}
\Delta_{\sM} u
= g^{\alpha \beta}_\sM(\nabla^\sM_\alpha \nabla^\sM_\beta -\Gamma^\gamma_{\alpha\beta}\nabla^\sM_\gamma)u
= g_\sM ^{\alpha \beta}(\nabla^\sM_\alpha \nabla^\sM_\beta -\Gamma^\gamma_{\alpha\beta}\nabla^\sM_\gamma)u,
\end{equation}
where $\nabla^\sM$ is the Levi-Civita connection for $(\sM, g_\sM)$, and
$\nabla^\sM_\alpha = \nabla^\sM_{\boldsymbol{\tau}_\alpha}$.

\medskip\noindent
In the sequel, it turns out to be less cumbersome to use latin letters for the local expression of $\Delta_\sM u$
in \eqref{Bochner-greek}.
With this agreement, we obtain  in local coordinates
 \begin{equation}
 \label{Bochner-latin}
 \begin{aligned}
 \Delta_{\sM} u
               &= g^{jk}_\sM(r) \Big({\partial}_j{\partial}_k u^i
               + {\partial}_k (\Gamma^i_{j\ell}(r) u^\ell)
               + \Gamma^i_{mk}(r)({\partial}_j u^m + \Gamma^m_{j\ell} (r) u^\ell)  \\
               &\hspace{5cm}- \Gamma^m_{jk}(r)({\partial}_m u^i + \Gamma^i_{m\ell}(r) u^\ell)\Big) \boldsymbol{\tau}_i\,. \\
               \end{aligned}
 \end{equation}
 Note that now $i,j,k,\ell,m\in \{1,2,3\}$.
Additionally, for notational brevity, we set $g^{ij}(r)= g^{ij}_\sM (r)$ and $g_{ij}(r)= g^{\sM}_{ij}(r)$.

In the following, we write $\tilde \Delta_\sM $ in case only tangential derivatives $\partial_j$  are considered in \eqref{Bochner-latin}.
 This means  that we consider the case where $j,k\in \{1,2\}$ in \eqref{Bochner-latin}.
 Moreover, we write $(\Delta_\sM)_r u $ for the case where only normal derivatives are involved, that is, when $j,k=3$ in  $\eqref{Bochner-latin}.$

\medskip\noindent
 We start by considering the {tangential terms} of $\tilde \Delta_\sM u$ in \eqref{Bochner-latin}, corresponding to $i\in\{1,2\}$
 and $j,k\in\{1,2\}$.

\medskip\noindent
Noting that $\Gamma^m_{33}(r)=\Gamma^3_{m3}(r)=0$ for $m\in \{1,2,3\}$
and employing \eqref{u-basis},
we get
\begin{equation}
\label{tangential}
\begin{aligned}
&(\tilde \Delta_\sM u)_{\rm tan} =\\
&\qquad   g^{jk}(r)
\left[
\begin{aligned}
&({\partial}_j{\partial}_k v^i +{\partial}_k(\Gamma^i_{j\ell}(r) v^\ell) + \  \Gamma^i_{j3}(r){\partial}_k w
 +  ( {\partial}_k\Gamma^i_{j3}(r)) w  \\
& +\Gamma^i_{mk}(r)({\partial}_j v^m +\Gamma^m_{j\ell} (r) v^\ell)  +  \Gamma^i_{mk}(r)\Gamma^m_{j3}(r) w
		+ \Gamma^i_{3k}(r) ({\partial}_j w + \Gamma^3_{j\ell }(r) v^\ell  ) \\
& - \Gamma^m_{jk}(r)({\partial}_m v^i + \Gamma^i_{m \ell }(r) v^\ell)   -  \Gamma^m_{jk}(r) \Gamma^i_{m3}(r) w
               -  \Gamma^3_{jk}(r)({\partial}_3 v^i + \Gamma^i_{3\ell}(r) v^\ell )
\end{aligned}
\right] \boldsymbol{\tau}_i\,  .
\end{aligned}
\end{equation}
Setting $r=0$  and using the convention $   g^{jk} (0)=g^{jk},$ $\Gamma^i_{j \ell}(0)=  \Gamma^i_{j \ell}, {\rm etc.,}$  we have
$$
\Delta_\Sigma v
=g^{jk}[ {\partial}_j {\partial}_k v^i + {\partial}_k(\Gamma^i_{j\ell} v^\ell)
+\Gamma^i_{mk}({\partial}_j v^m +\Gamma^m_{j\ell}  v^\ell)
- \Gamma^m_{jk}({\partial}_m v^i + \Gamma^i_{m \ell } v^\ell) ]\tau_i
$$
for the connection Laplacian $\Delta_\Sigma$ on $\Sigma$,
and we obtain with  \eqref{Chris-1}  the following approximation of \eqref{tangential} for the terms involving~$v$:
\begin{equation}
\label{tangential-v-approx}
\begin{aligned}
(\tilde \Delta_\sM v)_{\rm tan}
&  \approx \Delta_\Sigma v
+ g^{jk}(r)
 [\Gamma^i_{3k}(r) \Gamma^3_{j\ell }(r) v^\ell  - \Gamma^3_{jk}(r)(\partial_3 v^i + \Gamma^i_{3\ell}(r) v^\ell )]\;\boldsymbol{\tau }_i  \ \Big|_{r=0} \\
& =  \Delta_\Sigma v  + g^{jk} [ - l_{j k}  \partial_3 v^i   + ( - l^i_k l_{j\ell}  + l^i_\ell l_{jk})v^\ell  ]\; \tau_i  \\
& =  \Delta_\Sigma v + [  - l^j_j \partial_3 v^i  +(l^j_j l^i_\ell  -  l^i_k l^k_\ell)v^\ell] \; \tau_i   \\
& =   \Delta_\Sigma v  - H_\Sigma \partial_r v  +  {\rm Ric}_\Sigma \,v .
\end{aligned}
\end{equation}
For the terms  of $(\tilde \Delta_\sM u)_{\rm tan}$ involving derivatives of $w$ we derive from \eqref{tangential} the following approximation
\begin{equation}
\label{tangential-w-tangential}
g^{jk}(r) [\Gamma^i_{j3}(r) {\partial}_k w + \Gamma^i_{3k}(r){\partial }_j w] \; \boldsymbol{\tau}_i \Big|_{r=0}
= - 2 g^{jk} l^i_j \partial_k w\,\tau_i = - 2L_\Sigma\,\gd_\Sigma w.
\end{equation}
For the remaining terms in \eqref{tangential} containing $w$, we observe that
\begin{equation*}
\begin{aligned}
&  [\partial_k\Gamma^i_{j3}(r) + \Gamma^m_{j3}(r)\Gamma^i_{mk}(r)- \Gamma^m_{jk}\Gamma^i_{m3}(r) ]\Big|_{r=0}
=- [\partial_k l^i_j + \Gamma^i_{mk} l^m_j - \Gamma^m_{jk} l^i_m ]= -(l^i_j)_{|\,k} .\\
\end{aligned}
\end{equation*}
We can then conclude that
\begin{equation}
\label{tangential-w}
\begin{aligned}
w\; g^{jk}(r)[\partial_k\Gamma^i_{j3}(r) + \Gamma^m_{j3}(r)\Gamma^i_{mk}(r)- \Gamma^m_{jk}\Gamma^i_{m3}(r) ]\boldsymbol{\tau}_i \Big|_{r=0}
&=  - w\,g^{jk} (l^i_j)_{|\,k} \; \tau_i \\
&= -w\,(g^{jk} l^i_j)_{|\,k} \; \tau_i \\
&= -w\,(l^{ik})_{|\,k}  \; \tau_i \\
&= -w\, {\rm div} _\Sigma L^\sharp_\Sigma.
\end{aligned}
\end{equation}
There are additional tangential terms with $i \in \{1,2\}$ in \eqref{Bochner-latin},  resulting from normal derivatives for $j=k=3$.
Noting that $g^{m3}(r)= \delta^{m}_3$, $\Gamma^m_{33}(r)=\Gamma^3_{m3}(r)=0$ for $m\in \{1,2,3\}$
we obtain with \eqref{u-basis} the following expression
\begin{equation}
\label{normal-u}
\begin{aligned}
((\Delta_\sM)_r u)_{\rm tan}
&=[\partial^2_3  v^i  +  (\partial_3 \Gamma^i_{3\ell}(r))v^\ell  + \Gamma^i_{3\ell}(r)\partial_3 v^\ell
+\Gamma^i_{m3}(r)(\partial_3 v^m  + \Gamma^m_{3\ell}(r)v^\ell)  ]\; \boldsymbol{\tau}_i \\
&\approx
\left[\partial^2_3 v^i   +  (\partial_3 \Gamma^i_{3\ell}(r))v^\ell  + \Gamma^i_{3\ell}(r)\partial_3 v^\ell
+\Gamma^i_{m3}(r)(\partial_3 v^m  + \Gamma^m_{3\ell}(r)v^\ell)  \right]\; \boldsymbol{\tau}_i \Big|_{r=0}  \\
& = [\partial^2_r v^i  -  l^i_m l^m_\ell  v^\ell  - l^i_\ell \partial_r v^\ell -l^i_\ell \partial_r v^\ell + l^i_m l^m_\ell v^\ell] \;\tau_i \\
&=  [\partial^2_r v^i  - 2 l^i_\ell \partial_r v^\ell  ]\tau_i = \partial^2_r v  - 2 L_\Sigma \partial_r v.
\end{aligned}
\end{equation}
Here we used   the relation
\begin{equation*}
\begin{aligned}
\partial_r\big|_{r=0}\; \Gamma^i_{3\ell}(r)
& = \partial_r\big|_{r=0}\; [ g^{in}(r)(-l_{\ell n} + r l_{n k} l^k_\ell)] \\
& = -\left( \partial_r\big|_{r=0}\; g^{i n}(r) \right) l_{\ell n}  +  g^{i n}(0) l_{n k} l^k_\ell  \\
& = g^{ik}(0)\left(\partial_r\Big|_{r=0} g_{km}(r)\right) g^{m n}(0) l_{\ell n} + l^i_k l^k_\ell  \\
& =  g^{ik}(0)\left(-2 l_{km}\right) g^{m n}(0) l_{\ell n} + l^i_k l^k_\ell
 = -2 l^i_m l^m_\ell + l^i_k l^k_\ell  = -l^i_m l^m_\ell .
\end{aligned}
\end{equation*}
Combining \eqref{tangential}, \eqref{tangential-v-approx}, \eqref{tangential-w-tangential}, \eqref{tangential-w}, and \eqref{normal-u} yields
\begin{equation}
\begin{split}
\label{u-tangential-combined}
(\Delta_\sM u)_{\rm tan}
& \approx (\Delta_\Sigma  + \partial^2_r + {\rm Ric}_\Sigma )v   - ( 2L_\Sigma  + H_\Sigma)\partial_r v
-2 L_\Sigma \, {\rm grad}_\Sigma w  - w\, {\rm div}_\Sigma L^\sharp_\Sigma\\
&= (\Delta_\sN  + {\rm Ric}_\Sigma )v   - ( 2L_\Sigma  + H_\Sigma)\partial_r  v- 2 L_\Sigma \, {\rm grad}_\Sigma w  - w\, {\rm div}_\Sigma L^\sharp_\Sigma,
\end{split}
\end{equation}
where $\Delta_\sN$ is the connection Laplacian on $(\sN,g_\sN)$.

\bigskip\noindent
We will now introduce  some further approximations. In case $\Sigma$ is $C^3$-close to a sphere of  radius $a$,
it can be shown that
\begin{equation}
\label{approx-geometry}
\|  L_\Sigma \|_{C(\Sigma)}\approx \frac{1}{a},\quad  \| H_\Sigma \|_{C(\Sigma)} \approx \frac{1}{a}, \quad  \| {\rm div}_\Sigma L^\sharp_\Sigma \|_{C(\Sigma)} \approx \frac{1}{a}.
\end{equation}
\\
To be more precise,
let $S_a\subset \bR^3$ be the sphere of radius $a$. Given $\rho\in C(S_a)$, we consider hypersurfaces $\Sigma$ given by
$$
\Sigma:=\Sigma_\rho:=\{\,p+\rho(p)\nu_{S_a}(p):p\in S_a\,\},
$$
where $\nu_{S_a}$ is the outward pointing unit normal on $S_a$.
We say that $\Sigma$ is $C^3$-close to $S_a$ if $\rho\in C^3(S_a)$ and $\| \rho \|_{C(\Sigma)}\le \varepsilon$, for $\varepsilon>0$ to be determined.
Choosing $\varepsilon \approx \frac{1}{a}$, it can be shown that \eqref{approx-geometry} holds. More details will be given somewhere else.

\medskip
For the boundary conditions, we assume that
 $$w(\cdot, r)=0 \text{ for } r\in \{-h,0\}, \quad v(\cdot, -h)=0, \quad \partial_r v(\cdot, 0)=0.$$
 As $h$ is small,   the quantities
$w$ and $\partial_r v$
are much smaller in magnitute than the tangential velocity $v$.
Combining with \eqref{approx-geometry}, we argue that the terms
\begin{equation}
\label{terms-neglected}
2 w L_\Sigma v,\quad w H_\Sigma, \quad (2L_\Sigma + H_\Sigma)\partial_r v,\quad  w\, {\rm div}_\Sigma L^\sharp_\Sigma \quad\text{can be neglected}
\end{equation}
in \eqref{transport-M}, \eqref{div-M}, and \eqref{u-tangential-combined}.
As a consequence of these approximations, the incompressible condition $\dv_{\sM}=0 $ now reads
$\partial_r w(\cdot, r) + \dv_\Sigma v(\cdot, r)=0$. Together with the condition $w(\cdot, 0)=0$, we obtain
\begin{equation}
\label{w-div}
w(\cdot, r) =\int_r^0 \dv_\Sigma v(\cdot, \xi)\, d\xi.
\end{equation}
This shows that the term $\gd_\Sigma w$ is, in fact, of second order in $v$.
However, with~\eqref{lem: N-invertible} one verifies that
$$
\| \gd_\Sigma w\|_{L_q(\sN)} \le C h \| \Delta_\sN v\|_{L_q(\sN)}.
$$
Hence with \eqref{approx-geometry}, the term
\begin{equation}
\label{grad-neglected}
2 L_\Sigma \,\gd_\Sigma w\quad\text{can be neglected}
\end{equation}
in \eqref{u-tangential-combined},  as compared to $\Delta_\sN v$.
We also observe on the go that by~\eqref{w-div} and $w(\cdot, -h)=0$,
\begin{equation}
\label{div-bar-v}
\dv_\Sigma\, \bar v=0.
\end{equation}
We now assume that the external force $f$ in~\eqref {N-S eq} is given by 
$$f=-\varrho \gamma\, (\nu_\Sigma\circ \Pi_\Sigma),$$ 
where $\gamma $ is the gravitational acceleration and $\Pi_\Sigma$ is the projection of $S$ onto $\Sigma$,
see~\eqref{P-d}.
As is common in the modeling of geophysical fluid dynamics,  we then replace the vertical momentum equation  by the \emph{hydrostatic approximation}:
\begin{equation}
\label{hydrostatic equation}
 \partial_r \pi = -\rho g.
\end{equation}
Let $\pi_s(\cdot)=\pi(\cdot,0)$ be the pressure on the ocean surface. Then \eqref{hydrostatic equation} implies
$$
\pi(\cdot,r)= \pi_s(\cdot) -r \rho g.
$$
Taking $\gd_\Sigma$ on both sides yields
\begin{equation}
\label{grad-pis}
\gd_\Sigma \pi(\cdot,r)=  \gd_\Sigma \pi_s(\cdot).
\end{equation}
Upon setting $\varrho=\mu_s=1$,
and employing \eqref{transport-M}, \eqref{u-tangential-combined}, \eqref{terms-neglected}, \eqref{grad-neglected}, \eqref{div-bar-v} and \eqref{grad-pis}
we obtain the primitive equations \eqref{PE}.

Let us note that high-resolution models for geophysical flows, including, for example, complex terrain and thermal updrafts, require non-hydrostatic dynamics, in which the hydrostatic approximation is replaced by the vertical momentum equation in Euclidean coordinates of the form
$$
\partial_t w = - \frac{1}{\varrho} p_r - g +F_r ,
$$
where $F_r$ describes friction or turbulent forces.

\medskip\noindent
\begin{remark}
\label{comparison-LTW}
We compare our derivation with the one presented in \cite{LTW92} in the
special case  where $\Sigma =S_a$.

\smallskip
While we consider from the very beginning the  Navier-Stokes equations for an incompressible fluid,
the authors in \cite{LTW92}
start from the compressible Navier–Stokes equations including a Coriolis force term and temperature.
Exploiting the small aspect ratio of the atmosphere relative to the Earth's radius, they invoke the hydrostatic approximation, replacing the vertical momentum equation by the
relation
$$\partial_r p = -\varrho g.$$
The authors introduce pressure as a vertical coordinate (the so-called pressure-coordinate formulation). Integrating the
hydrostatic relation then allows the density and the vertical velocity to be eliminated in favor of the horizontal velocity and temperature.
The continuity equation is transformed into a divergence constraint involving the horizontal velocity and the vertical velocity in pressure coordinates.
This leads to a reformulation of the atmospheric equations as a system of evolution equations for the horizontal velocity and a thermodynamic variable.

Let us remark, however, that the reformulation of the primitive equations in pressure coordinates is rigorous only under the (non-physical) assumption that the pressure is constant on the boundary.

 \medskip\noindent
 We now compare our system \eqref{PE} with the equations derived in \cite{LTW92} for the  case \(\Sigma=S_a\).

 \medskip\noindent
 Using the same spherical coordinates $(\theta,\varphi)$ as in \cite{LTW92}, where $\theta\in(0,\pi)$ denotes the colatitude and $\varphi\in(0,2\pi)$ the longitude, and writing the velocity field in the form
$$
u = v_\theta e_\theta + v_\varphi e_\varphi + v_r e_r,
\qquad
e_\theta = \frac{1}{a}\,\frac{\partial}{\partial\theta},
\qquad
e_\varphi = \frac{1}{a\sin\theta}\,\frac{\partial}{\partial\varphi},
\qquad
e_r = \frac{\partial}{\partial r},
$$
it follows by the same arguments as above that the tangential component of the connection Laplacian satisfies
\begin{equation*}
\begin{aligned}
   (\Delta u)_{\rm tan}=
 &\Big[  \Delta^B_{S_a} v_\theta +  \frac{1}{r^2} \frac{\partial}{\partial r} \big( r^2 \frac{\partial v_\theta }{\partial r}\big)
 - \frac{2\cos\theta}{r^2\sin^2\theta}\frac{ \partial v_\varphi}{\partial \varphi} -\frac{v_\theta}{r^2 \sin^2\theta} + \frac{2}{r^2} \frac{\partial v_r}{\partial\theta} \Big] e_\theta
  \\
+ &\Big[  \Delta^B_{S_a} v_\theta +  \frac{1}{r^2} \frac{\partial}{\partial r} \big( r^2 \frac{\partial v_\theta }{\partial r}\big)
 + \frac{2\cos\theta}{r^2\sin^2\theta}\frac{ \partial v_\theta}{\partial \varphi} -\frac{v_\varphi}{r^2 \sin^2\theta} + \frac{2}{r^2\sin\theta} \frac{\partial v_r}{\partial\varphi} \Big]  e_\varphi.
 \end{aligned}
\end{equation*}
This formula is well-known in the literature. Here,
$\Delta^B_{S_a}$ denotes the Laplace-Beltrami operator acting on scalar functions.
Our approximation of \((\Delta u)_{\rm tan}\) corresponds here to setting \(r=a\), as is also done in \cite{LTW92}; see p.~244. This leads to
\begin{equation}
\label{3D-Bochner-approx}
\begin{aligned}
   (\Delta u)_{\rm tan}\approx
 &\Big[  \Delta^B_{S_a} v_\theta +  \frac{1}{r^2} \frac{\partial}{\partial r} \big( r^2 \frac{\partial v_\theta }{\partial r}\big)\Big|_{r=a}
 - \frac{2\cos\theta}{a^2\sin^2\theta}\frac{ \partial v_\varphi}{\partial \varphi} -\frac{v_\theta}{a^2 \sin^2\theta} + \frac{2}{a^2} \frac{\partial v_r}{\partial\theta} \Big] e_\theta
  \\
+ &\Big[  \Delta^B_{S_a} v_\theta +  \frac{1}{r^2} \frac{\partial}{\partial r} \big( r^2 \frac{\partial v_\theta }{\partial r}\big)\Big|_{r=a}
 + \frac{2\cos\theta}{a^2\sin^2\theta}\frac{ \partial v_\theta}{\partial \varphi} -\frac{v_\varphi}{a^2 \sin^2\theta} + \frac{2}{a^2\sin\theta} \frac{\partial v_r}{\partial\varphi} \Big]  e_\varphi.
 \end{aligned}
\end{equation}
A comparison with equation~(1.35) in \cite{LTW92} shows that the authors omit the term
$$
\frac{2}{a}
\Big(\frac{1}{a}\frac{\partial v_r}{\partial\theta}\,e_\theta
+\frac{1}{a\sin\theta}\frac{\partial v_r}{\partial\varphi}\,e_\varphi\Big)= \frac{2}{a}\,\gd_{S_a} v_r= -2L_{S_a}\gd_{S_a}v_r
$$
which appears in \eqref{3D-Bochner-approx}. This is precisely the same approximation that we made in \eqref{grad-neglected}.
Owing to the use of pressure coordinates in the vertical direction, the term
\[
\frac{1}{r^2}\frac{\partial}{\partial r}
\Bigl(r^2\frac{\partial v}{\partial r}\Bigr)\Big|_{r=a}
\]
 takes a slightly different form in \cite{LTW92}; see formula~(1.33) there.
In our formulation, by neglecting the term  $ (2L_\Sigma + H_\Sigma)\partial_r v$ in \eqref{u-tangential-combined},  the expression
\begin{equation*}
\begin{aligned}
& \Big[  \Delta^B_{S_a} v_\theta +  \frac{1}{r^2} \frac{\partial}{\partial r} \big( r^2 \frac{\partial v_\theta }{\partial r}\big)\Big|_{r=a}
 - \frac{2\cos\theta}{a^2\sin^2\theta}\frac{ \partial v_\varphi}{\partial \varphi} -\frac{v_\theta}{a^2 \sin^2\theta} \Big] e_\theta \\
+&\Big[  \Delta^B_{S_a} v_\theta +  \frac{1}{r^2} \frac{\partial}{\partial r} \big( r^2 \frac{\partial v_\theta }{\partial r}\big)\Big|_{r=a}
 + \frac{2\cos\theta}{r^2\sin^2\theta}\frac{ \partial v_\theta}{\partial \varphi} -\frac{v_\varphi}{r^2 \sin^2\theta}  \Big]  e_\varphi,
 \end{aligned}
\end{equation*}
takes the form
$$
(\Delta_{\sN}+\Ric)\,v,
$$
where the operator $\Delta_{\sN}+\Ric$ admits a natural geometric interpretation independent of the chosen coordinate system,
with $\sN = S_a \times (-h,0)$.\\
As noted on page 244  in  \cite{LTW92}, the authors omit the term
$$
v_r\,\frac{1}{a}\bigl(v_\theta e_\theta + v_\varphi e_\varphi\bigr)
=
-v_r L_{S_a}\bigl(v_\theta e_\theta + v_\varphi e_\varphi\bigr)
$$
that is originally contained in formula (1.5).
This corresponds to the term $-2wL_\Sigma v$  in \eqref{transport-M}, which we have likewise omitted.
We note in passing that our calculations suggest that the term
$
\frac{v_r}{a}\bigl(v_\theta e_\theta + v_\varphi e_\varphi\bigr)
$
appearing in \cite[(1.5)]{LTW92} should carry an additional factor of 2.
Finally, the term $w \operatorname{div}_\Sigma L_\Sigma^\sharp$ that was neglected in our approximation vanishes identically when $\Sigma=S_a$.

We emphasize that temperature effects and the Coriolis force have not been taken into account in the present manuscript.
 \end{remark}

\bigskip\noindent
{\bf Conflict of interest:} The authors assert that there is no conflict of interest to declare.

\medskip\noindent
{\bf Data availability statement:}
Data sharing is not applicable as no datasets were generated or analyzed for the manuscript.

\bigskip


\begin{thebibliography}{1}


\bibitem{Ada03} R. A. Adams, J.J.F. Fournier. \emph{Sobolev Spaces}. 2nd ed. Academic Press, 2003.

\bibitem{Ama95}
H. Amann,
\emph{Linear and Quasilinear Parabolic Problems. Vol. I. Abstract Linear Theory.}
Monogr. Math. \textbf{89}, Birkh\"auser Boston, Inc., Boston, MA, 1995.

\bibitem{Ama13}
H. Amann,
\emph{Function spaces on singular manifolds.}
Math. Nachr. \textbf{286}, no. 5-6, 436-475 (2013).

\bibitem{Ama16}
H. Amann,
\emph{Parabolic equations on uniformly regular Riemannian manifolds and degenerate initial boundary value problems.}
Recent developments of mathematical fluid mechanics, 43-77.
Adv. Math. Fluid Mech.
Birkh\"auser/Springer, Basel, 2016.

%

\bibitem{AmannBook19}
H. Amann,
\emph{Linear and Quasilinear Parabolic Problems. Vol. II. Function Spaces.}
Monogr. Math. \textbf{106},  Birkh\"auser/Springer, Cham, 2019.

\bibitem{AubinBook}
T. Aubin,
\emph{Nonlinear Analysis on Manifolds. Monge-Amp\'ere Equations.}
Grundlehren Math. Wiss. \textbf{252} [Fundamental Principles of Mathematical Sciences]
Springer-Verlag, New York, 1982.


\bibitem{CaTi07}
C. Cao,  E.~S. Titi,
\emph{Global well-posedness of the three-dimensional viscous primitive equations of large scale ocean and atmosphere dynamics.}
Ann. of Math. (2) \textbf{166}, no. 1, 245--267 (2007).


\bibitem{DHP03}
R. Denk, M. Hieber and J. Pr\"uss,
\emph{$\mathscr{R}$-boundedness, Fourier multipliers and problems of elliptic and parabolic type.}
Mem. Amer. Math. Soc. \textbf{166}, no. 788, 2003.

\bibitem{Dru09}
A. Drutsa,
\emph{Existence in the large of a solution to primitive equations in a domain with unveven bottom.}
Russian J. Num. Anal. Math. Model. {\bf 24}, no. 6, p515 (2009).

\bibitem{Dru11}
A. Drutsa,
\emph{Existence in the large of a solution of equations of large-scale ocean dynamics on a manifold.}
Sbornik Math. {\bf 202}, no. 10, 1463 (2011).

\bibitem{DSS24}
H. Du, Y. Shao and G. Simonett,
\emph{On a thermodynamically consistent model for magnetoviscoelastic fluids in 3D.}
J. Evol. Equ. {\bf 24}, no. 1, Paper No. 9, 51 pp (2024).

\bibitem{FGHHKW20}
K. Furukawa, Y. Giga, M. Gries, M. Hieber, A. Hussein and T. Kashiwabara.
\emph{Rigorous justification of the hydrostatic approximation for the primitive equations by scaled Navier-Stokes equations.}
Nonlinearity. {\bf 33}, no. 12,  (2020).

\bibitem{FGHHKW25}
K. Furukawa, Y. Giga, M. Gries, M. Hieber, A. Hussein and T. Kashiwabara.
\emph{The three limits of the hydrostatic  approximation.}
J. London Math. Soc. {\bf 111}, no. 2, e70130, (2025).



\bibitem{GHK17}
G.P. Galdi, M. Hieber and T. Kashiwabara,
\emph{Strong time-periodic solutions to the 3D primitive equations subject to arbitrary large forces.}
Nonlinearity {\bf 30}, no. 10, 3979-3992 (2017).

\bibitem{GGHHK17}
Y. Giga, M. Gries, M. Hieber, A. Hussein and T. Kashiwabara.
\emph{Bounded $H^\infty$-calculus for the hydrostatic Stokes operator on $L_p$-spaces and applications.}
Proc. Amer. Math. Soc. \textbf{145} (9), 3865-3876 (2017).

\bibitem{GGHHK20}
Y. Giga, M. Gries, M. Hieber, A. Hussein and T. Kashiwabara.
\emph{Analyticity of solutions to the primitive equations.}
Math. Nachr. {\bf 293}, no. 2, 284-304 (2020).

\bibitem{GGHHK20b}
Y. Giga, M. Gries, M. Hieber, A. Hussein and T. Kashiwabara.
\emph{The hydrostatic Stokes semigroup and well-posedness of the primitive equations on spaces of bounded functions.}
J. Func. Anal. {\bf 279}, no. 3, 108561 (2020).

\bibitem{GGHHK21}
Y. Giga, M. Gries, M. Hieber, A. Hussein and T. Kashiwabara.
\emph{The primitive equations in the scaling invariant space $L^\infty(L^1)$.}
J. Evol. Equ. {\bf 21}, no. 4, 4145-4169 (2021).

\bibitem{HHK16}
M. Hieber, A. Hussein and T. Kashiwabara,
\emph{Global strong $L^p$ well-posedness of the 3D primitive equations with heat and salinity diffusion.}
J. Differential Equations {\bf 261}, no. 12, 6950-6981 (2016).

\bibitem{HiKa16}
M. Hieber, T. Kashiwabara,
\emph{Global strong well-posedness of the three dimensional primitive equations in $L^p$-spaces.}
Arch. Ration. Mech. Anal. {\bf 221}, no. 3, 1077-1115 (2016).

\bibitem{HvNVW17}
T. Hyt\"{o}nen, J. v. Neerven, M. Veraar and  L. Weis,
\emph{Analysis in Banach Spaces. Vol. II. Probabilistic Methods and Operator Theory.}
 Series of Modern Surveys in Mathematics [Results in
              Mathematics and Related Areas. 3rd Series. A Series of Modern
              Surveys in Mathematics] \textbf{67},  Springer, Cham, 2017.

\bibitem{Kob07}
G.~M. Kobelkov,
\emph{Existence of a solution ``in the large'' for ocean dynamics equations.}
J. Math. Fluid Mech. \textbf{9}, 588--610 (2007).

\bibitem{Kor26}
P. Korn,
\emph{Global well-posedness of the viscous primitive equations with a free surface.}
Personal communication (2026).

\bibitem{KuZI07}
I. Kukavica, M. Ziane,
\emph{On the regularity of the primitive equations of the ocean.}
Nonlinearity \textbf{20}, 2739--2753 (2007).

\bibitem{LiTi18}
J. Li, E.~S. Titi,
\emph{Recent advances concerning certain class of geophysical flows.}
In: Handbook of Math. Anal. Mech. Viscous Fluids (eds.: Y. Giga, A. Novotny) Vol. 1, 933-971, Springer (2018).

\bibitem{LiTi19}
J. Li, E.~S. Titi,
\emph{The primitive equations as the small aspect ratio limit of the Navier--Stokes equations: rigorous justification of the hydrostatic approximation.}
J. Math. Pures Appl. \textbf{124}, 30--58  (2019).

\bibitem{LTW92}
J.-L. Lions, R. Temam and S. Wang,
\emph{New formulations of the primitive equations of atmosphere and applications.}
Nonlinearity \textbf{5}, no. 2, 237--288 (1992).

\bibitem{LTW92b}
J.-L. Lions, R. Temam and  S. Wang,
\emph{On the equations of the large-scale ocean.}
Nonlinearity {\bf 5}, no. 5, 1007-1053 (1992).

\bibitem{LTW95}
J.-L. Lions, R. Temam and S. Wang,
\emph{Mathematical theory for the coupled atmosphere--ocean models (CAO~III)},
J. Math. Pures Appl. \textbf{74}, 05--163 (1995).



\bibitem{MeSc12}
M. Meyries, R. Schnaubelt,
\emph{Interpolation, embeddings and traces of anisotropic ractional Sobolev spaces with temporal weights}.
 J. Funct. Anal. {\bf 262}, 1200-1229 (2012).


\bibitem{MeVe12}
M. Meyries, M. Veraar,
\emph{Sharp embedding results for spaces of smooth functions with power weights}.
Studia Math. {\bf 208}, 257-293 (2012).

\bibitem{Mul06}
P. M\"uller,
\emph{The Equations of Oceanic Motions}.
Cambridge Univ. Press, 291pp, (2006).

\bibitem{Ped87}
J. Pedlovsky,
\emph{Geophysical Fluid Dynamics.}
Second edition. Springer, New York, 1987.

\bibitem{Pet06}
P. Petersen,
\emph{Riemannian Geometry.}
Second edition. Graduate Texts in Mathematics \textbf{171}, Springer, New York, 2006.

\bibitem{PrSi16}
J. Pr\"uss, G. Simonett,
\emph{Moving Interfaces and Quasilinear Parabolic Evolution Equations.}
Monographs in Mathematics. Birkh\"auser Verlag. 2016.

\bibitem{PrWi17}
J. Pr\"uss,  M. Wilke,
\emph{Addendum to the paper ``On quasilinear parabolic evolution equations in weighted Lp-spaces II"}.
J. Evol. Equ. {\bf 17}, no. 4, 1381-1388 (2017).

\bibitem{PSW18}
J. Pr\"uss, G. Simonett and  M. Wilke,
\emph{Critical spaces for quasilinear parabolic evolution equations and applications}.
J. Differential Equations {\bf 264}, no. 3, 2028-2074 (2018).

\bibitem{PSW21}
J. Pr\"uss, G. Simonett and  M. Wilke,
\emph{On the Navier-Stokes equations on surfaces}.
J. Evol. Equ. {\bf 21}, no. 3, 3153-3179 (2021).


\bibitem{SaTu20}
M. Samavaki, J. Tuomela,
\emph{Navier-Stokes equations on Riemannian manifolds}.
J. Geom. Phys. {\bf 148} 103543, 15 pp. (2020).


\bibitem{SSW25}
Y. Shao, G. Simonett and M. Wilke,
\emph{The Navier-Stokes equations on manifolds with boundary.}
J. Differential Equations {\bf 416}, 1602-1659 (2025).


\bibitem{SiWi22}
G. Simonett,  M. Wilke,
\emph{$H^\infty$-calculus for the surface Stokes operator and applications}.
 J. Math. Fluid Mech. {\bf 24}, no. 4,  Paper No. 109, 23 pp.  (2022).

%
\bibitem{Tri78}
H. Triebel,
\emph{Interpolation Theory, Function Spaces, Differential Operators.}
North-Holland Publishing Co., Amsterdam-New York, 1978.



\end{thebibliography}
\end{document}